\documentclass[a4paper, 11pt,reqno]{amsart}
\usepackage[top = 1in, bottom = 1in, left=1.0in, right=1.0in]{geometry}

\usepackage{amssymb,amsmath,amsthm,graphicx,xcolor,mathtools,enumerate,mathrsfs}
\usepackage{tabularx}
\usepackage[shortlabels]{enumitem} 
\usepackage[british]{babel}
\usepackage[utf8]{inputenc}
\usepackage{thmtools} %Restating theorems

\allowdisplaybreaks

\usepackage{hyperref}
\hypersetup{colorlinks, linkcolor={red!50!black}, citecolor={green!50!black}, urlcolor={blue!50!black}}

\usepackage{caption}
\usepackage{subcaption}
\theoremstyle{plain}
\newtheorem{theorem}{Theorem}[section]
\newtheorem{proposition}[theorem]{Proposition}
\newtheorem{corollary}[theorem]{Corollary}
\newtheorem{lemma}[theorem]{Lemma}
\newtheorem{conjecture}[theorem]{Conjecture}
\newtheorem{claim}[theorem]{Claim}

\theoremstyle{definition}
\newtheorem{definition}[theorem]{Definition}
\newtheorem{construction}[theorem]{Construction}

\numberwithin{equation}{section}

\newcommand{\es}{\emptyset}
\newcommand{\eps}{\varepsilon}
\newcommand{\sm}{\setminus}
\renewcommand{\subset}{\subseteq}

\newcommand{\NATS}{\mathbb{N}}
\newcommand{\REALS}{\mathbb{R}}

\def\le{\leqslant}
\def\leq{\leqslant}
\def\ge{\geqslant}
\def\geq{\geqslant}

\newcommand{\cA}{\mathcal{A}}
\newcommand{\cB}{\mathcal{B}}
\newcommand{\cC}{\mathcal{C}}
\newcommand{\cF}{\mathcal{F}}
\newcommand{\cG}{\mathcal{G}}
\newcommand{\cJ}{\mathcal{J}}
\newcommand{\cH}{\mathcal{H}}
\newcommand{\cP}{\mathcal{P}}
\newcommand{\cQ}{\mathcal{Q}}
\newcommand{\cS}{\mathcal{S}}
\newcommand{\cT}{\mathcal{T}}

\newcommand{\cX}{\mathcal{X}}
\newcommand{\cY}{\mathcal{Y}}

\newcommand{\fS}{\mathfrak{S}}
\newcommand{\sF}{\mathscr{F}}

\newcommand{\expectation}{\mathbf{E}}

\newcommand{\vnorm}[1]{||#1||_1}
\newcommand{\ori}[1]{\smash{\overrightarrow{#1}}}

\newcommand{\bvec}[1]{\mathbf{#1}}
\newcommand{\hc}{\mathrm{hc}}

\newcommand{\ext}{\mathrm{ext}}

\newcommand{\reldeg}{\overline{\deg}} 
\newcommand{\Hy}{{\mathcal H}}
\newcommand{\Part}{{\mathcal P}}
\newcommand{\Qb}{{\bf Q}}
\newcommand{\reld}{d^*}

\newenvironment{proofclaim}[1][Proof of the claim]{\begin{proof}[#1]}{\end{proof}}

\usepackage{subcaption}
\usepackage{tikz}
\usetikzlibrary{calc}
\usepackage{xcolor}
\usetikzlibrary{matrix}

\usetikzlibrary{decorations.pathmorphing}
\usetikzlibrary{backgrounds}
\usetikzlibrary{shapes}
\tikzset{snake it/.style={decorate, decoration=snake}}

\definecolor{DarkDesaturatedBlue}{HTML}{3A3556}
\definecolor{VividOrange}{HTML}{F15918}
\definecolor{PureOrange}{HTML}{FFBA00}
\definecolor{LightGrayishPink}{HTML}{EEC5D5}
\definecolor{VerySoftBlue}{HTML}{B5AFDB}

\newcommand{\triple}[7]{
	\ifx\relax#4\relax
	\def\qoffs{0pt}
	\else
	\def\qoffs{#4}
	\fi
	\def\qhedge{
		($#1+#3!\qoffs!-90:#2-#3$) --
		($#2+#1!\qoffs!-90:#3-#1$) --
		($#3+#2!\qoffs!-90:#1-#2$) -- cycle}
	
	\coordinate (12) at ($#1!\qoffs!90:#2$);
	\coordinate (13) at ($#1!\qoffs!-90:#3$);
	\coordinate (23) at ($#2!\qoffs!90:#3$);
	\coordinate (21) at ($#2!\qoffs!-90:#1$);
	\coordinate (31) at ($#3!\qoffs!90:#1$);
	\coordinate (32) at ($#3!\qoffs!-90:#2$);
	
	\def\nqhedge{
		(13) let \p1=($(13)-#1$), \p2=($(12)-#1$) in
		arc[start angle={atan2(\y1,\x1)}, delta angle={atan2(\y2,\x2)-atan2(\y1,\x1)-360*(atan2(\y2,\x2)-atan2(\y1,\x1)>0)}, x radius=\qoffs, y radius=\qoffs] --
		(21) let \p1=($(21)-#2$), \p2=($(23)-#2$) in
		arc[start angle={atan2(\y1,\x1)}, delta angle={atan2(\y2,\x2)-atan2(\y1,\x1)-360*(atan2(\y2,\x2)-atan2(\y1,\x1)>0)}, x radius=\qoffs, y radius=\qoffs] --
		(32) let \p1=($(32)-#3$), \p2=($(31)-#3$) in
		arc[start angle={atan2(\y1,\x1)}, delta angle={atan2(\y2,\x2)-atan2(\y1,\x1)-360*(atan2(\y2,\x2)-atan2(\y1,\x1)>0)}, x radius=\qoffs, y radius=\qoffs] --
		cycle}
	
	\ifx\relax#5\relax
	\def\qlwidth{1pt}
	\else
	\def\qlwidth{#5}
	\fi
	
	\ifx\relax#7\relax
	\fill \nqhedge;
	\else
	\fill[#7]\nqhedge;
	\fi
	
	\ifx\relax#6\relax
	\draw[line width=\qlwidth,rounded corners=\qoffs]\nqhedge;
	\else
	\draw[line width=\qlwidth,#6]\nqhedge;
	\fi
}

\usepackage{comment}
\includecomment{comment}

\title[]{Minimum degree conditions\\for tight Hamilton cycles}
\date{\today}

\author[R.~Lang]{Richard Lang}
\address[R.~Lang]{University of Heidelberg,
	Institute for Computer Science,
	Im Neuenheimer Feld 205,
	69120 Heidelberg, Germany}
\email{lang@informatik.uni-heidelberg.de}

\author[N.~Sanhueza-Matamala]{Nicolás Sanhueza-Matamala}
\address[N.~Sanhueza-Matamala]{The Czech Academy of Sciences, Institute of Computer Science, Pod Vod\'{a}renskou v\v{e}\v{z}\'{\i} 2, 182 07 Prague, Czechia}
\email{nicolas@sanhueza.net}

\thanks{The research leading to these results was supported by the Czech Science Foundation, grant number GA19-08740S with institutional support RVO: 67985807 (N.~Sanhueza-Matamala) and also partially by the Deutsche Forschungsgemeinschaft (DFG, German Research Foundation) -- 42821240 (R. Lang)}
 
\begin{document}

\begin{abstract}
	We develop a new framework to study minimum $d$-degree conditions in $k$-uniform hypergraphs, which guarantee the existence of a tight Hamilton cycle.
	Our main theoretical result deals with the typical absorption, path cover and connecting arguments for all $k$ and $d$ at once, and thus sheds light on the underlying structural problems.
	Building on this, we show that one can study minimum $d$-degree conditions of $k$-uniform tight Hamilton cycles by focusing on the inner structure of the neighbourhoods.
	This reduces the matter to an Erdős--Gallai-type question for $(k-d)$-uniform hypergraphs, which is of independent interest.
	
	Once this framework is established, we can easily derive two new bounds.
	Firstly, we extend a classic result of R\"odl, Ruciński and Szemerédi for $d=k-1$ 
	by determining asymptotically best possible degree conditions for $d = k-2$ and all $k \ge 3$.
	This was proved independently by Polcyn, Reiher, Rödl and Schülke.
	Secondly, we provide a general upper bound of $1-1/(2(k-d))$ for the tight Hamilton cycle $d$-degree threshold in $k$-uniform hypergraphs, thus narrowing the gap to the lower bound of $1-1/\sqrt{k-d}$ due to Han and Zhao.
\end{abstract}

\maketitle
\thispagestyle{empty}
\vspace{-0.4cm}

\section{Introduction}
A widely researched question in modern graph theory is whether a given graph contains certain vertex-spanning substructures such as a perfect matching or a Hamilton cycle.
Since the corresponding decision problems are usually computationally intractable, we do not expect to find a `simple' characterisation of the (hyper)graphs that contain a particular spanning structure.
The extremal approach to these questions has therefore focused on easily verifiable sufficient conditions.
A classic example of such a result is Dirac's theorem~\cite{Dir52}, which states that a graph, whose minimum degree is at least as large as half the size of the vertex set, contains a Hamilton cycle.
Since its inception, Dirac's theorem has been generalised in numerous ways~\cite{Gou14,KO14}.
Here we study its analogues for hypergraphs.

To formulate Dirac-type problems in hypergraphs, let us introduce the corresponding degree conditions.
In this paper, we consider \emph{$k$-uniform} hypergraphs (or shorter \emph{$k$-graphs}), in which every edge consists of exactly $k$ vertices.
The \emph{degree}, written $\deg(S)$, of a subset of vertices $S$ in a $k$-graph $H$ is the number of edges which contain $S$.
It is often convenient to state results and problems in terms of the \emph{relative degree} $\reldeg(S) = \deg(S)/\binom{n-d}{k-d}$, where $n$ is the number of vertices of~$H$ and $d$ is the size of $S$.
The \emph{minimum relative $d$-degree} of a $k$-graph $H$, written $\overline{\delta}_d(H)$, is the minimum of $\reldeg(S)$ over all sets $S$ of $d$ vertices.
The case of $d=k-1$ is also known as the \emph{minimum codegree}.
Observe that minimum degrees exhibit a monotone behaviour: 
\begin{equation}\label{equ:degree-types-monotonicity}
\overline{\delta}_{k-1}(H) \leq \dotsb \leq \overline{\delta}_{1}(H).
\end{equation}
It is therefore not surprising that minimum degree conditions for the existence of Hamilton cycles were first studied for the structurally richer settings when $d$ is close to $k$.
Before we come to this, let us specify the notion of cycles that we are interested in.

Hamilton cycle in hypergraphs have been considered in several ways. 
The oldest variant are \emph{Berge cycles}~\cite{Ber72}, whose
minimum degree conditions were studied by Bermond,  Germa, Heydemann and Sotteau~\cite{BGHS78}.
In the last two decades, however, research has increasingly focused on a stricter notion of Hamilton cycles~\cite{ABCM17,GPW12,HLS17,PRR+20,PRRS21,RRR19,RRS06,RRS09a,RRS11,Sch19} introduced by~Kierstead and Katona~\cite{KK99}.
A \emph{tight cycle} in $k$-graph is a cyclically ordered set of at least $k+1$ vertices such that every interval of $k$ subsequent vertices forms an edge.

Asymptotic Dirac-type results for hypergraphs can be stated compactly using the notion of thresholds.
\begin{restatable}[Tight Hamilton cycle threshold]{definition}{defcyclethreshold}\label{def:cycle-threshold}
	For  $1 \leq d \leq k-1$, the \emph{minimum $d$-degree threshold for tight Hamilton cycles}, denoted by $\hc_d(k)$, is the smallest number $\delta >0$ with the following property:
	
	For every $\mu >0$ there is an $n_0 \in \NATS$ such that every $k$-graph $H$ on $n \geq n_0$ vertices with minimum relative degree $\overline{\delta}_d(H) \geq \delta + \mu$ contains a tight Hamilton  cycle.
\end{restatable}
For instance, we have $\hc_1(2)=1/2$ by Dirac's theorem.
Codegree thresholds for tight Hamilton cycles of larger uniformity were first investigated by Katona and Kierstead~\cite{KK99}, who showed that  $1/2\leq \hc_{k-1}(k) \leq 1-1/(2k)$ and conjectured the threshold to be $\hc_{k-1}(k)=1/2$.
This was confirmed in a seminal contribution by Rödl, Ruciński and Szemerédi~\cite{RRS06,RRS09a}.
For $k=3$ and large enough hypergraphs, the same authors also obtained an exact result~\cite{RRS11}.
For more background on these problems and their history, we refer the reader to the recent surveys of Kühn and Osthus~\cite{KO14}, Rödl and Ruciński~\cite{RR10}, Simonovits and Szemerédi~\cite{SS19} and Zhao~\cite{Zha16}.

Here, we investigate the thresholds $\hc_d(k)$ when $1 \leq d \leq k-2$. 
As noted by Kühn and Osthus~\cite{KO14} and Zhao~\cite{Zha16}, it appears that the problem gets significantly harder in this setting.
After preliminary results of Glebov, Person and Weps~\cite{GPW12}, Rödl and Ruciński~\cite{RR14}, Rödl, Ruciński, Schacht and Szemerédi~\cite{RRSS17} and Cooley and Mycroft~\cite{CM17}, it was shown by Reiher, Rödl, Ruciński, Schacht and Szemerédi~\cite{RRR19} that $\hc_{1}(3) = 5/9$, which resolves the case of $d=k-2$ when $k=3$.
With regards to general bounds, Rödl and Ruciński~\cite{RR10} conjectured that the threshold $\hc_d(k)$ coincides with the analogous threshold of perfect matchings.
However, this conjecture has been disproved.
Han and Zhao~\cite{HZ16} showed (amongst other things) that $\hc_d(k) \geq 1-1/\sqrt{k-d}$ (see Section~\ref{sec:discussion} for more details).
As a consequence, we have $\hc_d(k) \to 1$ when the difference $k-d$ goes to infinity.
This behaviour differs for perfect matchings, where the corresponding threshold is bounded away from $1$ by a universal constant, as discussed by Ferber and Jain~\cite{FJ19}.

The best general upper bound for $h_d(k)$ is due to Glebov, Person and Weps~\cite{GPW12}, who proved that $\hc_d(k) \leq \hc_1(k) \leq 1-1/(Ck^3)^{k-1}$ for a $C > 1$ independent of $d$ and $k$.
Note that this is a function of $k$.
Given the known lower bounds, it is natural to ask whether $\hc_d(k)$ can be bounded by a function of $k-d$ instead. 
Here we answer this question in the affirmative.

\begin{theorem}\label{thm:main-simple-general}
	For all $k \ge 2$ and $1 \leq d \leq k-1$, we have $\hc_{d}(k) {\leq 2^{-1/(k-d)}}.$
\end{theorem}
{To compare this with the above results, note that $2^{-1/(k-d)} \leq 1-1/(2(k-d))$.}

In the particular case of $k \ge 3$ and $d = k-2$, the construction of Han and Zhao shows that $\hc_{k-2}(k) \geq 5/9$.
Very recently, Polcyn, Reiher, Rödl, Ruciński, Schacht and Schülke~\cite{PRR+20} proved that $\hc_{2}(4)=5/9$.
They also conjectured that $\hc_{k-2}(k)=5/9$ for all $k \ge 5$~\cite{RRR+19}.
Here we resolve this problem.

\begin{theorem}\label{thm:main-simple-k-2}
	For all $k \ge 3$, we have $\hc_{k-2}(k) = 5/9$.
\end{theorem}

We remark that Theorem~\ref{thm:main-simple-k-2} was shown independently by Polcyn, Reiher, Ruciński and Schülke~\cite{PRRS21}.

The above presented bounds on the tight Hamilton cycle thresholds follow from a new method that is suitable to approach these issues in a general manner.
The argument has three parts.
First, it is shown that thresholds of tight Hamilton cycles can be studied in a structurally cleaner setting related to what we call Hamilton frameworks.
Next, we reduce the task of finding $k$-uniform Hamilton frameworks under minimum $d$-degree conditions to an Erdős--Gallai-type questions for $(k-d)$-graphs, which is of independent interest.
Finally, we derive Theorem~\ref{thm:main-simple-general} and~\ref{thm:main-simple-k-2} by means of two short solutions to the respective Erdős--Gallai-type problems.

\subsection{Hamilton frameworks}\label{sec:hamilton-frameworks}
In the following, we introduce the notion of Hamilton frameworks, which is a structural relaxation of Hamilton cycles.	
We then show that under minimum degree conditions the problem of finding tight Hamilton cycles can be reduced to finding Hamilton frameworks (Theorem~\ref{thm:framework}).

To motivate the concept of Hamilton frameworks, let us identify some characteristic properties of tight Hamilton cycles.
We start by observing that a tight cycle is connected, in the sense that one can navigate `tightly' between any two edges without leaving the cycle.
This property is formalised as follows.

\begin{definition}[Tight walk]\label{def:tight-walk}
	A \emph{(closed) tight walk} in a $k$-graph $G$ is a (cyclically) ordered set of vertices such that every interval of $k$ consecutive vertices forms an edge.
	The \emph{length} of a tight walk is its number of vertices.
\end{definition}

Note that vertices and edges in a tight walk are allowed to repeat.
In other words, a closed tight walk is a homomorphism of a tight cycle.
Tight walks allow us to define a notion of connectivity that is appropriate for tight cycles.

\begin{definition}[Tight connectivity]\label{def:tightconnectivity}
	A subgraph $H$ of a $k$-graph $G$ is \emph{tightly connected}, if every two edges of $H$ lie on a common tight walk.
	A \emph{tight component} of $G$ is an edge maximal tightly connected subgraph.
\end{definition}
Let us now consider a $k$-graph $G$ with a subgraph $H$.
We will discuss five natural conditions that $H$ needs to satisfy in order to find a tight Hamilton cycle.

Firstly, we require $H$ to be almost {vertex-spanning} in $G$.
Secondly, as explained above, a tight Hamilton cycle in $G$ is tightly connected and therefore we should ask the same for $H$.
This is called the \emph{connectivity} property.
To motivate the third property, let us make the following observation about divisibility.
Let $v(G)$ denote the size of the vertex set of $G$.
Clearly, $v(C) \equiv v(G) \bmod k$ for all tight Hamilton cycles $C$ of $G$.
Hence, we will ask for $H$ to contain a closed tight walk $W$ with $v(W) \equiv v(G) \bmod k$.
In fact, to disassociate this observation from the particular value of $v(G)$, we simply require $H$ to contain a tight walk of every length congruent modulo $k$, or equivalently a closed tight walk of length $1 \bmod k$.
We call this the \emph{divisibility} property.
Fourthly, note that tight Hamilton cycles are quite spacious, as they contain perfect fractional matchings.
Let us explain this in more detail.
We can think of a \emph{matching} as an edge weighting $\bvec w \colon E(G) \to \{0,1\}$ such that  $\sum_{e \ni v} \bvec w(e)\leq1$ for every vertex $v \in V(H)$.
As usual, the matching is \emph{perfect} if these bounds are met with equality.
The linear relaxation of this notion, where $\bvec w$ is allowed to take any value between $[0,1]$, is called a \emph{fractional matching}.
Every $k$-uniform tight Hamilton cycle has a perfect fractional matching, this can be seen by giving each edge weight $1/k$.
We will require the following strengthening of this property to which we refer as the \emph{space} property.

\begin{restatable}[Robustly matchable]{definition}{defrobustperfectmatching}\label{def:robust-perfect-matching}
	Let $H$ be a $k$-graph and $\gamma >0$.
	We say that $H$ is \emph{$\gamma$-robustly matchable} if the following holds.
	For every vertex weighting $\bvec b\colon V(H) \rightarrow [1-\gamma, 1]$, there is an edge weighting $\bvec w\colon E(H) \rightarrow [0, 1]$ with  $\sum_{e \ni v} \bvec w(e) = \bvec b(v)$ for every vertex $v \in V(H)$.
\end{restatable}

Lastly, it will be helpful to have a property that allows us to add a small number of vertices on short paths going into $H$.
Suppose for simplicity that $G$ is $2$-uniform and consider a vertex $v \in V(G)$.
Suppose that there are vertices $x,y \in V(H)$ such that the neighbourhoods of $x$ and $y$ intersect each with distinct vertices in the neighbourhood of $v$.
Say, $a$ is in the common neighbourhood of $x$ and $v$, and $b$ is in the common neighbourhood of $y$ and $v$.
Then $(x,a,v,b,y)$ is a path that attaches $v$ to $H$.
As we will see, this property is quite useful to include the last couple of vertices on an already large cycle and is easily generalisable to higher uniformities.
Now suppose that the minimum relative $1$-degree of $G$ is at least $\delta +o(1)$.
In this case, it suffices to require that $H$ has minimum relative $1$-degree at least $1-\delta$ to guarantee that the aforementioned short paths exist.
We will refer to this as the \emph{outreach} property.
Let us wrap up these observations in the following definition.

\begin{restatable}[Hamilton framework]{definition}{defframework}\label{def:framework}
	Let $k \geq 2$ and $\alpha,\gamma,\delta>0$.
	Suppose that $R$ is a $k$-graph.
	We say that a subgraph $H\subset R$ is an $(\alpha,\gamma,\delta)$-\emph{Hamilton framework}, if $H$ has the following properties:
	\begin{enumerate}[(F1)]
		\item \label{def:framework-spanning} $v(H) \geq (1-\alpha)v(R)$,  \hfill (spanning)
		\item \label{def:framework-connectivity} $H$ is tightly connected,  \hfill (connectivity)
		\item \label{def:framework-divisibility} $H$ contains a tight closed walk of length $1 \bmod k$,  \hfill (divisibility)
		\item \label{def:framework-space} $H$ is $\gamma$-robustly matchable and  \hfill (space)
		\item \label{def:framework-degree} $H$ has minimum relative vertex degree at least $1-\delta +\gamma$. \hfill ({outreach})
	\end{enumerate}
\end{restatable}

\newcommand{\hf}{{\mathrm{hf}}}
We argue that the properties of Hamilton frameworks are not only (essentially) necessary, but also sufficient to find tight Hamilton cycles under minimum degree conditions.
In support of this, it is shown that in order to find a tight Hamilton cycle in a hypergraph $G$ of large minimum degree, it suffices to find a Hamilton framework $H$ in an appropriately defined auxiliary graph $R$.
In our setting, $R$ will inherit the large minimum degree of $G$ with a small error term.
However, some of the edges of $R$ will be, informally speaking, perturbed and therefore need to be avoided.

To express the constant hierarchies in our results and definitions, we use the following standard notation.
We write $x \ll y$ to mean that for any $y \in (0, 1]$ there exists an $x_0 \in (0,1)$ such that for all $x \leq x_0$ the subsequent statements hold.
Hierarchies with more constants are defined in a similar way and are to be read from the right to the left.
Let us now define the Hamilton framework threshold.

\begin{restatable}[Hamilton framework threshold]{definition}{defframeworkthreshold}\label{def:framework-threshold}
	For  $1 \leq d \leq k-1$, the \emph{minimum $d$-degree threshold for $k$-uniform Hamilton frameworks}, denoted by $\hf_d(k)$, is the smallest number $\delta >0$ such that the following holds:
	
	Suppose $\eps,\alpha,\gamma,\mu >0$ and $t \in \NATS$ with $ 1/t  \ll \eps \ll \alpha \ll \gamma \ll \mu$.
	If $R$ is a $k$-graph on $t$ vertices with relative minimum $d$-degree at least $\delta+\mu$ and a set $I \subset E(R)$ of at most $\eps \binom{t}{k}$ \emph{perturbed} edges,
	then $R$ contains an $(\alpha,\delta,\gamma)$-Hamilton framework $H$ that avoids the edges of $I$.
\end{restatable}

Our first main structural result reduces the problem of bounding the tight Hamilton cycle threshold to bounding the Hamilton framework threshold.

\begin{theorem}[Framework Theorem]\label{thm:framework}
	For $1\leq d \leq k-1$, we have $\hc_d (k) \leq \hf_d(k).$
\end{theorem}
We remark that Theorem~\ref{thm:framework} does itself not lead to new bounds on the threshold for tight Hamilton cycles.	
However, similar to the work of Glock, Kühn, Lo, Montgomery and Osthus~\cite{GLMD19} for graph decompositions, it reduces the matter to structural questions closer to the core of the problem.
In particular, past contributions in this line of research had to iterate a series of complex arguments concerning absorption, cover and connection results in order to harness new structural ideas.	
This process is now encapsulated in Theorem~\ref{thm:framework}.
(An overview of the proof can be found in Section~\ref{sec:framework-threshold-overview}.)
Note that the space and outreach conditions of Ham\-ilton frameworks are somewhat stronger than the actual properties of tight Hamilton cycles.
However, these features appear to be easily satisfiable under minimum degree conditions.
It is therefore conceivable that future work on minimum degree thresholds of tight Hamilton cycles can by carried out by studying Hamilton frameworks alone (see Section~\ref{sec:discussion}).
To conclude the exposition of Hamilton frameworks, let us mention that this notion was inspired by our work with Garbe, Lo and Mycroft~\cite{GLL+21} on monochromatic cycle partitioning.
We now turn to finding frameworks under minimum degree conditions.

\subsection{Hamilton vicinities}\label{sec:hamilton-vicinities}
In light of the Framework Theorem (Theorem~\ref{thm:framework}), it suffices to find a Hamilton framework in a hypergraph $R$ of large minimum degree in order to bound the threshold of tight Hamilton cycles.
In the following, we study the existence of such frameworks through the structure of individual neighbourhoods or, more formally, link graphs of $R$.

\begin{definition}(Link graph, neighbourhood)
	Consider a $k$-graph $R$ and a set of $d$ vertices $S \subseteq V(R)$.
	We define the \emph{link graph of $S$ in $R$} as the $(k-d)$-uniform graph $L_R(S)$ with vertex set $V(R)$ and edge set
	$\{X \sm S \colon X \in N(S)\}$, where $N(S)=\{X\in E(R)\colon S\subset X \}$ is the \emph{neighbourhood} of $S$.
	If $R$ is  clear from the context, then we simply write $L(S)$.
	
\end{definition}

To illustrate the approach, let us consider the problem of deciding whether a hypergraph has a perfect fractional matching.
The \emph{density} of a fractional matching in a $k$-graph $H$ is the sum of its weights over all edges divided by~$v(H)$.
Hence any fractional matching in a $k$-graph has density between $0$ and $1/k$, and the fractional matchings which attain the value $1/k$ are precisely the perfect  fractional matchings.
A result of Alon, Frankl, Huang, Rödl, Ruciński and Sudakov~\cite{AFH+12} states that a hypergraph has a perfect fractional matching, provided that every link graph contains a fractional matching of sufficiently large density.

\begin{theorem}[{\cite[Proposition 1.1]{AFH+12}}]\label{thm:fractional-matching}
	For $1\leq d\leq k-1$, consider a $k$-graph~$H$. If every link graph of $H$ on $d$ vertices contains a matching of density $1/k$, then $H$ has a perfect fractional matching.
\end{theorem}

Hence, the problem of bounding the minimum $d$-degree threshold of perfect fractional matchings in $k$-graphs can be approached by studying edge density conditions under which large matchings appear in $(k-d)$-graphs; a well-known problem raised by Erdős~\cite{Erd65}.

In light of this, one can wonder whether there exists a similar reduction for the setting of tight Hamilton cycles.
Cooley and Mycroft~\cite{CM17} used such an approach to find almost spanning tight cycles under minimum vertex degree conditions in $3$-graphs.
Here, we solve this problem in its full generality and develop an analogue of Theorem~\ref{thm:fractional-matching} for tight Hamilton cycles for all $1 \leq d \leq k-1$.
We propose a series of properties, whose union we call Hamilton vicinities, and show that one can derive Hamilton frameworks from Hamilton vicinities under of minimum degree thresholds (Theorem~\ref{thm:hamilton-vicinities}).
Together with the Framework Theorem (Theorem~\ref{thm:framework}), this allows us to reduce the problem of bounding the minimum $d$-degree threshold of tight Hamilton cycles in $k$-graphs to the question under which density conditions $(k-d)$-graphs satisfy the aforementioned properties.

Let us now introduce these concepts more formally.
We call a set (edge) of $j$ elements a \emph{${j}$-set} (\emph{$j$-edge}).
For $1 \leq j \leq k$ and a $k$-graph $R$, we let the \emph{shadow graph} $\partial_j(R)$ of $R$ at level $j$ be the $j$-graph on $V(R)$ whose edges are the $j$-sets contained in the edges of $R$.
Our definitions will be based on edges $S$ in  $\partial_d(R)$ instead of $d$-sets of vertices in $V(R)$.
This will be helpful to avoid perturbed edges as encountered in the setting of the Hamilton framework threshold.
In practice, the $d$-th level shadow graph will always be nearly complete.
In a slight abuse of notation, we will often abbreviate $S \in E(\partial_d(R))$ with $S \in \partial_d(R)$.

\begin{definition}[Vicinity, generated graph]\label{def:vicinity-vanilla}
	Given a $k$-graph $R$, we say that a family $\cC = (C_S)_{S \in \partial_d(R)}$ is a \emph{$d$-vicinity}, if $C_S$ is a subgraph of $L(S)$ for every edge $S \in \partial_d(R)$.
	We define the $k$-graph $H$ \emph{generated by $\cC$} as the subgraph of $R$ with vertex set $V(H)=V(R)$ and edge set
	\[E(H)=\bigcup_{S \in \partial_d(R)} \{A \cup S \colon A \in C_S\}.\]
\end{definition}

We remark that, for a $d$-vicinity $\cC$ in $R$ which generates a graph $H$, we have $C_S \subset L_{H}(S)$ for every element $S \in \partial_d(R)$, however this is not necessarily an equality.
Some edges in $L_{H}(S)$ might only appear due to another $C_{S'} \in \cC$.
More precisely, it may happen that there exist $S' \neq S$ in $\partial_d(R)$ such that $X' \in C_{S'}$ and $S \subseteq X' \cup S'$.
In this case, $X' \cup S'$ belongs to $H$ and thus $X = (X' \cup S') \setminus S$ belongs to $L_H(S)$, but the definition of vicinities does not guarantee that $X$ belongs to $C_S$.

Now consider a  $k$-graph $R$ whose $d$-th level shadow graph is  complete (for simplicity) and a subgraph $H$ generated by a $d$-vicinity $\cC=(C_S)_{S \in \partial_d(H)}$.
What kind of properties of $C_S$ could guarantee that $H$ is a Hamilton framework (Definition~\ref{def:framework})?
We know from Theorem~\ref{thm:fractional-matching} that $H$ has a perfect fractional matching, provided that every member of the vicinity contains a large matching.
We will show that a marginal strengthening of this property suffices to cover the space property~\ref{def:framework-space}.

Next, let us have a closer look at tight connectivity, which is required for the connectivity property~\ref{def:framework-connectivity}.
We will obtain this condition through the combination of two separate properties.
For sake of illustration, suppose that $R$ is $3$-uniform and assume that each ($2$-uniform) element of the vicinity $(C_S)_{S \in \partial_d(R)}$ is tightly connected, which in $2$-graphs is equivalent to ordinary graph connectivity.
We refer to this property as \emph{inner connectivity}.
It follows that each of the $3$-uniform subgraphs $\{A \cup S \colon A \in E(C_S)\} \subset R$ is tightly connected as well.
Now suppose that, in addition to this, every two elements of the vicinity have an edge in common.
This property is called \emph{outer connectivity}.
It follows that $H$ must be tightly connected (see Figure~\ref{fig:connectivity}).
This idea generalises to hypergraphs with larger uniformities and hence we can obtain the tight connectivity of $H$ by combining inner and outer connectivity.

\begin{comment}

\begin{figure}
	\begin{tikzpicture}[thick]
	
	\coordinate (b1) at (-1.6,0.5);	
	\coordinate (b2) at (-0.6,-0.4);	
	\coordinate (b4) at (0.6,-0.4);	
	\coordinate (b3) at (1.6,0.5);	
	\coordinate (b2u) at (-0.6,-0.385);	
	\coordinate (b2d) at (-0.6,-0.415);	
	\coordinate (b4u) at (0.6,-0.385);	
	\coordinate (b4d) at (0.6,-0.415);	
	
	\coordinate (c0) at (-2.4,1);	
	\coordinate (c1) at (-3,0);	
	\coordinate (c2) at (-4,0.5);	
	\coordinate (c3) at (-4,-0.5);	
	
	\coordinate (c4) at (1,0.9);	
	\coordinate (c5) at (2.6,-0.7);	
	\coordinate (c6) at (4,1);	
	\coordinate (c7) at (3.9,0);	
	\coordinate (c8) at (4,-1);
	
	\coordinate (x1) at ($(b1)+(-90-10:3)$);
	\coordinate (x2) at ($(b3)+(-90+10:3)$);
	
	\triple{(x1)}{(b1)}{(b2)}{6pt}{1pt}{DarkDesaturatedBlue,opacity=0.8}{DarkDesaturatedBlue,opacity=0.2};
	\triple{(x1)}{(b2)}{(b4)}{6pt}{1pt}{DarkDesaturatedBlue,opacity=0.8}{DarkDesaturatedBlue,opacity=0.2};
	\triple{(x1)}{(c0)}{(b1)}{6pt}{1pt}{DarkDesaturatedBlue,opacity=0.8}{DarkDesaturatedBlue,opacity=0.2};
	\triple{(x1)}{(c1)}{(b1)}{6pt}{1pt}{DarkDesaturatedBlue,opacity=0.8}{DarkDesaturatedBlue,opacity=0.2};
	\triple{(x1)}{(c2)}{(c1)}{6pt}{1pt}{DarkDesaturatedBlue,opacity=0.8}{DarkDesaturatedBlue,opacity=0.2};
	\triple{(x1)}{(c3)}{(c1)}{6pt}{1pt}{DarkDesaturatedBlue,opacity=0.8}{DarkDesaturatedBlue,opacity=0.2};
	
	\triple{(x2)}{(b2)}{(b4)}{6pt}{1pt}{DarkDesaturatedBlue,opacity=0.8}{DarkDesaturatedBlue,opacity=0.2};
	\triple{(x2)}{(b4)}{(b3)}{6pt}{1pt}{DarkDesaturatedBlue,opacity=0.8}{DarkDesaturatedBlue,opacity=0.2};
	\triple{(x2)}{(c4)}{(b3)}{6pt}{1pt}{DarkDesaturatedBlue,opacity=0.8}{DarkDesaturatedBlue,opacity=0.2};
	\triple{(x2)}{(b3)}{(c5)}{6pt}{1pt}{DarkDesaturatedBlue,opacity=0.8}{DarkDesaturatedBlue,opacity=0.2};
	\triple{(x2)}{(c5)}{(c6)}{6pt}{1pt}{DarkDesaturatedBlue,opacity=0.8}{DarkDesaturatedBlue,opacity=0.2};
	\triple{(x2)}{(c5)}{(c7)}{6pt}{1pt}{DarkDesaturatedBlue,opacity=0.8}{DarkDesaturatedBlue,opacity=0.2};
	\triple{(x2)}{(c5)}{(c8)}{6pt}{1pt}{DarkDesaturatedBlue,opacity=0.8}{DarkDesaturatedBlue,opacity=0.2};
	
	\path [draw=VividOrange] (b1) -- (b2);
	\path [draw=VividOrange] (c1) -- (b1);
	\path [draw=VividOrange] (c0) -- (b1);
	\path [draw=VividOrange] (c1) -- (c2);
	\path [draw=VividOrange] (c1) -- (c3);
	\path [draw=VividOrange] (c1) -- (c3);
	\path [draw=VividOrange] (b2u) -- (b4u);
	\path [draw=PureOrange] (b2d) -- (b4d);
	\path [draw=PureOrange] (b3) -- (b4);
	\path [draw=PureOrange] (b3) -- (c4);
	\path [draw=PureOrange] (c5) -- (b3);
	\path [draw=PureOrange] (c5) -- (c6);
	\path [draw=PureOrange] (c5) -- (c7);
	\path [draw=PureOrange] (c5) -- (c8);
	
	\node at (x1) [label={[label distance=3pt]270:$S_1 = \{ x_1 \}$}] {};
	\node at (x2) [label={[label distance=3pt]270:$S_2 = \{ x_2 \}$}] {};
	\node at (b2) [label={[label distance=3pt]90:$y_1$}] {};
	\node at (b4) [label={[label distance=3pt]90:$y_2$}] {};
	
	\tikzstyle{every node}=[circle,draw,fill,inner sep=0pt,minimum width=3pt]
	
	\draw (b1) node{};
	\draw (b2) node{};
	\draw (b3) node{};
	\draw (b4) node{};
	
	\draw (c0) node{};
	\draw (c1) node{};
	\draw (c2) node{};
	\draw (c3) node{};
	\draw (c4) node{};
	\draw (c5) node{};
	\draw (c6) node{};
	\draw (c7) node{};
	\draw (c8) node{};
	
	\draw (x1) node{};
	\draw (x2) node{};
	
	\end{tikzpicture}
	
	\caption{Inner and outer connectivity.
		Suppose $k = 3$ and $d = 1$, thus $k-d = 2$.
		Consider a $3$-graph $R$.
		We look at vicinities in the link graphs of vertices of $R$.
		In this case, $S_1, S_2$ are sets of size $d = 1$.
		The orange edges indicate the $2$-graph $C_{S_1}$  and the yellow edges indicate the $2$-graph $C_{S_2}$.
		The inner connectivity property of vicinities implies that both $C_{S_1}$ and $C_{S_2}$ are tightly connected (or simply connected in the case of $2$-graphs).
		As a result, the $3$-graph $H_i = \{ A \cup S_i : A \in E(C_{S_i}) \}$ is tightly connected for each $i=1,2$.
		The outer connectivity property of vicinities implies that there exists an edge in the intersection $C_{S_1} \cap C_{S_2}$; in the example given above $y_1 y_2$ is such an edge.
		As a result, the union $3$-graph $H_1 \cup H_2$ is tightly connected.
	}
	\label{fig:connectivity}
\end{figure}

\end{comment}

In a similar vein, the divisibility property~\ref{def:framework-divisibility} of frameworks can be derived from a combination of local and global structure.
We define the following two substructures.

\begin{restatable}[Switcher]{definition}{defswitcher}\label{def:switcher}
	A \emph{switcher} in an $\ell$-graph $C$ is an edge $A$ of $C$ with a distinguished \emph{central vertex} $a\in A$ such that, for every $b \in A$, the $(\ell-1)$-sets $A \setminus \{ a \}$ and $A \setminus \{ b \}$ share a neighbour in $C$; more formally, there is a vertex $c \neq a,b$ such that $(A \cup \{c\}) \sm \{a\}$ and  $(A \cup \{c\}) \sm \{b\}$ are both edges in $C$.
\end{restatable}

\begin{restatable}[Arc]{definition}{defarc}\label{def:arc}
	Let $R$ be a $k$-graph with a $d$-vicinity $\cC= (C_S)_{S \in \partial_d(R)}$.
	We say that a tuple of $k+1$ vertices  $(v_1,\dots,v_{k+1})$ is an \emph{arc} for $\cC$, if
	\begin{itemize}
		\item $\{v_1,\dots,v_{d}\} \in \partial_d(R)$ with $\{v_{d+1},\dots,v_{k}\} \in C_{\{v_1,\dots,v_{d}\}}$ and
		\item $\{v_2,\dots,v_{d+1}\} \in \partial_d(R)$ with $\{v_{d+2},\dots,v_{k+1}\} \in C_{\{v_2,\dots,v_{d+1}\}}$.
	\end{itemize}
\end{restatable}

We will show that $H$ has a tight cycle of length $1 \bmod k$, provided that $\cC$ admits an arc and every member of $\cC$ has a switcher.
From this it follows that $H$ has the divisibility property~\ref{def:framework-divisibility}.

Finally,  by inequality~\eqref{equ:degree-types-monotonicity} the spanning and outreach properties~\ref{def:framework-spanning} and~\ref{def:framework-degree} are satisfied provided that every element of $\cC$ has large edge density.
To be precise, the \emph{edge density} of a $k$-graph $G$ on $n$ vertices is defined as $e(G)/\binom{n}{k}$.
Motivated by this discussion, we summarise the structural properties of vicinities as follows:

\begin{restatable}[Hamilton vicinity]{definition}{defvicinity}\label{def:vicinity-Hamilton}
	Let $1 \leq d \leq k-1$ and $\gamma,\delta > 0$.
	Suppose that $R$ is a $k$-graph. 
	We say that a $d$-vicinity $\cC= (C_S)_{S \in \partial_d(R)}$ is \emph{$(\gamma,\delta)$-Hamilton}, if for every $S,S' \in \partial_d(R)$:
	\begin{enumerate}[label=\textnormal{(V\arabic*)}]
		\item \label{def:vicinity-inner-connectivity} $C_S$ is tightly connected, \hfill (inner connectivity)
		\item \label{def:vicinity-outer-connectivity} $C_S$ and $C_{S'}$ intersect in an edge, \hfill (outer connectivity)
		\item \label{def:vicinity-divisibility} $C_S$ has a switcher and the vicinity $\cC$ has an arc, \hfill (divisibility)
		\item \label{def:vicinity-space} $C_S$ has a fractional matching of density $ 1/k + \gamma$ and \hfill (space)
		\item \label{def:vicinity-density} $C_S$ has edge density at least $1-\delta + \gamma$.  \hfill ({density})
	\end{enumerate}
\end{restatable}

Recall that in order to bound the Hamilton framework threshold, we need to find a Hamilton framework $H$ in a graph $R$, where $H$ avoids a small set of perturbed edges of $R$.
When working with vicinities, it will be convenient to ignore vertex sets that are contained in too many perturbed edges.
The following definition captures this observation with appropriate precision.
For a $k$-graph $G$ on vertex set $V$, we denote by $\overline{G}$ the complement of $G$, that is the $k$-graph on $V$ that contains precisely the $k$-sets of $V$ that are not in $G$.

\begin{restatable}[Perturbed degree]{definition}{defperturbeddegree} \label{def:perturbed-degree}
	Let $1\leq d \leq k-1$ and $\alpha,\delta>0$.
	We say that a $k$-graph $R$ has \emph{$\alpha$-perturbed} minimum relative $d$-degree at least $\delta$, if the following holds for every $1\leq j\leq d$:
	\begin{enumerate}[(P1)]
		\item \label{itm:perturbed-degree-density} every edge of $\partial_j(R)$ has relative degree at least $\delta$ in $R$,
		\item \label{itm:perturbed-degree-shadow-almost-complete} $\overline{\partial_j(R)}$ has edge density at most $\alpha$ and
		\item \label{itm:perturbed-degree-lower-levels} each $(j-1)$-edge of $\partial_{j-1}(R)$ has relative degree less than $\alpha$ in $\overline{\partial_j(R)}$.
	\end{enumerate} 
\end{restatable}

We remark that in our context the transition from perturbed edges to perturbed minimum degree comes at negligible costs  (see Lemma~\ref{lem:degree-cleaning}).
When looking for Hamilton vicinities, it will also be convenient to ignore \emph{isolated} vertices, which are vertices not contained in any edge.
Next, we define a threshold for Hamilton vicinities.

\newcommand{\hv}{{\mathrm{hv}}}
\begin{restatable}[Hamilton vicinity threshold]{definition}{defvicinitythreshold}\label{def:vicinity-threshold}
	For $1 \leq d\leq k-1$, the \emph{minimum} \emph{$d$-degree} \emph{threshold} \emph{for} \emph{$k$-uniform} \emph{Hamilton} \emph{vicinities}, denoted by $\hv_d(k)$, is the smallest number $\delta >0$  such that the following holds: 
	
	 Suppose $\alpha,\gamma,\mu >0$ and $t \in \NATS$ satisfy $1/t  \ll \alpha \ll \gamma \ll \mu$.
	Every $k$-graph $R$ on $t$ vertices with $\alpha$-perturbed minimum relative $d$-degree at least $\delta+\mu$ and no isolated vertices admits a $(\gamma,\delta)$-Hamilton $d$-vicinity.
\end{restatable}

Our second main structural result shows that we can bound the Hamilton framework threshold by the Hamilton vicinity threshold.
\begin{theorem}[Vicinity Theorem]\label{thm:hamilton-vicinities}
	For $1 \leq d \leq k-1$, we have $\hf_d(k) \leq \hv_d(k)$.
\end{theorem}

Note that a straightforward exercise shows that $\hv_{d}(k)\leq 1/2$ for all $k \ge 2$ and $d=k-1$, which reproduces the results of R\"odl, Ruciński and Szemerédi~\cite{RRS09a} on codegree thresholds.
The following two lemmas (whose proofs are also fairly short) imply Theorem~\ref{thm:main-simple-general} and~\ref{thm:main-simple-k-2} in conjunction with Theorem~\ref{thm:framework} and Theorem~\ref{thm:hamilton-vicinities}.

\begin{restatable}{lemma}{lemvicinitythresholdgeneral}\label{lem:vicinity-threshold-general}
	For $k \ge 2$ and $1 \leq d \leq k-1$, we have $\hv_d(k) \leq 2^{-1/(k-d)} $.
\end{restatable}

\begin{restatable}{lemma}{lemvicinitythresholdspecific}\label{lem:vicinity-threshold-d-k=2}
	For $k \geq 3$, we have $\hv_{k-2}(k) \leq  5/9$.
\end{restatable}

\subsection{Organisation of the paper}
The rest of the paper is organised into two parts.
The first part is dedicated to the structural analysis of the auxiliary graph $R$ and spans Section~\ref{sec:vicinities->framworks}--\ref{sec:finding-vicinities}.
In Section~\ref{sec:vicinities->framworks}, it is described how Hamilton frameworks can be obtained from Hamilton vicinities and together with a proof of the Vicinity Theorem (Theorem~\ref{thm:hamilton-vicinities}).
In Section~\ref{sec:finding-vicinities}, we show how to find Hamilton vicinities under minimum degree conditions and prove Lemma~\ref{lem:vicinity-threshold-general} and~\ref{lem:vicinity-threshold-d-k=2}.

The second part includes Section~\ref{sec:framework-threshold-overview}--\ref{sec:absorbing} and covers the proof of the Framework Theorem (Theorem~\ref{thm:framework}).
In Section~\ref{sec:framework-threshold-overview} we outline the proof in a simplified setting and explain how the properties of Hamilton frameworks translate into the typical absorption, path cover and connecting arguments used to construct tight Hamilton cycle.
We deploy these techniques in the setting of hypergraph regularity which we introduce in Section~\ref{sec:regularity}.
In Section~\ref{sec:framework-threshold-proof} we give a formal proof of the Framework Theorem subject to an Absorption Lemma (Lemma~\ref{lem:absorption}) and a Cover Lemma (Lemma~\ref{lem:cover}).
The remaining sections are devoted to the proof of these two lemmas.
In Section~\ref{sec:tools} we collect tools to work with regularity and use them to give the proof of the Cover Lemma in Section~\ref{sec:connecting}.
In Section~\ref{sec:preabsorbing} we prove further technical results which are required for the proof of the Absorbing Lemma in Section~\ref{sec:absorbing}.
We conclude the paper with a discussion and open problems in Section~\ref{sec:discussion}.

\section{From vicinities to frameworks}\label{sec:vicinities->framworks}
The goal of this section is to prove the Vicinity Theorem (Theorem~\ref{thm:hamilton-vicinities}).
We will derive this result from the following three lemmas.

\begin{restatable}[Connectivity and divisibility]{lemma}{lemconnectivity}\label{lem:connectivity}
	Let $1 \leq d\leq k-1$, and let $R$ be a $k$-graph with a $d$-vicinity $\cC= (C_S)_{S \in \partial_d(R)}$.
	Suppose that, for every $S,S' \in \partial_d(R)$:
	\begin{enumerate}[label=\textnormal{(V\arabic*)}]
		\item  $C_S$ is tightly connected,  
		\item  $C_S$ and $C_{S'}$ intersect if $|S \cap S'|=d-1$ and  
		\item  $C_S$ has a switcher and the vicinity $\cC$ has an arc.  
	\end{enumerate}
	Then the graph $H \subset R$ generated by $\cC$ satisfies:
	\begin{enumerate}[label=\textnormal{(F\arabic*)}] \addtocounter{enumi}{1}
		\item   $H$ is tightly connected and
		\item   $H$ contains a tight closed walk of length $1 \bmod k$.   
	\end{enumerate}
\end{restatable}

\begin{restatable}[Space]{lemma}{lemspace}\label{lem:space}
	Let $1 \leq d \leq k-1$ and $\gamma, \alpha, \delta > 0$ such that $\alpha, \gamma \ll 1/k$.
	Let $R$ be a $k$-graph with $\alpha$-perturbed minimum $d$-degree at least $\delta$ and no isolated vertices.
	Let $\cC= (C_S)_{S \in \partial_d(R)}$ be $d$-vicinity in $R$.
	Suppose that, for every $S \in \partial_d(R)$:
	\begin{enumerate}[label=\textnormal{(V\arabic*)}]\addtocounter{enumi}{3}
		\item   $C_S$ has a fractional matching of density $ 1/k + \gamma  $.  
	\end{enumerate}
	Then the graph $H \subset R$ generated by $\cC$ satisfies:
	\begin{enumerate}[label=\textnormal{(F\arabic*)}]\addtocounter{enumi}{3}
		\item  $H$ is $\gamma$-robustly matchable.  
	\end{enumerate}
\end{restatable}

\begin{restatable}[Perturbed degree]{lemma}{lemdegreecleaning} \label{lem:degree-cleaning}
	Let $t,d,k$ be integers with $1 \leq d \leq k-1$ and $\delta, \alpha, \eps >0$ with $1/t \ll \eps \ll \alpha \ll \delta, 1/k$.
	Let $R$ be a $k$-graph on $t$ vertices with minimum relative $d$-degree $\delta_{d}^*(R) \ge \delta$.
	Let $I$ be a subgraph of $R$ of edge density at most $\eps$.
	Then there exists a vertex spanning subgraph $R' \subseteq R - I$ of $\alpha$-perturbed minimum relative $d$-degree at least $\delta-\alpha$.
\end{restatable}

We postpone the proofs of these results to the following subsections and continue by showing  the Vicinity Theorem.

\begin{proof}[Proof of Theorem~\ref{thm:hamilton-vicinities}]
	Let $\delta = \hv_d(k)$  and $\eps,\alpha,\gamma >0$ with $t_0 \in \NATS$ such that
	\[\frac{1}{t} \ll \eps \ll \alpha \ll \alpha' \ll \gamma \ll \mu \ll \delta,\frac{1}{k}.\]
	More precisely, we choose these constants such that $t,\eps,\alpha,\mu$ are compatible with the constant hierarchy given by the definition of the {minimum}  {$d$-degree}  {threshold}  {for}  {$k$-uniform}  {Hamilton}  {vicinities} (Definition~\ref{def:vicinity-threshold}),
	moreover $t,\eps,2\alpha,\mu$ satisfy the conditions of Lemma~\ref{lem:space} and $t,\eps,\alpha,\delta$  satisfy the conditions of Lemma~\ref{lem:degree-cleaning}.
	
	Consider a $k$-graph $R$ on $t\geq t_0$ vertices with minimum relative $d$-degree at least $\delta+ 3\mu$ and a set $I$ of at most $\eps \binom{t}{k}$ perturbed edges.
	In order to bound the 	{minimum}  {$d$-degree}  {threshold}  {for}  {$k$-uniform}  {Hamilton}  frameworks (Definition~\ref{def:framework-threshold}), we need to show that $R$ contains an $(\alpha, \gamma, \delta)$-Hamilton framework $H$ that avoids all edges of $I$.
	
	We start by selecting a subgraph of $R$, whose degree properties are well-behaved.
	By Lemma~\ref{lem:degree-cleaning}, there exists a vertex spanning subgraph $R' \subseteq R - I$ of $\alpha$-perturbed minimum relative $d$-degree at least $\delta+ 2\mu$.
	By~\ref{itm:perturbed-degree-shadow-almost-complete} of perturbed degrees (Definition~\ref{def:perturbed-degree}) we have $\overline{\partial_1(R')} \leq  \alpha t$, which is equivalent to say that $R'$ has at most $\alpha t$ isolated vertices.
	Let $R''$ be obtained from $R'$ by deleting its isolated vertices.
	It follows that $R''$ has still $(2\alpha)$-perturbed minimum relative $d$-degree at least $\delta+\mu$ and at least $(1-\alpha)t$ vertices, none of which is isolated.
	
	By the definition of the Hamilton vicinity threshold, $R''$ has a $(2\gamma, \delta)$-Hamilton $d$-vicinity $\cC=(C_S)_{S \in \partial_d(R)}$ which generates a graph $H$.
	Note that, by construction, $H$ does not have any edge in common with $I$.
	Moreover, $V(H) = V(R'')$, and thus $H$ spans at least $(1 - \alpha)t$ vertices, so~\ref{def:framework-spanning} of Definition~\ref{def:framework} holds.
	By Lemma~\ref{lem:connectivity} and~\ref{lem:space} the generated graph $H$ also satisfies~\ref{def:framework-connectivity}--\ref{def:framework-space} of Definition~\ref{def:framework}.
	
	It remains to show that $H$ has property~\ref{def:framework-degree}.
	Recall that $R''$ has no isolated vertices and $R''$ has $(2\alpha)$-perturbed minimum relative $d$-degree at least $\delta + \mu$.
	By repeatedly applying Definition~\ref{def:perturbed-degree}~\ref{itm:perturbed-degree-lower-levels}, we deduce  that each $v \in V(R'')$ is contained in at least $(1 - 2 \alpha)^{d-1} \binom{v(R'')-1}{k-1} \ge (1 - 2 d \alpha) \binom{v(R'')-1}{k-1}$ many $d$-sets in $\partial_d(R'')$.
	Since $\partial_d(R'') = \partial_d(H)$, this implies that each $v \in V(H)$ has relative degree at least $1 - 2 d \alpha$ in the $d$-graph $\partial_d(H)$.
	Moreover, every $d$-set in $\partial_d(H)$ has relative degree at least $1 - \delta + 2 \gamma$ in $H$, because $H$ was generated from a $(2\gamma, \delta)$-Hamilton $d$-vicinity and Definition~\ref{def:vicinity-Hamilton}~\ref{def:vicinity-density}.
	Combining this with $\alpha \ll \gamma, 1/k$, we obtain that $H$ satisfies $\delta^\ast_1(H) \ge (1 - 2 d \alpha)(1 - \delta + 2 \gamma) \ge 1 - \delta + \gamma$, which implies property~\ref{def:framework-degree}.
	Thus, $H$ is an $(\alpha, \gamma, \delta)$-framework.
\end{proof}

The remainder of this section is dedicated to the proofs of Lemma~\ref{lem:connectivity},~\ref{lem:space} and~\ref{lem:degree-cleaning}.
The arguments consist of a series of isolated observations.
In order to preserve the overall narrative, we will use the following naming conventions:
\begin{center}
	\begin{tabular}{ c  *{1}{|l} }
		symbol & role in the final proof\\
		\hline
		$R$, $d$, $k$, $t$ & basic $k$-graph with $t$ vertices and minimum large $d$-degree
		\\ 	\hline $H$ & subgraph of $R$, candidate for a Hamilton framework
		\\ 	\hline $L(S), \ell$ & link graph in $R$ of a $d$-set $S$ of $V(R)$, is an $\ell$-graph with $\ell =k-d$
		\\ 	\hline $C$ & subgraph of an $L(S)$, element of a (Hamilton) $d$-vicinity
	\end{tabular}
\end{center}

\subsection{Connectivity and divisibility}\label{sec:connected+dense->odd-cycle}
In this section, it is shown how the connectivity and divisibility properties of Hamilton vicinities (Definition~\ref{def:vicinity-Hamilton}) imply the respective properties of Hamilton frameworks (Definition~\ref{def:framework}).
More precisely, we prove Lemma~\ref{lem:connectivity}.

%\lemconnectivity*
We start by introducing a directed version of tight connectivity.
This is followed by a few technical observations about switchers (Proposition~\ref{prop:switcher->tightly-connected}--\ref{prop:switcher->closed-tight-walk-mod-k}) and the connectivity of graphs generated by vicinities (Proposition~\ref{prop:inner-walk}--\ref{prop:outer-walk-multi-step}), which lead into the proof of Lemma~\ref{lem:connectivity}.
Finally, we discuss a few possibilities of generalising the notion of vicinities with regards to connectivity and divisibility.

\newcommand{\ord}[1]{{\overrightarrow{#1}}}
Tight cycles and walks visit the vertices of edges in certain orders.
While tight connectivity guarantees that there is tight walk between any two edges, it does not tell us anything about the vertex ordering in which these edges are connected.
To be more precise about this, we define a \emph{directed edge} in a $k$-graph to be a $k$-tuple whose vertices correspond to an \emph{underlying} edge.
For instance the directed edges $(a,b,c)$, $(c,b,a)$ share the underlying edge $\{a,b,c\}$.
The next definition is a strengthening of tight connectivity.

\begin{definition}[Strong connectivity]\label{def:strongtconnectivity}
	A uniform hypergraph is \emph{strongly connected}, if every two directed edges are on a common tight walk.
\end{definition}

In the following, we often have to combine several tight walks to a longer one.
Let us introduce some notation to express this.
Given two sequences of vertices $A=(a_1,\dotsc,a_s)$ and $B=(b_1,\dotsc, b_q)$, we denote their \emph{concatenation} by \[AB =(a_1,\dotsc,a_s)(b_1,\dotsc, b_q) = (a_1,\dotsc,a_s,b_1,\dotsc, b_q).\]
If $m \ge 3$ and $A_1, A_2, \dotsc, A_m$ are sequences of vertices, then we iteratively define $A_1 \dotsm A_m = A_1 ( A_2 \dotsm A_m )$, as usual.

Since we have defined tight walks as sequences of vertices (Definition~\ref{def:tight-walk}), this allows us to talk about concatenation of tight walks.
But note that even if $A=(a_1,\dots,a_s)$ and $B=(b_1,\dots, b_q)$ are tight walks in a $k$-graph $R$, it does not follow from the definition that the sequence $AB$ is a tight walk (since for some $1 \leq i \leq k-1$, the edge $\{ a_{s - k + i + 1}, \dotsc, a_{s}, b_1, \dotsc, b_i \}$ might not be present in $R$).
However, this will never be an issue in all of our uses of this notation, and we will always be clear to mention that a given concatenation of tight walks is a tight walk if that is true in the given context.

Now let us see how switchers can help us to control the ordering of tight walks and, together with this, their length.
We start by recalling the definition.

\defswitcher*

As mentioned above, strong connectivity implies tight connectivity, but the converse is in general not true.
However, tightly connected hypergraphs which contain a switcher (Definition~\ref{def:switcher}) are also strongly connected as the next proposition shows.
Informally speaking, a switcher allows us to switch the direction of an edge in a closed walk, hence the name.

\begin{proposition}\label{prop:switcher->tightly-connected}
	Let $C$ be a tightly connected $\ell$-graph that contains a switcher.
	Then $C$ is strongly connected.
\end{proposition}
\begin{proof}
	Let $A$ be a switcher in $C$ with central vertex $a \in A$.
	More precisely, $A$ is an edge in $C$ and, for every $b \in A$, the $(\ell-1)$-sets $A \setminus \{ a \}$ and $A \setminus \{ b \}$ share a neighbour in $C$.
	We show that any two directed edges with underlying edge $A$ are on a tight walk.
	To see how this implies that $C$ is strongly connected, consider any two directed edges $\cB_1$ and $\cB_2$ of $C$.
	Since $H$ is tightly connected, there are walks $W_1$ and $W_2$ starting in $\cB_1$ and $\cB_2$ respectively and ending with the edge $A$.
	By the above, there is a tight walk $W$ starting and ending in $A$ such that $W_1WW_2$ is a tight walk.
	
	It remains to prove that any two directed edges $\cA$ and $\cB$ with underlying edge $A$ are on on a tight walk.
	Note that it suffices to discuss the case where $\cA$ and $\cB$ differ in exactly two coordinates $i$ and $j$, with $i <j$.
	Moreover, we can assume that the $i$-th entry of $A$ is the central vertex $a$.
	Let $b$ denote  the $j$-th entry of $\cA$ and $\cB$, respectively.
	By definition of a switcher, there is a vertex $c \notin A$ such that $(A \cup \{c\})\sm \{a\}$ and $(A \cup \{c\})\sm \{b\}$ are edges of $H$.
	Let $\cC$ be the $k$-tuple obtained from $\cA$ by replacing the $i$-th entry with $c$ and the $j$-th entry with $a$.
	We also denote $\cA=(x_1,\dots,{a},\dots,{b},\dots x_k)$.
	It  follows that the concatenation
	\begin{align*}
	{\cA\cC\cB=(x_1,\dots,{a},\dots,{b},\dots x_k)(x_1,\dots,{c},\dots,{a},\dots x_k)(x_1,\dots,{b},\dots,{a},\dots ,x_k)}
	\end{align*}
	is a tight walk as desired.
\end{proof}

Recall that the length of a tight walk is the number of its vertices.
The next proposition tells us that a switcher can  also be used to control the divisibility properties of a tight walk.
\begin{proposition}\label{prop:switcher->closed-tight-walk-mod-k}
	Let $C$ be a tightly connected $\ell$-graph that contains a switcher.
	Then $C$ has a closed tight walk of length $-1 \bmod \ell$.
\end{proposition}
\begin{proof}
	The statement is trivial for $\ell =1$.
	Similarly, for $\ell =2$, the switcher generates a triangle, which covers this case.
	This allows us to assume that $\ell \geq 3$.
	Let $A=\{a_1,\dots,a_\ell\}$ be a switcher in $C$ with central vertex $a_1$.
	Hence, $A \sm \{a_1\}$ and $A \sm \{a_i\}$ share a neighbour $b_i$ for each $2 \leq i \leq   \ell$.
	We define $\cA_0=(a_1,\dots,a_\ell)$, $\cA_1=(b_2,a_1)(a_3,a_4,\dots,a_\ell)$ and, for $2 \leq i \leq \ell-2$,
	\begin{align*}
	\cA_i=	(a_2,\dots,a_{i})(b_{i+1},a_1)(a_{i+2},\dots,a_{\ell}).
	\end{align*}
	Finally, let $\cA_{\ell-1}=(a_2,\dots,a_{\ell-1},b_{\ell-1})$.
	Note that, $\cA_0,\dots,\cA_{\ell-2}$ are each $\ell$-tuples, while $\cA_{\ell-1}$ is an $(\ell-1)$-tuple.
	We claim that the concatenation $W:=\cA_0\cA_1\dotsm\cA_{\ell-1}$ is a tight walk.
	This follows because the $\ell-1$ vertices to the left of each $b_{i+1}$ consist of $A\sm \{a_1\}$ and the $\ell-1$ vertices to the right of each $b_{i+1}$ consist of $A\sm \{a_{i}\}$.
	Since $W$ has length $-1 \bmod \ell$, the proposition follows.
\end{proof}

In the following, we will often use the names of directed edges to refer to their underlying edges.
For instance, for a directed edge $\cX$ with underlying edge $X$, we will write $C_\cX$ meaning $C_X$, where $C_X$ is an element of a vicinity.
Or we will choose a directed edge $\cY$ in a set of (undirected) edges $E$, meaning that $\cY$ is (some) directed edge with underlying edge in $E$.

We now turn to connecting up edges of graph $H$ generated by a vicinity $\cC$.
The next proposition states that we can navigate within $H$ by looping around an element of $\cC$.

\begin{proposition}\label{prop:inner-walk}
	For $1 \leq d \leq k-1$, let $R$ be $k$-graph with a subgraph $H$ that is generated by a $d$-vicinity $\cC$.
	Suppose that $\cC$ satisfies conditions~\ref{def:vicinity-inner-connectivity} and~\ref{def:vicinity-divisibility}  of Lemma~\ref{lem:connectivity}.
	Consider a directed edge $\cS \in \partial_d(R)$ and two directed edges $\cA,\cB \in C_S$.
	Then there is a tight walk $W$ in $H$ starting with $\cS\cA$ and ending with $\cS\cB$.
	Moreover, $W$ has length $0 \bmod k$.
\end{proposition}
\begin{proof}
	We denote $\cC= (C_S)_{S \in \partial_d(R)}$.
	By~\ref{def:vicinity-divisibility} and Proposition~\ref{prop:switcher->closed-tight-walk-mod-k}, there is a tight closed walk $W_1$ of length $-1 \bmod (k-d)$ in $C_S$.
	By~\ref{def:vicinity-inner-connectivity} and Proposition~\ref{prop:switcher->tightly-connected}, there is a tight walk $W_2$ in $C_S$ that starts with $\cA$, ends with $\cB$ and contains $W_1$ as a subwalk.
	Let $q$ denote the length of $W_2$.
	We obtain another tight walk $W_3$ from $W_2$ by replacing $W_1$ with the concatenation of $q+1 \bmod (k-d)$ copies of $W_1$.
	Hence $W_3$ is a tight walk in $C_S$, which starts with $\cA$, ends with $\cB$ and has length $0 \bmod (k-d)$.
	
	Suppose that $W_3$ has length $m(k-d)$ and write $W_3=(a_1,\dots,a_{m(k-d)})$.
	For each $1 \leq i \leq m-1$, let $W'_i = \cS (a_{m\cdot i+1},\dots,a_{m\cdot i+d-k})$.
	Note that each $W'_i$ (as an unordered tuple) is an edge of $H$, and also note that $W'_0 = \cS \cA$ and $W'_{m-1} = \cS \cB$.
	We define $W$ by	
	\begin{align*}
	W = W'_0 W'_1 \dotsm W'_{m-1}.
	\end{align*}
	Note that $W$ is a tight walk in $H$, starting with $\cS\cA$ and ending with $\cS\cB$.
	Moreover, $W$ has length $0 \bmod k$.
\end{proof}

The next two propositions allow us two navigate between different elements of a vicinity $\cC$.

\begin{proposition}\label{prop:outer-walk-one-step}
	For $1 \leq d \leq k-1$, let $R$ be $k$-graph with a subgraph $H$ that is generated by a $d$-vicinity $\cC= (C_S)_{S \in \partial_d(R)}$.
	Suppose that $\cC$ satisfies conditions~\ref{def:vicinity-inner-connectivity}--\ref{def:vicinity-divisibility} of of Lemma~\ref{lem:connectivity}.
	Consider directed edges $\cS,\cT \in \partial_d(R)$ and $\cA \in C_\cS$, $\cB \in C_\cT$.
	Suppose that $\cS$ and $\cT$ differ in exactly one coordinate.
	Then there is a tight walk $W$ in $H$ starting with $\cS\cA$ and ending with $\cT\cB$.
	Moreover, $W$ has length $0 \bmod k$.
\end{proposition}
\begin{proof}
	Let $\cS= (s_1,\dots,s_i,\dots,s_d)$ and $\cT= (s_1,\dots,t_i,\dots,s_d)$, where $s_i \neq t_i$.
	By property~\ref{def:vicinity-outer-connectivity} of Definition~\ref{def:vicinity-Hamilton}, there is a directed edge  $\cQ$ in $C_{\cS} \cap C_{\cT}$.
	It follows that 
	\begin{align}\label{equ:outer-walk-one-step}
	\cS\cQ\cT=(s_1,\dots,s_i,\dots,s_d)\cQ(s_1,\dots,t_i,\dots,s_d)
	\end{align}
	is a tight walk in $H$.
	Since $\cQ$ is in $C_{\cS} \cap C_{\cT}$ and by Proposition~\ref{prop:inner-walk}, there is a tight walk $W_1$ that starts with $\cS\cA$ and ends with $\cS\cQ$.
	Similarly, there is a tight walk $W_2$ starting with $\cT\cQ$ and ending with $\cT\cB$.
	Moreover, both $W_1$ and $W_2$ have length $0 \bmod k$.
	It follows by equation~\eqref{equ:outer-walk-one-step} that $W_1W_2$ is tight walk, as desired.
\end{proof}

\begin{proposition}\label{prop:outer-walk-multi-step}
	Let $R$ be a $k$-graph with a subgraph $H$ that is generated by a $d$-vicinity $\cC= (C_S)_{S \in \partial_d(R)}$.
	Suppose that $\cC$ satisfies conditions~\ref{def:vicinity-inner-connectivity}--\ref{def:vicinity-divisibility}  of  Lemma~\ref{lem:connectivity}.
	Consider directed edges $\cS,\cT \in \partial_d(R)$ and $\cA \in C_\cS$, $\cB \in C_\cT$.
	Then there is a tight walk $W$ starting with $\cS\cA$ and ending with $\cT\cB$.
	Moreover, $W$ has length $0 \bmod k$.
\end{proposition}

\begin{proof}
	Let $0 \leq r \leq d$ be the number of indices in which $\cS$ and $\cT$ differ (as ordered tuples).
	We prove the proposition by induction on $r$.
	If $r = 1$ then the result follows from Proposition~\ref{prop:outer-walk-one-step}.
	Hence, we assume $r \ge 2$ and that the result is known for $r-1$.
	
	By property~\ref{def:vicinity-outer-connectivity} of Definition~\ref{def:vicinity-Hamilton}, there exists an edge $Q$ in $C_\cS \cap C_\cT$.
	Let $q \in Q$ and note that $q \notin \cS,\cT$.
	Suppose $\cS$ and $\cT$ differ in an index $i$.
	We obtain $\cS'$ from $\cS$ and $\cT'$ from $\cT$ by replacing in both cases the $i$-th coordinate with $q$.
	
	Observe that $\cS' \in \partial_d(R)$ since the vertices of $\cS'$ are contained in $\cS\cC$ and $\cS \cup \cC$ is an edge of $R$.
	This allows us to pick a directed edge $\cA'$ in $C_{{\cS'}}$.
	By Proposition~\ref{prop:outer-walk-one-step}, there is a tight walk $W_1$ beginning with $\cS \cA$ and ending with $\cS' \cA'$.
	Analogously, we obtain a directed edge $\cB' \in C_{\cT'}$ and a walk $W_3$ that starts with $\cT' \cB'$ and ends in $\cT \cB$.
	Moreover, we can assume that the lengths of $W_1$ and $W_3$ are each $0 \bmod k$.
	
	By construction, $\cS'$ and $\cT'$ differ in $r-1$ coordinates.
	Hence, the induction hypothesis implies that there exists a walk $W_2$ of length $0 \bmod k$ that begins with $\cS' \cA'$ and ends with $\cT' \cA'$.
	It follows that the concatenation $W_1 W_2 W_3$ presents the desired walk.
\end{proof}

Let us now derive Lemma~\ref{lem:connectivity} from the above observations.
For the second part of the proof, we recall the definition of an arc.
\defarc*
\begin{proof}[Proof of Lemma~\ref{lem:connectivity}]
	We start by showing that $H$ is tightly connected.
	Denote the vicinity by $\cC= (C_S)_{S \in \partial_d(R)}$ and consider two arbitrary edges $X$ and $Y$ of $H$.
	Since $\cC$ generates $H$, we can partition $X= S \cup A$ and $Y = T \cup B$ such that $A \in C_S$ and $B \in C_T$.
	It follows by Proposition~\ref{prop:outer-walk-multi-step} that there is a tight walk starting with $X$ and ending with $Y$.
	
	Next, we show that $H$ contains a closed walk of length $1\bmod k$.
	By assumption, $\cC$ admits an arc $(v_1,\dots,v_{k+1})$.
	By Proposition~\ref{prop:outer-walk-multi-step}, there is a tight walk $W$ of length $0 \bmod k$ starting with $(v_2,\dots,v_{k+1})$ and ending with $(v_1,\dots,v_{k})$.
	By Definition~\ref{def:arc}, $W(v_{k+1})$ is a closed tight walk of length $1 \bmod k$.
\end{proof}

We conclude our analysis of connectivity and divisibility with two remarks.
Firstly, it is worth noting that to obtain the connectivity property~\ref{def:framework-connectivity} for $H$, it suffices that the (generating) vicinity $\cC$ has the inner and outer connectivity property~\ref{def:vicinity-inner-connectivity} and~\ref{def:vicinity-outer-connectivity}, the divisibility property~\ref{def:vicinity-divisibility} is not needed for this.
Moreover, the outer connectivity property~\ref{def:vicinity-outer-connectivity} can be relaxed further as follows.
Suppose that there exists a tight walk $W$ in the $d$-th shadow $\partial_d(R)$ that visits all $d$-edges of $\partial_d(R)$ such that $C_S$ and $C_{S'}$ intersect for every two consecutive $d$-tuples $S$ and $S'$ in $W$.
Under this assumption, it can be shown that $H$ is tightly connected.
Secondly, the divisibility property~\ref{def:vicinity-divisibility} can also be somewhat relaxed.
To find a tight closed walk of length $1 \bmod k$, it suffices that only the elements $C_S$ of the vicinity which correspond to the arc (and the walk from the arc's last $k$ vertices to the arc's first $k$ vertices) contain switchers.
We will not need these strenghtened versions of the lemmas, so we omit the proofs.

\subsection{Space}\label{sec:fractional-matching}
In this section, we show how the space property of Hamilton vicinities (Definition~\ref{def:vicinity-Hamilton}) implies the respective property of Hamilton frameworks (Definition~\ref{def:framework}).
More precisely, we prove Lemma~\ref{lem:space}.
To this end, let us recall the definition of robustly matchable hypergraphs.

\defrobustperfectmatching*

%\lemspace*

To prove Lemma~\ref{lem:space}, we will use techniques from linear programming.
In the following, $\vnorm{\cdot}$ denotes the $L_1$-norm.

\begin{definition}[$\bvec b$-fractional matching]\label{def:b-fractional-matching}
	Let $H$ be a $k$-graph with vertex weighting $\bvec b\colon V(H) \to [0, 1]$.
	A \emph{$\bvec b$-fractional matching} is a function $\bvec w\colon E(H) \to [0, 1]$ such that $\sum_{e \ni v} \bvec w(e) \leq \bvec b(v)$ for every vertex $v \in V(H)$.
	We say that $\bvec w$ is \emph{perfect}, if these inequalities are met with equality.
\end{definition}

The \emph{size} of  a fractional matching $\bvec w$ is $\vnorm{\bvec w } = \sum_{e \in E(H)} \bvec w(e)$.
To put this definition into context, we note that a $\gamma$-robustly matchable graph $H$ admits a $\bvec b$-fractional matching of size $\vnorm{\bvec b}/k$ for every  vertex weighting $\bvec b\colon V(H) \rightarrow [1-\gamma, 1]$.

The problem of finding a $\bvec b$-fractional matching of maximum size in a $k$-graph $H$ can be expressed as a linear program:
\begin{align*}
\text{maximise} \qquad& \vnorm{\bvec w}=\sum_{e \in {E(H)}} \mathbf{w}(e)   \nonumber \\
\text{subject to}  \qquad&   \text{$\mathbf{w}(e) \geq 0$ for all $e \in E(H)$,} \\
&  \text{$\sum_{e \ni v} \mathbf{w}(e) \leq \bvec{b}(v)$ for all $v \in V(H)$.}
\end{align*}
where $\bvec{w} \in \mathbb{R}^{E(H)}$.
Let $\nu(H, \bvec{b})$ be the optimal value of this problem, that is, the maximum of $\vnorm{\bvec w}$ attained by a $\bvec b$-fractional matching $\bvec w$ of $H$.
The dual linear program is
\begin{align*}
\text{minimise} \qquad& \bvec c \cdot \bvec b =\sum_{v \in {V(H)}} \mathbf{c}(v) \bvec{b}(v)   \nonumber \\
\text{subject to}  \qquad&   \text{$\mathbf{c}(v) \geq 0$ for all $v \in V(H)$,} \\
&  \text{$\sum_{v \in e} \mathbf{c}(v) \geq 1$ for all $e \in E(H)$,}
\end{align*}
where $\bvec c \in \mathbb{R}^{V(H)}$.
The feasible vectors $\mathbf{c} \in \mathbb{R}^{V(H)}$ of the dual program above are called \emph{fractional vertex covers of $H$}.
Let $\tau(H, \bvec{b})$ be the optimal value of the dual program, namely, the minimum $\bvec c \cdot \bvec b$ attained over all fractional vertex covers $\bvec c$ of $H$.
Note that by strong duality, we have $\nu(H, \bvec{b}) = \tau(H, \bvec{b})$.

The proof of the next proposition uses ideas of Alon, Frankl, Huang, Rödl, Ruci\'{n}ski and Sudakov~\cite[Proposition 1.1]{AFH+12}.

\begin{proposition} \label{proposition:usingduality}
	Let $H$ be a $k$-graph and $\bvec b\colon V(H) \rightarrow [0, 1]$.
	Suppose there exists ${m} \leq \vnorm{\bvec b}/k$ such that, for every $v \in V(H)$, the link graph $L_H(\{v\})$ has a $\bvec b$-fractional matching of size ${m}$.
	Then $H$ has a $\bvec b$-fractional matching of size ${m}$.
\end{proposition}

\begin{proof}
	To find a contradiction, suppose $H$ has no $\bvec b$-fractional matching of size ${m}$.
	This is equivalent to $\nu(H, \bvec{b}) < {m}$.
	Let ${m}^\ast = \nu(H, \bvec{b})$, by duality we have $\tau(H, \bvec{b}) = {m}^\ast$.
	Thus there exists a fractional vertex cover $\bvec{x} \in \mathbb{R}^{V(H)}$ with ${m}^\ast = \bvec x \cdot \bvec b = \sum_{v \in V(H)} \bvec{x}(v) \bvec{b}(v)$.
	
	Consider $u \in V(H)$ for which $\bvec{x}(u)$ is minimal.
	Since
	\[\bvec{x}(u) \vnorm{\bvec b}   \leq \sum_{v \in V(H)} \bvec{x}(v) \bvec{b}(v) = {m}^\ast < {m} \leq \frac{\vnorm{\bvec b}}{k},\]
	it follows that $\bvec{x}(u) < 1/k$.
	Define $\bvec{x'} \in \mathbb{R}^{V(H)}$ by
	\begin{align*}
	\bvec{x'}(v) = \frac{\bvec{x}(v) - \bvec{x}(u)}{1 - k\bvec{x}(u)}
	\end{align*}
	for all $v \in V(H)$.
	
	We show that $\bvec{x'}$ is a fractional vertex cover of $H$ with $\bvec{x}' \cdot \bvec b < {m}$.
	We use ${m}^\ast < {m} \leq \vnorm{\bvec b}/k$ to obtain that
	\begin{align*}
	\bvec{x}' \cdot \bvec b &= \sum_{v \in V(H)} \bvec{x'}(v) \bvec{b}(v)\\
	& = \frac{1}{1 - k \bvec{x}(u)} \left( \sum_{v \in V(H)} \bvec{x}(v) \bvec{b}(v) - \bvec{x}(u) \sum_{v \in V(H)} \bvec{b}(v) \right) \\
	& = \frac{1}{1 - k \bvec{x}(u)} \left( {m}^\ast - \bvec{x}(u) \vnorm{\bvec b} \right)
	< \frac{1}{1 - k \bvec{x}(u)} \left( {m} - \bvec{x}(u) {m} k \right)
	= {m}.
	\end{align*}
	From $\bvec{x}(u) < 1/k$ we get $\bvec{x}'(v) \ge 0$ for all $v \in V(H)$.
	Using that $\bvec{x}$ is a fractional vertex cover of $H$ and also that $H$ is a $k$-graph we see, for every $e \in E(H)$,
	\begin{align*}
	\sum_{v \in e} \bvec{x'}(v)
	& = \frac{1}{1 - k \bvec{x}(u)} \left[ \sum_{v \in e} \bvec{x}(v) - k \bvec{x}(u) \right]
	\ge 1
	\end{align*}
	Thus $\bvec{x'}$ is indeed a fractional vertex cover of $H$ with $\bvec{x}' \cdot \bvec{b} < {m}$.
	
	Crucially, also note that $\bvec{x'}(u) = 0$.
	Since the vertex $u$ is isolated in the link graph $L_H(\{u\})$,
	we infer that $\bvec{x'}$ is also a fractional vertex cover of $L_H(\{u\})$,
	and $\bvec{x}' \cdot \bvec{b} < m$. 	
	This implies that $\tau(L_H(\{u\}), \bvec{b}) < {m}$.
	By strong duality, we also have $\nu(L_H(\{u\}), \bvec{b}) < {m}$.
	But this contradicts the assumption that $L_H(\{u\})$ has a $\bvec b$-fractional matching of size ${m}$.
\end{proof}

If $\bvec{b}$ is the all-ones function, then $\bvec{b}$-fractional matchings correspond to the usual fractional matchings in hypergraphs.
Thus the following corollary is immediate.

\begin{corollary} \label{corollary:liftingfractionalmatchings}
	Let $H$ be a $k$-graph and ${m} \leq v(H)/k$.
	Suppose that, for every $v \in V(H)$, the link graph $L_H(\{v\})$ has a fractional matching of size ${m}$.
	Then $H$ has a fractional matching of size $m$.
\end{corollary}

Now we show that in the setting of graphs with $\alpha$-perturbed minimum $d$-degree $\delta$, large fractional matchings in the link graphs of $d$-sets can be lifted to large fractional matchings in the link graphs of vertices.

\begin{proposition} \label{proposition:fromdlinksto1links}
	Let $1 \leq d \leq k-1$ and $ \alpha,\gamma, \delta > 0$ such that $\alpha, \gamma \ll 1/k$.
	Let $H$ be a $k$-graph on $t$ vertices with $\alpha$-perturbed minimum $d$-degree $\delta$.
	Suppose that for every $S \in \partial_d(S)$, the link graph $L(S)$ contains a fractional matching of size at least $(1/k + \gamma)t$.
	Then, for every (singleton) edge $S'\in \partial_1(H)$, the link graph $L(S')$ contains a fractional matching of size at least $(1/k + \gamma)t$.
\end{proposition}

\begin{proof}
	Given $t$, $\gamma$, $\alpha$, $\delta$ and $k \ge 2$, we proceed by induction on $d$.
	Note that the base case $d = 1$ is immediate.
	
	Suppose now that we are given $1 < d \leq k-1$ and that the result is known for all $d' < d$.
	Let $S \subseteq V(H)$ be a $(d-1)$-set in $\partial_{d-1}(H)$.
	Consider an edge $S'  \in \partial_1(L_H(S))$.
	Note that this implies in particular that $S \cup S'$ is an edge in $\partial_d(H)$.
	By assumption, $L_H(S \cup S')$ has a fractional matching of size at least $(1/k + \gamma) t$.
	Let $H'$ be the subgraph of $L_H(S)$ induced on the non-isolated vertices of $L_H(S)$.
	We deduce that the link graph $L_{H'}(\{v\})$ contains a fractional matching of size at least $(1/k + \gamma)t$ for every $v \in V(H')$.

	Recall that $H$ has $\alpha$-perturbed minimum $d$-degree $\delta$.
	Hence by Definition~\ref{def:perturbed-degree}~\ref{itm:perturbed-degree-lower-levels}, $S$ has at most $\alpha t$ neighbours in $\overline{\partial_{d}(R)}$.
	It follows that  $v(H') = |V(H')| = |\partial_1(L_H(S))| \ge (1 - \alpha) t$.	
	Since $\alpha, \gamma \ll 1/k$, we have $(1/k + \gamma)t \leq |V(H')|/(k - d + 1)$.
	Thus we can invoke Corollary~\ref{corollary:liftingfractionalmatchings} with
	\begin{center}
		\begin{tabular}{ c  *{4}{|c} }
			object/parameter & $H'$ & $k-d+1$ & $v(H')$ & $(1/k + \gamma)t$ \\
			\hline
			playing the role of & $H$ & $k$ & $t$ & $m$
		\end{tabular}
	\end{center}
	to deduce $H'$ (and thus $L_H(S)$) has a fractional matching of size $(1/k + \gamma)t$.
	Since $S$ was arbitrary, we have deduced that for any $S \in \partial_{d-1}(H)$, the link graph $L_H(S)$ contains a fractional matching of size $(1/k + \gamma)t$.
	Hence, we are done by the induction hypothesis in the case $d' = d - 1$. 
\end{proof}	

\begin{proof}[Proof of Lemma~\ref{lem:space}]
	We start by showing that the link graphs of vertices of $H$ contain large matchings.
	Suppose $H$ has $t$ vertices.
	By assumption, $C_S$ contains a fractional matching of size $(1/k + \gamma)t$ for every $S \in \partial_d(H)$.
	Moreover, by definition of the reduced graph (Definition~\ref{def:vicinity-vanilla}), $C_S$ is a subgraph of the link graph of $S$ in $H$.
	By Proposition~\ref{proposition:fromdlinksto1links}, it follows that the link graph $L_H(\{v\})$ contains a fractional matching of size $(1/k + \gamma)t$ for every $v \in \partial_1(H)$.
	Moreover, from $\partial_1(R)=V(R)$ and the definition of the induced graph, we derive  that $\partial_1(H)=V(H)=V(R)$.
	
	Now we show that $H$ is $\gamma$-robustly matchable (Definition~\ref{def:robust-perfect-matching}).
	Consider an arbitrary vertex weighting $\bvec b: V(H) \rightarrow   [1 - \gamma, 1]$.
	We have to find a $\bvec b$-matching $\bvec w \in  \REALS^{E(H)}$ such that $\sum_{e \ni v} \bvec w(e) = \bvec b(v)$ for every vertex $v \in V(H)$.
	This is equivalent to find a $\bvec b$-matching with size $\vnorm{\bvec b}/k$, where  $\vnorm{\bvec b} = \sum_{v \in V(H)} \bvec b(v)$.
	
	We will use Proposition~\ref{proposition:usingduality} to find $\bvec w$.
	To this end, consider an arbitrary vertex $v$ in $V(H)$.
	Let $\bvec{x}$ be a fractional matching in $L_H(\{v\})$ of size at least $(1/k + \gamma)t$ (which we have by assumption) and let $\bvec{w}' = (1 - \gamma) \bvec{x}$.
	Since $1 - \gamma \leq \bvec{b}(u)$ for all $u \in V(H)$, it follows that $\bvec{w}'$ is a $\bvec{b}$-fractional matching in $L_H(\{v\})$.
	Moreover, $\bvec w'$ has size at least $(1 - \gamma)(1/k + \gamma)t \ge t/k \ge \vnorm{\bvec b}/k$.
	By scaling, we can assume that $\bvec w'$ has size exactly $\vnorm{\bvec b}/k$.
	Hence, the conditions of Proposition~\ref{proposition:usingduality} are satisfied with $m = \vnorm{\bvec b}/k$.
	It follows that $H$ has a $\bvec{b}$-fractional matching of size $\vnorm{\bvec b}/k$;
	which is, as discussed previously, enough.
\end{proof}

\subsection{Perturbed degree}\label{sec:degree-cleaning}

In the following, we show that a hypergraph $H$ with a small number of perturbed edges contains a subgraph $H'$, which avoids all perturbed edges while retaining essentially the minimum degree of $H$.
This is formalised in Lemma~\ref{lem:degree-cleaning}, which we prove at the end of this section.
We remark that Lemma~\ref{lem:degree-cleaning} is a strengthened version of a lemma of Han, Lo and the second author~\cite[Lemma 8.8]{HLS17}.

%\defperturbeddegree*
%
%\lemdegreecleaning*

We start by introducing further notations.
For a hypergraph $G$, it will be convenient to refer to its edge size by $e(G)$ and its vertex size (as done before)  by $v(G)$.
To distinguish between degrees of different subgraphs, we will denote the degree of a set $S$ in $G$ by $\deg_G(S)$  and its relative degree by $\reldeg_G(S)$.
As before, we refer to sets of $j$ elements (usually vertices) as $j$-sets.

Consider a graph $R$ with a subgraph $I$ as in Lemma~\ref{lem:degree-cleaning}.
The following definition identifies the edges of each shadow graph level of $I$ that have large degree in higher levels.

\begin{definition}[Gradation]
	Given a $k$-graph $I$ and $\beta > 0$, the \emph{$\beta$-gradation of I} is a sequence of uniform hypergraphs $(I_1, \dotsc, I_k)$ which are defined recursively as follows:
	\begin{enumerate}[(i)]
		\item Initially $I_k = I$ and
		\item given $I_{j+1}$ we obtain $I_j$ from $\partial_j(I_{j+1})$ by deleting all edges with relative degree less than $\beta$ in $I_{j+1}$.
	\end{enumerate}
\end{definition}

In order to obtain a subgraph $R'$ of $R$ with well-behaved degree properties, we need to avoid the elements of the gradation of $I$.
The next definition identifies the edges of $R$ that are affected by a gradation.

\begin{definition}[Degree perturbation]
	Suppose a $k$-graph $R$ has a subgraph $I$ with $\beta$-gradation $(I_1, \dotsc, I_k)$ for a $\beta >0$.
	For $1 \leq j \leq d$, let $F_j$ be the subgraph of $R$ consisting of those edges which contain at least one edge of $I_j$.
	We define the \emph{$d$-degree $\beta$-perturbation of $I$ in $R$} as the $k$-graph $F = \bigcup_{j=1}^d F_j$.
\end{definition}

The proof of Lemma~\ref{lem:degree-cleaning} can be sketched as follows.
Given the setup of the statement, we will choose $\beta$ such that $\eps \ll \beta \ll \alpha$.
Let $F$ be the $d$-degree $\beta$-perturbation of $I$  in $R$.
We define $R' = R - I - F$, that is, $R'$ is the $k$-graph obtained from $R$ by deleting both $I$ and its degree perturbation $F$.
We prove Lemma~\ref{lem:degree-cleaning} by showing that $R'$ has the desired perturbed minimum degree.
To analyse the degree properties of $R'$, we consider a suitable gradation $(C_1, \dotsc, C_k)$ of $F$ and make the following two observations.
Firstly, the $j$-th shadow graph of $R'$ and ${C_j}$ have no edge in common (Proposition~\ref{claim:cleaninglemma-largedegree-notspoiled}).
Secondly, for $\delta>0$ as in  Lemma~\ref{lem:degree-cleaning}, every edge $Y$ outside of ${C_j}$ must have relative degree at least $\delta - \alpha$ in $R'$ (Proposition~\ref{claim:cleaninglemma-notspoiled-largedegree}).
Together, we obtain the following equivalence for every set $Y$ of $j$ vertices:
\begin{align*}
Y \notin C_j \quad  \Leftrightarrow \quad \reldeg_{R'}(Y) \ge \delta - \alpha \quad \Leftrightarrow \quad Y \in \partial_j(R').
\end{align*}
It then follows rather quickly that $R'$ has the desired perturbed relative degree.

We will approach the formal proof of Lemma~\ref{lem:degree-cleaning} with a series of short observations about gradations and degree perturbations.
The next proposition provides conditions under which the edge density of a hypergraph bounds the edge density of its gradation.

\begin{proposition} \label{proposition:cleaninglemma-sizes}
	Let $\eps,\beta > 0$ with $\beta \ge \eps^{1/k}$.
	Let $I$ be a $k$-graph with edge density at most $\eps$, and let $(I_1, \dotsc, I_k)$ be the $\beta$-gradation of $I$.
	Then, for each $1 \leq j \leq k$, the $j$-graph $I_j$ has edge density at most $\beta^j$.
\end{proposition}
\begin{proof}
	We use induction on $k-j$.
	If $k-j = 0$, then $I_k = I$ has edge density at most $\eps \leq \beta^k$.
	Now consider $k - j > 0$.
	Suppose that $I$  has $t$ vertices.
	Double counting and the definition of $I_j$ gives
	\begin{align*}
	e(I_{j})\beta(t - j) \leq \sum_{X \in E(I_j)} \deg_{I_{j+1}}(X) = (j+1) e(I_{j+1}).
	\end{align*}
	Hence, the inductive hypothesis for $k-(j+1)$ implies
	\begin{align*}
	e(I_{j}) \leq \frac{j+1}{\beta(t - j)} e(I_{j+1}) \leq \beta^{j} \binom{t}{j},
	\end{align*}
	as desired.
\end{proof}

Given a gradation of $I$, the $j$-sets (of vertices of $I$) outside of level $I_j$ have, by definition, low degree in the following level $I_{j+1}$.
The next proposition allows us to bound their degree in the higher levels.

\begin{proposition} \label{proposition:cleaninglemma-degrees}
	Let $1 \leq i \leq j \leq k-1$ and $\beta >0$.
	Let $I$ be a $k$-graph on $t$ vertices and $(I_1, \dotsc, I_k)$ be the $\beta$-gradation of $I$.
	Then each $i$-set $X$ outside of ${I_{i}}$ satisfies $\deg_{I_{j}}(X) \leq (j - i) \beta t^{j - i}$.
\end{proposition} 

\begin{proof}
	We prove the statement by induction on $j$ starting from $i$.
	The base case $j = i$ holds trivially, so assume $j \ge i + 1$.
	
	Fix an $i$-set $X$ outside of ${I_{i}}$.
	Denote by $B$ the set of all supersets $Y \supseteq X$ of size $j - 1$.
	To estimate $\deg_{I_j}(X)$, we consider the degree of each $Y \in B$ in $I_j$.
	Clearly, we have
	\begin{align}
	\deg_{I_j}(X)
	& \leq \sum_{Y \in B} \deg_{I_j}(Y) \nonumber \\
	& = \sum_{Y \in B \cap E(I_{j - 1})} \deg_{I_{j}}(Y) + \sum_{Y \in B \setminus E(I_{j - 1})} \deg_{I_{j}}(Y). \label{equation:cleaninglemma-twosums}
	\end{align}
	
	We estimate the sums in the last term separately.		
	By the induction hypothesis, the proposition holds for $j - 1$.
	Hence, we have
	\begin{align*}
	|B \cap E(I_{j - 1})|
	& =
	\deg_{I_{j - 1}}(X)
	\leq (j - 1 - i) \beta t^{j - 1 - i}.
	\end{align*}
	Since each $Y \in B \cap E(I_{j - 1})$ has degree at most $t$ in $I_{j}$, we obtain
	\begin{align}
	\sum_{Y \in B \cap E(I_{j - 1})} \deg_{I_{j}}(Y)
	& \leq (j - 1 - i) \beta t^{j - i}. \label{equation:cleaninglemma-sum1}
	\end{align}
	
	On the other hand, by the definition of $I_{j - 1}$, each $(j-1)$-set $Y$ outside of ${I_{j - 1}}$ is contained in less than $\beta (t - j + 1) \leq \beta t$ edges of $I_{j}$.
	We can also (crudely) bound $|B \setminus E(I_{j - 1})| \leq |B| \leq \binom{t}{j  - 1 - i} \leq t^{j - 1 - i }$.
	Together, this implies
	\begin{align}
	\sum_{Y \in B \setminus E(I_{j - 1})} \deg_{I_{j}}(Y)
	& \leq |B \setminus E(I_{j - 1})| \beta t
	\leq \beta t^{j - i}. \label{equation:cleaninglemma-sum2}
	\end{align}
	Combining inequalities~\eqref{equation:cleaninglemma-sum1} and~\eqref{equation:cleaninglemma-sum2} with inequality~\eqref{equation:cleaninglemma-twosums}, we get
	\begin{align*}
	\deg_{I_{j}}(X)
	& \leq (j - i - 1) \beta t^{j - i} + \beta t^{j - i} = (j - i) \beta t^{j - i},
	\end{align*}
	as desired.
\end{proof}

The next proposition shows that we can bound the edge density of a degree perturbation, for appropriately chosen constants.

\begin{proposition}\label{prop:surface-damage}
	Let $1 \leq d \leq k-1$ and $0< \eps \ll 1/k$.
	Suppose $R$ is a $k$-graph  with a subgraph $I$ of edge density at most $\eps$.
	Then the $d$-degree $(\eps^{1/k})$-perturbation $F$ of $I$ in $R$ has edge density at most $2^k \eps^{1/k}$.
\end{proposition}
\begin{proof}
	Let $(I_1, \dotsc, I_k)$ be the $(\eps^{1/k})$-gradation of $I$ and $F = \bigcup_{j=1}^d F_j$ be the {$d$-degree $(\eps^{1/k})$-perturbation} of $I$ in $R$.
	Suppose that $R$ has $t$ vertices.
	Trivially, each set of $j$ vertices is contained in at most $\binom{t-j}{k-j}$ edges of $R$.
	By Proposition~\ref{proposition:cleaninglemma-sizes}, each $I_j$ has edge density at most $\eps^{j/k}$.
	This implies that
	\begin{align*}
	e(F_j)
	\leq e(I_j) \binom{t-j}{k-j}
	\leq \eps^{j/k} \binom{t}{j} \binom{t-j}{k-j}
	= \eps^{j/k} \binom{k}{j} \binom{t}{k}.
	\end{align*}
	Hence, we obtain
	\begin{align*}
	e(F)
	\leq \sum_{j=1}^d e(F_j)
	\leq \sum_{j=1}^d \eps^{j/k} \binom{k}{j} \binom{t}{k}
	\leq 2^k \eps^{1/k} \binom{t}{k},
	\end{align*}
	as desired.
\end{proof}

Now we can prove the aforementioned proposition, which states that no edge of $\partial_j(R')$ is in ${C_j}$ as well.

\begin{proposition} \label{claim:cleaninglemma-largedegree-notspoiled}
	Let $1 \leq d \leq k-1$, $1/t \ll \eps \ll \alpha  \ll \delta,1/k$ and $\beta = \eps^{1/(2k)}$.
	Let $R$ be a $k$-graph on $t$ vertices with a vertex spanning subgraph $I$ of edge density at most $\eps$.
	Denote by $F$ the $d$-degree $(\eps^{1/k})$-perturbation of $I$ in $R$.
	Suppose that $F$ has edge density at most $\beta$.
	Define $R' = R - (F \cup I)$.
	Let $(C_1, \dotsc, C_k)$ be the $(\beta^{1/k})$-gradation of $F$.
	Then it follows that $\partial_j(R')$ and ${C_j}$ are disjoint for every $1 \leq j \leq d$.
\end{proposition}

\begin{proof}
	Suppose the proposition is false and $Y$ an edge in both $\partial_j(R')$ and $C_j$.
	Since $(C_1, \dotsc, C_k)$ is the $(\beta^{1/k})$-gradation of $F$, there exist at least $\beta^{1/k} (t - j)$ edges in $C_{j+1}$ which contain $Y$.
	Iterating this, we obtain that $Y$ is contained in at least $\beta^{(k-j)/k} \binom{t-j}{k-j} \ge \beta \binom{n-j}{k-j}$ edges of $C_k = F$.
	Equivalently, $\deg_F(Y) \ge \beta \binom{n-j}{k-j}$.
	We will find a contradiction by showing that $\deg_F(Y) < \beta \binom{n-j}{k-j}$.
	Since $F = \bigcup_{\ell = 1}^d F_\ell$ and $\beta = \eps^{1/(2k)} \ll 1/k$, it is enough to prove that, 
	\begin{align}\label{equ:cleaning-lemma-large-degree-not-spoiled-1}
	\text{$\deg_{F_{\ell}}(Y) \leq k^{4k+1} \eps^{1/k} \binom{n-j}{k-j}$ for all $1 \leq \ell \leq  d$.}
	\end{align}
	
	Fix an arbitrary $1 \leq \ell \leq  d$.
	For $i \in \{0, \dotsc, \ell \}$, let 
	\begin{align}\label{equ:def-Iiell}
	I^i_\ell = \{ Z \in I_\ell \colon  |Z \cap Y| = i \}.
	\end{align}
	Then $\{ I^0_\ell, \dotsc, I^{\ell}_\ell\}$ is a partition of the edges of $I_\ell$.
	Recall the definition of the $d$-degree $(\eps^{1/k})$-perturbation $F$ of $I$ in $R$.
	From this, it follows that, for every $k$-set $X' \in F_\ell$, there exists $Z \in I_\ell$ such that $Z \subseteq X'$.
	We can therefore write $F_\ell = \bigcup_{i=0}^{\ell} F^i_\ell$, where for each $0 \leq i \leq \ell$ and each edge $X'$ of $F^i_\ell$, there exists $Z \in I^i_\ell$ such that $Z \subseteq X'$.
	Hence, to show inequality~\eqref{equ:cleaning-lemma-large-degree-not-spoiled-1}, it suffices to prove that, 
	\begin{align}\label{equ:cleaning-lemma-large-degree-not-spoiled-2}
	\text{$\deg_{F^i_\ell}(Y) \leq \eps^{1/k} k^{4k} \binom{t-j}{k-j}$ for all $0 \leq i \leq \ell$.}
	\end{align}
	
	Now fix an arbitrary $0 \leq i \leq \ell$.
	Each $X' \in F^i_{\ell}$ contributing to $\deg_{F^i_{\ell}}(Y)$ must contain both $Y$ and a set $Z \in I^i_{\ell}$, so $|Y \cup Z| = j + \ell - i$.
	If $j+\ell-i > k$, then $\deg_{F^i_{\ell}}(Y) = 0$, so assume $j+\ell-i \leq k$ from now on.
	We have
	\begin{align} \label{equation:cleaninglemma-boundYandFil}
	\deg_{F^i_{\ell}}(Y) \leq |I^i_{\ell}| \binom{t}{k - (j + \ell - i)}.
	\end{align}
	
	We will finish by bounding
	\begin{align} \label{equation:cleaninglemma-boundIiiel}
	|I^i_{\ell}|    \leq \eps^{1/k} k^{2k} t^{\ell - i}.
	\end{align}
	To see how this concludes the proof note that together with inequality~\eqref{equation:cleaninglemma-boundYandFil}, we obtain
	\begin{align*}
	\deg_{F^i_{\ell}}(Y)
	& \leq k^{2k} \eps^{1/k} t^{\ell - i} \binom{t}{k - (j + \ell - i)} \\
	& \leq \eps^{1/k} k^{2k} t^{\ell - i} t^{k - (j + \ell - i)}
	\leq \eps^{1/k} k^{4k} \binom{t - j}{k - j},
	\end{align*}
	where we used $1/t \ll 1/k$ and crudely bounded $i, j, \ell \leq k$.
	This gives inequality~\eqref{equ:cleaning-lemma-large-degree-not-spoiled-2} as desired.
	
	It remains to show inequality~\eqref{equation:cleaninglemma-boundIiiel}.
	Recall the definition of $I^i_{\ell}$ in equation~\eqref{equ:def-Iiell} and that $(I_1, \dotsc, I_k)$ is the $(\eps^{1/k})$-gradation of $I$.
	If $i = 0$, then Proposition~\ref{proposition:cleaninglemma-sizes} implies $$|I^0_{\ell}| \leq e(I_{\ell}) \leq \eps^{\ell/k} \binom{t}{\ell} \leq \eps^{1/k} k^{2k} t^{\ell}.$$
	If $i > 0$, then every $Z \in I^i_{\ell}$ intersects with $Y$ in a set $Y' = Z \cap Y$ of size $i$.
	Note that $Y'$ is not in $I_i$.
	(Otherwise, it follows from $Y' \subset Y$ and the definition of the $d$-degree perturbation that all edges containing $Y$ are in $F_i$.
	But this contradicts $Y \in \partial_j(R')$ as $R'$ contains no edges of $F$.)
	Thus by Proposition~\ref{proposition:cleaninglemma-degrees} with $\ell$ playing the role of $j$, we obtain $\deg_{I_\ell}(Y') \leq (\ell - i) \eps^{1/k} t^{\ell-i}$.
	Moreover, there are $\binom{j}{i}$ possible choices for $Y'$.
	This gives inequality~\eqref{equation:cleaninglemma-boundIiiel} as
	$$|I^i_{\ell}|   \leq \binom{j}{i} (\ell - i) \eps^{1/k} t^{\ell - i} \leq \eps^{1/k} k^{2k} t^{\ell - i},$$
	where we crudely bounded $i, j, \ell \leq k$.
\end{proof}

The following proposition states that every $j$-set $Y$ outside of ${C_j}$ has large relative degree in $R'$.

\begin{proposition} \label{claim:cleaninglemma-notspoiled-largedegree}
	Let $1 \leq d \leq k-1$ and $0 < \eps,\beta \ll \alpha  \ll \delta$.
	Let $R$ be a $k$-graph  with a vertex spanning subgraph $I$ of edge density at most $\eps$.
	Denote by $F$ the $d$-degree $(\eps^{1/k})$-perturbation of $I$ in $R$.
	Suppose that $F$ has edge density at most $\beta$.
	Define $R' = R - (F \cup I)$.
	Let $(C_1, \dotsc, C_k)$ be the $(\beta^{1/k})$-gradation of $F$.
	Then $\reldeg_{R'}(Y) \ge \delta - \alpha$ for any $j$-set $Y$ outside of ${C_j}$ and $1 \leq j \leq d$.
\end{proposition}
\begin{proof}
	Suppose that $R$ has $t$ vertices.
	Recall that $F =C_k$ by definition of the gradation.
	We apply Proposition~\ref{proposition:cleaninglemma-degrees} to a $j$-set $Y$ outside of ${C_j}$ with $(j,k,F,\beta^{1/k})$ playing the role of $(i,j,I,\beta)$.
	Hence
	\begin{align}
	\deg_F(Y) = \deg_{C_k}(Y) \leq (k - j) \beta^{1/k} t^{k - j} \leq \frac{\alpha}{2} \binom{t-j}{k-j}, \label{equation:cleaninglemma-YinF}
	\end{align}
	where the last inequality follows from $\beta \ll \alpha$.
	
	We claim that $Y \notin I_j$.
	For sake of contradiction, suppose that $Y \in I_j$.
	Since $F$ is the $d$-degree $(\eps^{1/k})$-perturbation of $I$ in $R$, it follows that every edge of $R$ containing $Y$ belongs to $F$.
	Thus we obtain $$\deg_{F}(Y) = \deg_{R}(Y) \ge \delta \binom{t-j}{k-j} > (\alpha/2) \binom{t-j}{k-j},$$ where the last inequality follows from $\alpha \ll \delta$.
	But this contradicts inequality~\eqref{equation:cleaninglemma-YinF}, showing the above claim.
	
	We apply Proposition~\ref{proposition:cleaninglemma-degrees} to $Y \notin I_j$ with $(j,k,\eps^{1/k})$ playing the role of $(i,j,\beta)$.
	This gives
	\[ \deg_I(Y) = \deg_{I_k}(Y) \leq (k-j) \eps^{1/k} t^{k-j} \leq \frac{\alpha}{2} \binom{t-j}{k-j}, \]
	again using $\eps \ll \alpha$.
	So, together with inequality~\eqref{equation:cleaninglemma-YinF}, we have
	\begin{align*}
	\deg_{R'}(Y)
	\ge \deg_{R}(Y) - \deg_{I}(Y) - \deg_{F}(Y)
	\ge (\delta - \alpha) \binom{t-d}{k-d},
	\end{align*}
	as needed.
\end{proof}

Finally, the proof of Lemma~\ref{lem:degree-cleaning}.

\begin{proof}[Proof of Lemma~\ref{lem:degree-cleaning}]
	Let $(I_1, \dotsc, I_k)$ be the $(\eps^{1/k})$-gradation of $I$ and $F = \bigcup_{j=1}^d F_j$ the {$d$-degree $(\eps^{1/k})$-perturbation} of $I$ in $R$.
	Define $R' = R - (F \cup I)$.
	Thus $R'$ is a vertex spanning subgraph of $R$ which avoids $I$.
	To finish, we prove that $R'$ has $\alpha$-perturbed minimum relative $d$-degree at least $\delta-\alpha$.
	
	Let $\beta = \eps^{1/(2k)}$ and note that $2^k \eps^{1/k} \leq \beta$ as $\eps \ll 1/k$.
	Thus the {$d$-degree perturbation} $F$ has edge density at most $\beta$ by Lemma~\ref{prop:surface-damage}.
	Let $(C_1, \dotsc, C_k)$ be the $(\beta^{1/k})$-gradation of $F$.
	By Proposition~\ref{claim:cleaninglemma-largedegree-notspoiled}, it follows that $\partial_j(R')$ and ${C_j}$ are disjoint for every $1 \leq j \leq d$.
	Moreover, by Proposition~\ref{claim:cleaninglemma-notspoiled-largedegree}, we have $\reldeg_{R'}(Y) \ge \delta - \alpha$ for any  $j$-set $Y$ outside of ${C_j}$ and $1 \leq j \leq d$.
	Together, this implies that, for any $1 \leq j \leq d$ and set of $j$ vertices $Y$, we have the equivalence
	\begin{align*}
	Y \notin C_j\quad  \Rightarrow \quad   \reldeg_{R'}(Y) \ge \delta - \alpha \quad \Rightarrow \quad Y \in \partial_j(R') \quad \Rightarrow \quad 	Y \notin C_j.
	\end{align*}
	
	In other words, each $C_j$ contains precisely the sets of $j$ vertices with relative degree less than $\delta - \alpha$ in $R'$.
	This yields~\ref{itm:perturbed-degree-density} of Definition~\ref{def:perturbed-degree}.
	By Proposition~\ref{proposition:cleaninglemma-sizes} and $\eps \ll \alpha$, it follows that $C_d$ has edge density at most $\eps^{d/(2k)} \leq \alpha $. 
	This gives~\ref{itm:perturbed-degree-shadow-almost-complete}.
	Finally, for $1 \leq j \leq d-1$, any edge outside of ${C_j}$ has relative degree less than $\eps^{1/(2k)} \leq \alpha$ in $C_{j+1}$	as $(C_1, \dotsc, C_k)$ is an $(\beta^{1/k})$-gradation.
	This implies~\ref{itm:perturbed-degree-lower-levels}.
	Hence, $R$ has $\alpha$-perturbed minimum relative $d$-degree at least $\delta-\alpha$.
\end{proof}

\section{How to obtain Hamilton vicinities} \label{sec:finding-vicinities}
In this section, we determine the $d$-vicinity threshold of $k$-graphs for $d=k-2$ (Lemma~\ref{lem:vicinity-threshold-general}) and provide a general upper bound (Lemma~\ref{lem:vicinity-threshold-d-k=2}).
Given a $k$-graph $R$ of large perturbed minimum degree, our goal is to find a Hamilton vicinity $\cC$ (Definition~\ref{def:vicinity-Hamilton}).
In other words, we have to select a subgraph $C_S$ of every link graph $L(S)$ with $S \in \partial_d(R)$, such that collectively, the $C_S$'s satisfy the vicinity properties~\ref{def:vicinity-inner-connectivity}--\ref{def:vicinity-density}.
In the following, we will show a series of propositions which prepare us for this task.

We begin with a few comments on notation.
Recall that a $j$-set is a set of $j$ elements, typically vertices. 
Moreover, for $1 \leq j \leq k$ and a $k$-graph $R$, we defined the $j$-th (level) shadow graph $\partial_j(R)$ to be the $j$-graph on $V(R)$ whose edges are the $j$-sets contained in the edges of $R$.
The edge size of the $j$-th shadow graph is denoted by  $e_j(H)=e(\partial_j(R))$.
To distinguish between degrees of different subgraphs, we will denote the degree of a vertex set $S$ in $R$ by $\deg_R(S)$.

We will use (a consequence of) the Kruskal--Katona theorem to bound the size of the different shadow graphs of a hypergraph.
Lovász's formulation of the Kruskal--Katona theorem states that, for any $x > 0$, if $G$ is a $k$-graph with $e(G)\geq\binom{x}{k}$ edges, then $e_j(G) \ge \binom{x}{j}$, for every $1 \leq j \leq k$ (see, for instance, Theorem~2.14 and Remark~2.17 in~\cite{FranklTokushige2018}).
By approximating the binomial coefficients, we deduce:

\begin{theorem}[Kruskal--Katona theorem]\label{thm:KK}
	For $1\leq d\leq k-1$ and $1/t \ll \eps \ll 1/k $, let $G$ be a $k$-graph on $t$ vertices and edge density $\delta$.
	Then $\partial_d(G)$ has edge density at least $\delta^{d/k}-\eps$.
\end{theorem}

\subsection{Dense and spacious tight components}\label{sec:dense-and-spacious-components}

The following proposition was shown by Allen, Böttcher, Cooley and Mycroft~\cite[Proposition 10]{ABCM17}.
We include its short proof for the sake of completeness.

\begin{proposition} \label{prop:large-component}
	Any $\ell$-graph $L$ has a tight component $C$ such that $$\frac{e_\ell(C)}{e_{\ell-1}(C)}  \ge \frac{e_\ell(L)}{e_{\ell-1}(L)}.$$
	In particular, if $\delta, \nu, \nu'$ denote the edge densities of $L, C$ and $\partial(C)$ respectively, then $\nu /\nu'\ge \delta $.
\end{proposition}
\begin{proof}
	Let $C_1,\dots,C_s$ be the tight components of $L$.
	Fix $1 \leq q \leq s$ which maximises $e_{\ell}(C_q) / e_{\ell-1}(C_q)$.
	By the definition of tight components, we have $\sum_{i=1}^s e_\ell(C_i) = e_\ell(L)$ and $\sum_{i=1}^s e_{\ell-1}(C_i) = e_{\ell-1}(L)$.
	It follows that
	\begin{align*}
	e_\ell(L) = \sum_{i=1}^s e_\ell(C_i) \leq \frac{e_\ell(C_q)}{e_{\ell-1}(C_q)} \sum_{i=1}^s e_{\ell-1}(C_i) = \frac{e_\ell(C_q)}{e_{\ell-1}(C_q)}  e_{\ell-1}(L).
	\end{align*}
	We obtain the second part by bounding the edge density of $\partial_{\ell-1}(L)$ {by~$1$}.
\end{proof}

\begin{proposition}\label{prop:dense-subgraph}
	For $1/t \ll \eps \ll \delta$, let $L$ be an $\ell$-graph on $t$ vertices with a subgraph $C$.
	Denote the edge densities of $L$, $C$ and $\partial (C)$ by $\delta$, $\nu$ and $\nu'$ respectively.
	If $\nu /\nu'\geq  \delta$, then $\nu \geq \delta^\ell - \eps$.
\end{proposition}
\begin{proof}
	Choose $\eps' > 0$ such that $1/t \ll \eps' \ll \eps$.
	We have $\nu \geq \nu' \delta \geq (\nu^{(\ell-1)/\ell}-\eps') \delta$, where the last inequality follows from the Kruskal--Katona theorem (Theorem~\ref{thm:KK}).
	Solving this for $\nu$ (and using $\eps' \ll \eps$) gives the desired inequality.
\end{proof}
The next result is due to Frankl and bounds the size of a largest matching by the ratio of the number of edges and number of edges in the shadow graph.
\begin{lemma}[{\cite[Theorem 9.3]{FranklTokushige2018}}]\label{lem:large-matching}
	For $s \geq 1$, let $C$ be an $\ell$-graph with
	\begin{align*}
	e_\ell(C) \geq (s-1)e_{\ell-1}(C)+1.
	\end{align*}
	Then $C$ contains a matching with at least $s$ edges.
\end{lemma}

\subsection{Arcs and switchers}\label{sec:arcs-and-switchers}

Here we show how to obtain arcs~(Definition~\ref{def:switcher}) and switchers~(Definition~\ref{def:arc}) which are required for the divisibility condition~\ref{def:vicinity-divisibility} of Hamilton vicinities (Definition~\ref{def:vicinity-Hamilton}).

%\defswitcher*

The proof of the following result uses ideas from Reiher, Rödl, Ruciński, Schacht and Szemerédi~\cite[Lemma 6.2]{RRR19}.

\begin{proposition}\label{prop:switcher}
	For $1/t \ll \mu$, let $L$ be an $\ell$-graph on $t$ vertices with edge density at least $\delta +\mu$, where $\delta = \ell/(\ell + 2 \sqrt{\ell-1})$.
	Consider a vertex spanning subgraph $C\subset L$ and denote the edge densities of $C$ and $\partial (C)$ by $\nu$ and $\nu'$ respectively.
	Suppose that $\nu/ \nu' \geq \delta +\mu $.
	Then $C$ has a switcher.
\end{proposition}

\begin{proof}
	We abbreviate $d(X)=\deg_C(X)$ for sets of vertices $X\subset V(L)$.
	For any edge $A = \{a_1, \dotsc, a_\ell \}$ in $C$, define \[ f(A) = \sum_{i = 1}^\ell \frac{t}{d(A \setminus \{ a_i \})}. \]
	By double counting, the assumption $\nu \geq \nu' (\delta +\mu)$ and $1/t \ll \mu$, it follows that
	\begin{align}
	\sum_{e \in C} f(e) =  t e_{\ell-1}(C) =  \nu' t\binom{t}{\ell-1} \leq  \frac{\nu}{\delta +\mu}  t\binom{t}{\ell-1} <  \frac{\ell}{\delta}  e_\ell(C). \label{eq:shadow}
	\end{align}
	
	Say that an edge $A$ of $C$ is \emph{zentral}, if $f(A) < \ell + 2 \sqrt{\ell-1}$ and let $Z \subset E(C)$ be the set of zentral edges.
	Then inequality~\eqref{eq:shadow} implies $(\ell + 2 \sqrt{\ell-1}) e(C \setminus Z) < \ell \delta^{-1} e(C)$.
	Together with $\delta = \ell/(\ell + 2 \sqrt{\ell-1})$ this gives 
	\[ e(C \setminus Z) <   \frac{1}{\ell + 2 \sqrt{\ell-1}}  \frac{\ell}{\delta} e(C) = e(C). \]
	Hence there is at least one zentral edge $A = \{ a_1, \dotsc, a_\ell \}$.
	
	We can assume that $d(A \setminus \{ a_1 \}) \ge  \dotsb \ge d(A \setminus \{ a_\ell \}).$
	Since $A$ is zentral, the inequalities imply that
	\begin{align}
	\ell + 2 \sqrt{\ell-1} > \sum_{i=1}^\ell \frac{t}{d(A \setminus \{ a_i \})} \ge \left( \frac{\ell-1}{d(A \setminus \{ a_1 \})}+ \frac{1}{d(A \setminus \{ a_\ell \})}\right) t. \label{eq:usingcentral}
	\end{align}
	Then, by the Cauchy-Schwarz inequality,
	\begin{align*}
	\left( \frac{\ell-1}{d(A \setminus \{ a_1 \})}+ \frac{1}{d(A \setminus \{ a_\ell \})} \right)(d(A \setminus \{ a_1 \}) + d(A \setminus \{ a_\ell \}))  
	\ge (\sqrt{\ell-1} + 1)^2 = \ell + 2\sqrt{\ell-1}.
	\end{align*}
	Together with inequality~\eqref{eq:usingcentral}, we get \[ d(A \setminus \{ a_1 \}) + d(A \setminus \{ a_\ell \}) > t. \]
	This implies that $A$ is a switcher with central vertex $a_1$.
	Indeed, for every $1 \leq j  \leq \ell$,
	\[ d(A \setminus \{ a_1 \}) + d(A \setminus \{ a_j \}) \ge d(A \setminus \{ a_1 \}) + d(A \setminus \{ a_\ell \}) > t \] and therefore,
	\[ |N_C(A \setminus \{ a_1 \}) \cap N_C(A \setminus \{ a_j \}) | > 0. \qedhere \]
\end{proof}

Next, we show that a sufficiently dense hypergraph contains an arc (Definition~\ref{def:arc}).
%\defarc*

\begin{proposition}\label{prop:arc}
	Let $1\leq d \leq k-1$, $t \in \NATS$ and $\gamma, \delta >0$ with $1/t \ll \eps \ll \delta$ and $\delta+\delta^{1-1/(k-d)} > 1+\eps$.
	Let $R$ be a $k$-graph on $t$ vertices with a subgraph $H$ that is generated by a $d$-vicinity $\cC$.
	Suppose that each $C_S \in \cC$ has edge density at least $\delta+\mu$. 
	Then $\cC$ admits an arc.
\end{proposition}
\begin{proof}
	Consider an arbitrary set $S=\{v_{1},\dots,v_{d}\}\in \partial_d(R)$.
	By averaging, there is a vertex $v_{d+1}$ with relative vertex degree at least $\delta$ in $C_S$.
	Set $S'=\{v_{2},\dots,v_{d+1}\}$ and note that $S' \in \partial_d(R)$.
	By assumption and the Kruskal--Katona theorem (Theorem~\ref{thm:KK}) that the $(k-d-1)$-graph $\partial{(C_{S'})}$ has edge density at least $\delta^{1-1/(k-d)}-\eps$.
	
	By choice of $v_{d+1}$ and the pigeonhole principle, the $(k-d-1)$-graphs  $\partial_{k-d-1}({C_{S'}})$ and $L(\{v_{1},\dots,v_{d+1}\})$  must share a set $\{v_{d+2},\dots,v_{k}\}$.
	Since $\{v_{d+2},\dots,v_{k}\}$ is in $\partial_{k-d-1}({C_{S'}})$ there is a another vertex $v_{k+1}$ such that $\{v_{d+2},\dots,v_{k+1}\}$ is in $C_{S'}$.
	It follows that $(v_1,\dots,v_{k+1})$ is an arc.
\end{proof}

\subsection{Upper bounds for the vicinity threshold}\label{sec:vicinity-threshold-general}
 In this subsection, we show that $\hv_d(k) \leq {2^{-1/(k-d)}}$ for all $1 \leq d \leq k-1$.
Recall the definitions of a Hamilton Vicinity (Definition~\ref{def:vicinity-Hamilton}) and the Hamilton Vicinity threshold (Definition~\ref{def:vicinity-Hamilton}).
%\defvicinity*
%\defvicinitythreshold*
%\lemvicinitythresholdgeneral*

\begin{proof}[Proof of Lemma~\ref{lem:vicinity-threshold-general}]
	For $1\leq d \leq k-1$, let  $\ell =k-d$ and $\delta={{2^{-1/\ell}}}$.
	Let $t \in \NATS$ and $\alpha,\gamma,\mu >0$ with
	$$\frac{1}{t} \ll \alpha \ll\gamma \ll \mu \ll \delta.$$
	Consider a $k$-graph $R$ on $t$ vertices with $\alpha$-perturbed minimum relative $d$-degree at least $\delta+\mu$.
	Proposition~\ref{prop:large-component} allows us to select, for every $S \in \partial_d(R)$, a tight component $C_S \subset L(S)$ such that  
	\begin{align}\label{equ:vicinity-threshold-general-large-component}
	\nu_S/\nu_S'\geq  \delta +\mu,
	\end{align}
	where $\nu_S$ and $\nu_S'$ denote the edge densities of $C_S$ and $\partial(C_S)$, respectively.
	We claim that the $d$-vicinity $\cC:=(C_S)_{S \in \partial_d(R)}$ is $(\gamma, \delta)$-Hamilton.
	Our task now is to verify that $\cC$ satisfies properties~\ref{def:vicinity-inner-connectivity}--\ref{def:vicinity-density}.
	
	We start with the connectivity and outreach conditions.
	Note that $\cC$ has property~\ref{def:vicinity-inner-connectivity} by definition.
	Proposition~\ref{prop:dense-subgraph} together with inequality~\eqref{equ:vicinity-threshold-general-large-component} imply that
	\begin{align}\label{equ:vicinity-threshold-general-dense-component}
	\nu_S \geq \delta^{\ell} + \frac{\mu}{2}.
	\end{align}
	{Note that $\delta^\ell = 1/2$.}
	Together with inequality~\eqref{equ:vicinity-threshold-general-dense-component}, this implies that each $C_S \in \cC$ has edge density greater than $1/2 + \mu/2 \ge 1/(2(k-d)) + \gamma$.
	Hence, the vicinity $\cC$ has property~\ref{def:vicinity-outer-connectivity} and~\ref{def:vicinity-density}.
	
	Next, we verify the divisibility condition.
	{We begin by showing that every element of $\cC$ contains a switcher.
	For $k-d=1$, this is trivially implied by $\delta>0$.
	Now assume that $k-d\geq 2$.
	Since $(1+\frac{1}{\ell})^\ell \geq 2$, we have $2^{-1/\ell} \geq \ell/(\ell + 1)$.
	Together, this gives
	$$\delta = 2^{-1/\ell} \geq \frac{\ell}{\ell + 1} \geq \frac{\ell}{\ell + 2 \sqrt{\ell-1}}.$$}
	Thus, by Proposition~\ref{prop:switcher} and inequality~\eqref{equ:vicinity-threshold-general-large-component}, every element of $\cC$ contains a switcher.
	Next, we show that $\cC$ has an arc.
	Moreover, as {$\delta \geq \delta^\ell \geq 1/2$}, we have ${\delta}+({\delta})^{1-1/\ell}>1+c_{\ell}$ for a constant $c_{\ell}$ depending only on $\ell$.
	Since $\ell =k-d$, it follows by Proposition~\ref{prop:arc} that $\cC$ has indeed an arc.
	All together, this yields that $\cC$ has property~\ref{def:vicinity-space}.
	
	Finally, we show that $\cC$ satisfies the space condition.
	It follows that, for every $C_S \in \cC$,
	\begin{align*}
	\frac{e_\ell(C_S)}{e_{\ell-1}(C_S)} &= \frac{\nu_S}{\nu_S'} \frac{\binom{t-d}{\ell}}{\binom{t-d+1}{\ell-1}}
	\geq (\delta+\mu) \frac{\binom{t-d}{\ell}}{\binom{t-d+1}{\ell-1}} 
	\geq {\left(\frac{1}{2^{1/\ell} \ell }+\frac{\gamma}{\ell}\right) t \geq \left(\frac{1}{\ell+1}+\gamma\right)t \geq \left(\frac{1}{k}+\gamma\right)t,}
	\end{align*}
	{where we used in the last inequality that $2^{1/\ell} \leq (\ell+1)/\ell$ and $\mu \ll 1/\ell$.}
	Therefore, Proposition~\ref{lem:large-matching} implies that, every $C_S \in \cC$ has a matching of size at least $(\frac{1}{k}+\gamma)t$.
	Recall that the density of a fractional matching in $C$ is the sum of its edge weights divided by $t$.
	Since in a matching, every edge has weight one, we obtain that $\cC$ has property~\ref{def:vicinity-space}.
	This concludes the proof.
\end{proof}

\subsection{Optimal bounds for the vicinity threshold}\label{sec:vicinity-threshold-specific}
Finally, we determine the minimum $d$-degree threshold of $k$-uniform Hamilton vicinities when $d=k-2$ (Lemma~\ref{lem:vicinity-threshold-d-k=2}).
The proof relies on the following lemma, which was shown by Cooley and Mycroft~\cite[Lemma 3.2]{CM17}.
We state it in a slightly stronger form and with the additional property of containing a switcher.
\begin{lemma}  \label{lem:cooley-mycroft}
	Let $1/t \ll \gamma \ll \mu$.
	Suppose that $L_1$ and $L_2$ are $2$-graphs on a common vertex set of size $t$ such that $L_1,L_2$ have edge density at least $5/9 + \mu$ each.
	For $i =1,2$, let $C_i$ be a component in $L_i$ with a maximum number of edges.
	Then the following statements hold for each $i=1,2$:
	\begin{enumerate}[label=\textnormal{(\roman*)}]
		\item \label{itm:cooleymycroft-edge-in-common} $C_1$ and $C_2$ have an edge in common,
		\item \label{itm:cooleymycroft-triangle} $C_i$ has a switcher,
		\item \label{itm:cooleymycroft-fractional-matching} $C_i$ has a fractional matching of density $1/3+\gamma$ and
		\item \label{itm:cooleymycroft-density} $C_i$ has edge density at least $4/9 + \gamma$.
	\end{enumerate}
\end{lemma}

Before we prove Lemma~\ref{lem:cooley-mycroft}, let us show how this implies  $\hv_{k-2}(k) \leq  5/9$ for $k \geq 3$.

\begin{proof}[Proof of Lemma~\ref{lem:vicinity-threshold-d-k=2}]
	Let $k,t\geq 3$, $d=k-2$ and $\alpha,\gamma,\mu >0$ with
	\[\frac{1}{t} \ll \alpha \ll\gamma \ll \mu \ll \frac{5}{9}.\]
	Consider a $k$-graph $R$ on $t$ vertices with $\alpha$-perturbed minimum relative $d$-degree at least $5/9+\mu$.
	For every $S \in \partial_d(R)$, let $C_S$ be a tight component of $L(S)$ with  a maximum number of edges.
	We claim that the $d$-vicinity $\cC:=(C_S)_{S \in \partial_d(R)}$ is $\gamma$-Hamilton.
	Indeed, by definition $\cC$ has property~\ref{def:vicinity-inner-connectivity}.
	By Lemma~\ref{lem:cooley-mycroft}, $\cC$ also satisfies properties~\ref{def:vicinity-outer-connectivity},~\ref{def:vicinity-space} and~\ref{def:vicinity-density}.
	Moreover, every $C_S \in \cC$ contains a switcher.
	Since $4/9+(4/9)^{1-1/2}=1+1/9$, it follows by Proposition~\ref{prop:arc} that $\cC$ has an arc.
	Together, we obtain that $\cC$ has property~\ref{def:vicinity-space}.
\end{proof}

It remains to show Lemma~\ref{lem:cooley-mycroft}.
The proof follows along the lines of the original lemma by Cooley and Mycroft.
For sake of completeness, we give the details below.
We use the following theorem of Erd\H{o}s and Gallai~\cite{EG}, which gives a tight sharp bound on the smallest possible size of a maximum matching in a graph of given order and size.

\begin{theorem}[Erd\H{o}s and Gallai~\cite{EG}]\label{erdosgallai}
	Let $n,s \in \NATS$ with $s \leq n/2$.
	Suppose $G$ is a graph on $n$ vertices and 
	\[e(G) > \max \left\{\binom{2s-1}{2}, \binom{n}{2} - \binom{n-s+1}{2} \right\}.\]
	Then $G$ admits a matching with at least $s$ edges.
\end{theorem}

\begin{proof}[Proof of Lemma~\ref{lem:cooley-mycroft}]	
	During the proof we make repeated use of the fact that for $0 < x < 1$ we have $\binom{xt}{2} < x^2\binom{t}{2}$.
	We will show the following series of claims about $C_1$ and $C_2$, which in particular easily imply all of the results stated in the lemma:
	\begin{enumerate}[(a)]
		\item \label{item:cooleymycroft-a} both $C_1$ and $C_2$ span at least $(2/3 + \sqrt{\mu/2})t$ vertices,
		\item \label{item:cooleymycroft-b} both $C_1$ and $C_2$ have at least $(4/9 + \mu) \binom{t}{2}$ edges,
		\item \label{item:cooleymycroft-c} both $C_1$ and $C_2$ have a matching of size at least $(2/3 + \mu)t$,
		\item \label{item:cooleymycroft-d} $C_1$ and $C_2$ have an edge in common and
		\item \label{item:cooleymycroft-e} both $C_1$ and $C_2$ have a switcher.
	\end{enumerate}
	
	Let us prove part~\ref{item:cooleymycroft-a}.
	Let $1 \leq i \leq 2$ and suppose for a contradiction that every connected component in $L_i$
	has at most $(2/3+\sqrt{\mu/2})t$ vertices.
	Then we may form disjoint sets $A$ and $B$ such
	that $V(L_i) = A \cup B$, such that $A$ and $B$ can each be written as a union of
	connected components of $L_i$, and such that $|A|, |B| \leq (2/3+\sqrt{\mu/2})t$. We then have
	\[ e(L_i) \leq \binom{|A|}{2} + \binom{|B|}{2} \leq \binom{\left(\frac{2}{3} + \sqrt{\frac{\mu}{2}} \right)t}{2} + \binom{\left(\frac{1}{3} - \sqrt{\frac{\mu}{2}} \right)t}{2} < \left(\frac{5}{9} + \mu \right) \binom{t}{2},\]
	giving a contradiction.
	
	Now we prove part~\ref{item:cooleymycroft-b}.
	Indeed, for any $1 \leq i \leq 2$, part~\ref{item:cooleymycroft-a} implies that the number of edges of $L_i$ which are not in $C_i$ is at most
	\[\binom{t-v(C_i)}{2} < \binom{\left(\frac{1}{3}-\sqrt{\frac{\mu}{2}}\right)t}{2} < \left(\frac{1}{9}+{\frac{\mu}{2}}\right)\binom{t}{2}.\]
	
	Now we check that part~\ref{item:cooleymycroft-c} holds.
	Let $1 \leq i \leq 2$ be arbitrary, and let $x$ be such that $v(C_i) = (1-x)t$, so $x$ is the proportion of vertices of $V$
	which are not in $C_i$. In particular $0 \leq x < 1/3-\sqrt{\mu/2}$ by part~\ref{item:cooleymycroft-a}.
	Observe that at most $\binom{xt}{2} \leq x^2\binom{t}{2}$ edges of $L_i$ are not in $C_i$, so $C_i$ has more than $\left(5/9 - x^2 \right) \binom{t}{2}$ edges.
	It is easily checked that the inequality 
	\[\left(\frac{5}{9} - x^2 \right) \binom{t}{2} > \binom{2s -1}{2} = \max \left\{\binom{2 s-1}{2},  \binom{t'}{2} - \binom{t'-s+1}{2} \right\}\]
	holds for $t' =\lceil(1-x)t\rceil$, $s = \lceil(1/3+\gamma)t\rceil$ any $0 \leq x < 1/3-\sqrt{\mu/2}$ as $1/t \ll \gamma \ll \mu$.
	So by Theorem~\ref{erdosgallai} the component $C_i$ admits a matching $M$ of size $(2/3+\mu)t$ as claimed.	
	
	To see part~\ref{item:cooleymycroft-d}, fix $\alpha$ and $\beta$ so that $|V(C_1)| = (1-\alpha)t$ and
	$|V(C_2)| = (1 - \beta) t$. By part~\ref{item:cooleymycroft-a}, $0 \leq \alpha, \beta < 1/3$ and
	$|V(C_1) \cap V(C_2)| \geq (1 - \alpha - \beta) t$. Similarly as before, at most
	$\binom{\alpha t}{2} \leq \alpha^2\binom{t}{2}$ edges of $L_1$ are not in $C_1$, and at most
	$\binom{\beta t}{2} \leq \beta^2\binom{t}{2}$ edges of $L_2$ are not in $C_2$. Now suppose for a
	contradiction that $C_1$ and $C_2$ have no edges in common. Then we have
	\begin{align*}
	\left(\frac{5}{9}-\alpha^2\right)\binom{t}{2} & + \left(\frac{5}{9} - \beta^2\right) \binom{t}{2}
	< e(C_1) + e(C_2) \\
	& \leq \binom{v(C_1)}{2} + \binom{v(C_2)}{2} - \binom{|V(C_1) \cap V(C_2)|}{2} \\ 
	& \leq \binom{(1-\alpha) t}{2} + \binom{(1-\beta) t}{2} - \binom{(1- \alpha - \beta)t}{2} \\
	& = \left((1-\alpha)^2 + (1-\beta)^2 - (1-\alpha-\beta)^2\right)\binom{t}{2}\\
	&\hspace{1cm} - \tfrac{t}{2}\left(\alpha(1-\alpha)+\beta(1-\beta)-(\alpha+\beta)(1-\alpha - \beta)\right)\\
	& \leq (1-2\alpha\beta) \binom{t}{2}.
	\end{align*}
	So $1/9 < \alpha^2 + \beta^2 - 2 \alpha \beta = (\alpha - \beta)^2$, which implies that $|\alpha - \beta| > 1/3$, contradicting the fact that $0 \leq \alpha, \beta < 1/3$. We deduce that $C_1$ and $C_2$ must have an edge in common, as required.
	
	We conclude the proof by showing part~\ref{item:cooleymycroft-e}.
	For reasons of symmetry, it suffices to  show that $C_1$ has a switcher.
	By Proposition~\ref{prop:large-component}, $L_1$ contains a component $C$ such that $\nu /\nu'\geq  5/9+\mu$, where $\nu$ denotes the edge density of $C$ and $\nu'$ denotes the edge density of $\partial_1(C)$.
	Since $5/9 \geq 2/(2+2\sqrt{2-1})$, it follows by Proposition~\ref{prop:switcher} that $C$ contains a switcher.
	Moreover, it follows by Proposition~\ref{prop:dense-subgraph}, that $\nu \geq (5/9+\mu/2)^2$.
	So $C$ has at least $(5/9)t$ vertices.
	Therefore we have $C=C_1$.
\end{proof}

\section{From framework to Hamilton cycles -- an overview}\label{sec:framework-threshold-overview}

In the second part of this paper, we show that the existence of Hamilton frameworks implies the existence of tight Hamilton cycles, namely Theorem~\ref{thm:framework}.
The proof is based on a combination of two well-known approaches, the \emph{absorption method} and the \emph{method of connected matchings}.
The modern absorption method was introduced by Rödl, Ruciński and Szemerédi~\cite{RRS06}. 
A precursor of this approach can be found in the works of Erdős, Pyber and Gyárfás~\cite{EGP91} and of Krivelevich~\cite{Krivelevich1997}.
The method of connected matchings is based on the work of Koml\'os, S\'ark\"ozy and Szemerédi~\cite{KSS97}, which itself relies on the regularity method of Szemerédi~\cite{Sze76}.
It was introduced by  \L uczak~\cite{Luc99} for graphs, generalised to $3$-graphs by Haxell, \L uczak, Peng, Rödl, Ruciński and Skokan~\cite{HLP+09} and later to all uniformities by Allen, Böttcher, Cooley und Mycroft~\cite{ABCM17}, using the hypergraph regularity lemma developed by Rödl and Skokan~\cite{RS04}, Rödl and Schacht~\cite{RS07} and Gowers~\cite{Gow07}.

In what follows, we outline the proof of  Theorem~\ref{thm:framework} for the case of 2-graphs.
More precisely, we will assume that the Hamilton framework threshold is bounded by $\hf_1(2) \leq 1/2$ and use this to prove that the (tight) Hamilton cycle threshold is at most $\hc_1(2) \leq 1/2$.
For convenience, let us recall the definition of Hamilton frameworks.

\defframework*

%\defframeworkthreshold*

\begin{proof}[Proof sketch of $\hc_1(2) \leq 1/2$, assuming $\hf_1(2) \leq 1/2$]
	Let $1/n \ll 1/t \ll \eps \ll \alpha \ll \gamma \ll \mu$.
	Suppose that $G$ is a $2$-graph on $n$ vertices and with minimum relative degree at least $1/2 + \mu$.
	The goal is to show that $G$ has a Hamilton cycle.
	
	We start by setting up the regularity framework.
	Using a regularity lemma, we obtain an $\eps$-regular partition of $V(G)$ which yields a \emph{reduced graph} $R$ on $t$ vertices that approximately captures the edge structure of $G$.
	By our assumption, $\hf_1(2) = 1/2$.
	Therefore, by definition of the framework threshold (Definition~\ref{def:framework-threshold}) the reduced graph $R$ contains a $(\alpha,\gamma,\delta)$-Hamilton framework $H$ satisfying conditions~\ref{def:framework-spanning}--\ref{def:framework-degree}.
	
	In the following, we use the properties of the Hamilton framework $H$ to construct a Hamilton cycle in $G$.
	We begin by setting up a special path.
	
	\begin{claim}[{Absorption Claim}]\label{cla:absorption}
		There is a path $P$ in $G$ such that for every set $S$ of even size and size at most $\eps n$, there is a path $P'$ on vertex set $V(P) \cup S$ with the same ends as $P$.
		Moreover, $P$ has size at most $\alpha n$ and extends in an appropriate way into $H$.
	\end{claim}

	The proof of the Absorption Claim relies on the outreach condition~\ref{def:framework-degree} of Hamilton frameworks.
	In the second step, we extend $P$ to a cycle $C$ that contains almost all vertices of $H$.
	
	\begin{claim}[Cover Claim]\label{cla:cover}
		There is a cycle $C$ which contains $P$ as a subpath and covers all but $\alpha n$ vertices of $G$.
		Moreover, the set $V(G) \sm V(C)$ of uncovered vertices has even size.
	\end{claim}

	The proof of the Cover Claim uses the method of connected matchings, which allows us to find an almost spanning cycle in $V(G)$ provided that $H$ is connected and contains a perfect fractional matching.
	Note that $H$ satisfies these properties by the connectivity condition~\ref{def:framework-connectivity} and the space condition~\ref{def:framework-space}.
	However, note that we also need to circumvent up to $\alpha n$ vertices occupied by $P$.
	At this point, it is crucial that $H$ is $\gamma$-robustly matchable and $\alpha \ll \gamma$.
	Moreover, in order to include $P$ has a subpath of $C$, we rely on $P$ extending appropriately into $H$.
	Finally, the divisibility condition~\ref{def:framework-divisibility} is used, to guarantee that $V(G) \sm V(C)$ has even size.
	
	We finish the construction of $C$ by absorbing the remaining (even and at most $\alpha n$) vertices into $C$.
	This results in the desired Hamilton cycle.
\end{proof}
Let us remark that our proof uses the regularity tools of Allen, Böttcher, Cooley and Mycroft~\cite{ABCM17}, which were developed for this type of application.
We  note furthermore that our approach benefits of two ideas adapted from our joint work with Garbe, Lo and Mycroft~\cite{GLL+21}.
Firstly, we separate the action of adjusting the parity of the number of leftover vertices from the construction of the absorbing path.
This way we can treat the two problems separately, which significantly simplifies the analysis.
Secondly, the notion of $\gamma$-robust matchability allows us to circumvent the absorbing path in the regularity setting.
(An alternative approach would be to select the absorbing paths \emph{before} applying a regularity lemma.
However, this comes with a lot of technical issues with regards to the connectivity properties of $P$, some of which we were ultimately not able to overcome.)

\section{Regularity}\label{sec:regularity}

The proof of the Framework Theorem (Theorem~\ref{thm:framework}) relies on the Regular Slice Lemma of Allen et al.~\cite{ABCM17}, which itself was derived from the Strong Hypergraph Regularity Lemma of R\"odl and Schacht~\cite{RS07}.
The setting of the Regular Slice Lemma allows for a notionally lighter application of hypergraph regularity and permits us to work with a reduced graph in (mostly) the same way as in the case of graph regularity.
Our exposition follows loosely that of Allen et al.~\cite{ABCM17}, with a few modifications  tailored to our uses.

In what remains, we will use the following usual notation.
Given reals $a, b, c$ with $c > 0$, we write $a = b \pm c$ to mean $b - c \leq a \leq b + c$.

\subsection{Hypergraph regularity} \label{subsec:reg}
A {hypergraph} $\Hy=(V,E)$ is a \emph{complex}, if its edge set is down-closed, meaning that whenever $e\in E$ and $e'\subseteq e$ we have $e'\in E$. 
All complexes considered here have the property that $\{v\} \in E$ for every $v \in V$.
A \emph{$k$-complex} is a complex in which all edges have cardinality at most $k$.
Given a complex $\Hy$, we use $\Hy^{(i)}$ to denote the $i$-graph obtained by taking all vertices of $\Hy$ and those edges of cardinality exactly $i$.
We write $e_i(\Hy)$ to mean $|E(\Hy^{(i)})|$, that is the number of edges of size $i$ in $\Hy$.

Let $\Part$ be a partition of a vertex set $V$.
We refer to the underlying set of $\cP$ as $V(\Part)=V$.
We say that a subset $S \subseteq V$ is \emph{$\Part$-partite}, 
if $|S \cap Z| \leq 1$ for every $Z \in \Part$.
Similarly, we say that a hypergraph $\Hy$ is \emph{$\Part$-partite}, if all of its edges are $\Part$-partite.
In this case we refer to the parts of $\Part$ as the \emph{vertex classes} of $\Hy$.
A hypergraph $\Hy$ is \emph{$s$-partite}, if there is a partition $\Part$ of $V(\Hy)$ into $s$ parts for which $\Hy$ is $\Part$-partite. 

Let $\Hy$ be a $\Part$-partite hypergraph.
For any $X \subseteq \Part$, we write $V_X$ for $\bigcup_{Z \in X} Z$.
Note that $X$ is a partition of $V_X$.
The \emph{index} of a $\Part$-partite set $S \subseteq V$ is $i(S) := \{Z \in \Part\colon |S \cap Z| = 1\}$.
We write $\Hy_X$ to denote the collection of edges in $\Hy$ with index $X$.
So $\Hy_X$ can be regarded as an $|X|$-partite $|X|$-graph on the vertex set $V_X$,
whose vertex classes are exactly those in $X$.
Equivalently, $\Hy_X$ is the induced subgraph of $\Hy^{(|X|)}$ in~$V_X$.
In a similar manner we write $\Hy_{X^<}$ for the $X$-partite hypergraph on vertex set $V_X$ whose edge set is $\bigcup_{X' \subsetneq X} \Hy_X$.
Observe that if $\Hy$ is a $\Part$-partite $k$-complex and $X \subseteq \Part$ is a $k$-set,
then $\Hy_{X^<}$ is an $X$-partite $(k-1)$-complex.

For $i \geq 2$, let $\Part_i$ be a partition of a vertex set $V$ into $i$ parts,
let $H_i$ be any $\Part_i$-partite $i$-graph
and let $H_{i-1}$ be any $\Part_i$-partite $(i-1)$-graph, on the common vertex set $V$.
We denote by $K_i(H_{i-1})$ the $\Part_i$-partite $i$-graph on $V$ whose edges are all $\Part_i$-partite $i$-sets in $V$ which are supported on $H_{i-1}$
(that is induce a copy of the complete $(i-1)$-graph~$K_i^{i-1}$ on~$i$ vertices in~$H_{i-1}$).
The \emph{density of $H_i$ with respect to $H_{i-1}$} is then defined to be
\[ d(H_i|H_{i-1}) = \frac{e(K_i(H_{i-1})\cap	H_i)}{e(K_i(H_{i-1}))}\] 
if $e(K_i(H_{i-1}))>0$.
For convenience we take $d(H_i|H_{i-1}) =0$, if ${e(K_i(H_{i-1})) = 0}$.
So $d(H_i|H_{i-1})$ is the proportion of $\Part_i$-partite copies of $K^{i-1}_i$ in $H_{i-1}$ which are also edges of $H_i$.
When $H_{i-1}$ is clear from the context, we simply refer to $d(H_i | H_{i-1})$ as the \emph{relative density of $H_i$}.

More generally, if $\Qb =(Q_1,Q_2,\ldots,Q_r)$ is a collection of~$r$ not necessarily disjoint subgraphs of~$H_{i-1}$,
we define $K_i(\Qb) =\bigcup_{j=1}^r K_i(Q_j)$ and
\[d(H_i |\Qb) = \frac{e(K_i(\Qb)\cap H_i)}{e(K_i(\Qb))},\]
if
$e(K_i(\Qb))>0$.
Similarly as before we take $d(H_i |\Qb) =0$, if $e(K_i(\Qb)) =0$.

We say that $H_i$ is \emph{$(d_i,\eps,r)$-regular with respect to~$H_{i-1}$}, if we have $d(H_i|\Qb) = d_i \pm \eps $ for every $r$-set $\Qb$ of subgraphs of $H_{i-1}$ such that $e(K_i(\Qb)) > \eps e(K_i(H_{i-1}))$.
We often refer to $(d_i,\eps,1)$-regularity simply as \emph{$(d_i,\eps)$-regularity};
also, we say simply that $H_i$ is \emph{$(\eps,r)$-regular with respect to~$H_{i-1}$} to mean that there exists a $d_i$ for which $H_i$ is \emph{$(d_i,\eps,r)$-regular with respect to~$H_{i-1}$}.
Given a (non necessarily $\Part$-partite) $i$-graph $G$ whose vertex set contains that of $H_{i-1}$,
we say that $G$ is \emph{$(d_i,\eps,r)$-regular with respect to~$H_{i-1}$},
if the $i$-partite subgraph of $G$ induced by the vertex classes of $H_{i-1}$ is $(d_i,\eps,r)$-regular with respect to $H_{i-1}$.
Similarly as before, when $H_{i-1}$ is clear from the context, we refer to the density of this $i$-partite subgraph of $G$ with respect to $H_{i-1}$ as the \emph{relative density of~$G$}.

\subsection{Equitable complexes and regular slices}\label{subsec:rhgraph} 
The Regular Slice Lemma says that any $k$-graph $G$ admits a regular slice (Definition~\ref{def:regular-slice}).
Informally speaking, a regular slice is a partite $(k-1)$-complex $\cJ$ whose vertex classes have equal size,
whose subgraphs $\cJ^{(2)}, \dotsc, \cJ^{(k-1)}$ satisfy certain regularity properties,
and which moreover has the property that $G$ is regular with respect to $\cJ^{(k-1)}$.
The first two of these conditions are formalised in the following definition:
we say that a $(k-1)$-complex $\cJ$ is \emph{$(t_0,t_1,\eps)$-equitable}, if it has the following properties.
\begin{enumerate}[(i)]
	\item $\cJ$ is $\Part$-partite for a $\Part$ which partitions $V(\cJ)$ into $t$ parts of equal size, where $t_0\le t\le t_1$.
	We refer to $\Part$ as the \emph{ground partition} of $\cJ$,
	and to the parts of $\Part$ as the	\emph{clusters} of $\cJ$.
	\item There exists a \emph{density vector} $\mathbf{d}=(d_2, \dotsc, d_{k-1})$
	such that for each $2\le i\le k-1$ we have $d_i\ge 1/t_1$ and $1/d_i\in\NATS$,
	and for each $A \subseteq \cP$ of size $i$, the $i$-graph $\cJ^{(i)}[V_A]$ induced on $V_A$ is $(d_i, \eps)$-regular with respect to $\cJ^{(i-1)}[V_A]$.
\end{enumerate}
If $\cJ$ has density vector $\mathbf{d} = (d_2, \dotsc, d_{k-1})$, then we will say that $\cJ$ is \emph{$(d_2, \dotsc, d_{k-1}, \eps)$}-regular, or \emph{$(\mathbf{d},\eps)$-regular}, for short.

For any $k$-set $X$ of clusters of $\cJ$,
we write $\hat{\cJ}_X$ for the $k$-partite $(k-1)$-graph $\cJ^{(k-1)}_{X^<}$.  
Given a $(t_0,t_1,\eps)$-equitable $(k-1)$-complex $\cJ$, a $k$-set $X$ of clusters of $\cJ$ and a $k$-graph $G$ on $V(\cJ)$,
we say that \emph{$G$ is $(d,\eps_k,r)$-regular with respect to $X$},
if $G$ is $(d,\eps_k,r)$-regular with respect to $\hat{\cJ_X}$.  
We will also say that \emph{$G$ is $(\eps_k,r)$-regular with respect to $X$}, if there exists a $d$ such that $G$ is $(d,\eps_k,r)$-regular with respect to $X$.
We write $\reld_\cJ(X)$ for the relative density of $G$ with respect to $\hat{\cJ_X}$,
or simply $\reld(X)$ if $\cJ$ is clear from the context, which will always be the case in applications. 

We can now present the definition of a regular slice.

\begin{definition}[Regular slice]\label{def:regular-slice}
	Given $\eps,\eps_k>0$, $r,t_0,t_1\in\NATS$, a $k$-graph $G$ and a $(k-1)$-complex $\cJ$ with $V(\cJ) \subset V(G)$,
	we call $\cJ$ a \emph{$(t_0,t_1,\eps,\eps_k,r)$-regular slice for $G$},
	if $\cJ$ is $(t_0,t_1,\eps)$-equitable
	and $G$ is $(\eps_k,r)$-regular with respect to all but at most	$\eps_k\binom{t}{k}$ of the $k$-sets of clusters of $\cJ$, where $t$ is the number of clusters of~$\cJ$.
\end{definition}

If we specify the density vector $\mathbf{d}$ and the number of clusters $t$ of an equitable complex or a regular slice, then it is not necessary to specify $t_0$ and $t_1$ (since the only role of these parameters is to bound $\mathbf{d}$ and $t$). In this situation we write that $\cJ$ is $(\cdot, \cdot, \eps)$-equitable, or is a $(\cdot,\cdot,\eps,\eps_k,r)$-regular slice for $G$.

\subsection{Representative degrees and rooted degrees}

A regular slice $\cJ$ for a $k$-graph $G$ does not necessarily resemble $G$ in the way a regular partition for graphs does (see~\cite[Section 4.3]{ABCM17}).
Because of this reason we will require that our regular slices have to satisfy additional properties with respect to~$\cJ$,
ensuring that $\cJ$ approximates $G$ faithfully and provides a suitable environment to apply the tools of the regularity method.
We will summarise these properties in a structure called regular setup (Definition~\ref{def:setup}).
In the following, we introduce some definitions to express those extra properties.

Given a regular slice $\cJ$ for a $k$-graph $G$, we use a weighted reduced graph to record the relative densities $\reld(X)$ for $k$-sets $X$ of clusters of $\cJ$ and (another) reduced graph to record $k$-sets which are regular as well as dense.

\begin{definition}[Weighted reduced $k$-graph] \label{definition:weightedreduced}
	Let $G$ be a $k$-graph and let $\cJ$ be a $(t_0,t_1,\eps,\eps_k,r)$-regular slice for $G$.
	We define the \emph{weighted reduced $k$-graph of $G$}, denoted $R(G)$, to be the complete weighted $k$-graph whose
	vertices are the clusters of $\cJ$, and where each edge $X$ is given weight $\reld(X)$
	(so in particular, the weight is in $[0,1]$).
	
	Similarly, for $d_k>0$, we define the \emph{$d_k$-reduced $k$-graph of $G$}, denoted $R_{d_k}(G)$, to be the subgraph of $R(G)$ containing only the edges $X$ for which $G$ is $(\eps_k, r)$-regular with respect to $X$ and $\reld(X) \geq d_k$. 
	Note that $R(G)$ and  $R_{d_k}(G)$ depend on $\cJ$, but this will always be clear from the context.
\end{definition}

Let us formalise what it means for a vertex to be well-represented by a regular slice.
Given a $k$-graph $G$ on $n$ vertices and a vertex $v \in V(G)$,
recall that $\deg_G(v)$ is equal to the number of edges of $G$ which contain $v$
and $\reldeg_G(v) = \deg_G(v)/\binom{n-1}{k-1}$ is the relative degree of $v$ in $G$.
This number has a probabilistic interpretation: if we select a $(k-1)$-set of vertices $T \subseteq V(G) \setminus \{v\}$ uniformly at random, then $\reldeg_G(v)$ is exactly the probability that $T \cup \{v\}$ is an edge of $G$.

Given a $(t_0, t_1, \eps)$-equitable $(k-1)$-complex $\cJ$ with $V(\cJ) \subset V(G)$,
the \emph{rooted degree of $v$ supported by~$\cJ$}, written $\deg_G(v;\cJ)$,
is defined as the number of $(k-1)$-sets $T$ in $\cJ^{(k-1)}$ such that $T \cup \{v\}$ form an edge in $G$
(note that we do not require the edge $T \cup \{v\}$ to be contained in $\cJ$).
Then the \emph{relative degree $\reldeg_G(v;\cJ)$ of $v$ in $G$ supported by $\cJ$} is defined as $\reldeg_G(v;\cJ) = \deg_G(v;\cJ) / e(\cJ^{(k-1)})$.
Again we have a natural probabilistic interpretation: if we select a $(k-1)$-set of vertices $T \subseteq V(G)$ uniformly at random, then $\reldeg_G(v;\cJ)$ is the probability that $T \cup \{v\}$ is an edge of $G$ after conditioning on the event that $T$ is a $(k-1)$-edge in $\cJ^{(k-1)}$.

\begin{definition}[Representative rooted degree] \label{definition:representativerooted}
	Let $\eta > 0$, let $G$ be a $k$-graph and let $\cJ$ be a $(t_0,t_1,\eps,\eps_k,r)$-regular slice for $G$.
	We say that $\cJ$ is \emph{$\eta$-rooted-degree-representative} if, for 
	any vertex $v \in V(G)$, we have
	\[\big|\reldeg_G(v;\cJ)-\reldeg_G(v) \big|<\eta. \]
\end{definition}

\subsection{Regular setups and the Regular Slice Lemma}

Now we are ready to present the regular setup, which we will use throughout the remainder of the paper.

\begin{definition}[Regular setup] \label{def:setup}
	Let $k,m,t, r \in \mathbb{N}$ and $\eps,\eps_k, d_2,\dots,d_k >0$.
	We say that $(G,G_{\cJ},\cJ,\cP,R)$ is a \emph{$(k,m,t, \eps , \eps_k, r , d_2,\dots,d_k)$-regular setup}, if 
	\begin{enumerate}[label=\textnormal{(S\arabic*)}]
		\item $G$ is a $k$-graph and $G_{\cJ} \subseteq G$,
		\item $\cJ$ is a $(\cdot,\cdot,\eps,\eps_k,r)$-regular slice for $G$ with density vector $\mathbf{d}=(d_2, \dotsc, d_{k-1})$,
		\item $\cP$ is the ground partition of $\cJ$ with $t$ clusters each of size $m$,
		\item $R$ is a subgraph of $R_{d_k}(G)$, the $d_k$-reduced graph of~$G$ and
		\item for each $X \in E(R)$, $G_{\cJ}$ is $(d_k,\eps_k, r)$-regular with respect to $X$.
	\end{enumerate}
	We further say that $(G,G_{\cJ},\cJ,\cP,R)$ is \emph{representative} if
	\begin{enumerate}[resume,label=\textnormal{(S\arabic*)}]
		\item \label{item:regsetup-rootedrepresent} $\cJ$ is $\eps_k$-rooted-degree-representative (Definition~\ref{definition:representativerooted}).
	\end{enumerate}
	Given $\mathbf{d} = (d_2, \dotsc, d_{k-1}, d_{k})$, we will simply write \emph{$(k,m,t,\eps, \eps_k, r, \mathbf{d})$-regular setup} to mean $(k,m,t, \eps , \eps_k, r , d_2,\dots,d_k)$-regular setup.
	If we wish to refer to a regular setup but a parameter is irrelevant for the given situation, then we will omit it from the notation.
	For instance, we will write $(k, m, \cdot, \eps, \eps_k, r, \mathbf{d})$-regular setups in situations where the total number of clusters is not important.
\end{definition}

The Regular Slice Lemma of Allen, Böttcher, Cooley and Mycroft~\cite{ABCM17} ensures that every sufficiently large $k$-graph has a representative regular slice.
Given the existence of a regular slice, it is simple to derive the existence of a regular setup.
The following lemma is therefore stated directly in terms of regular setups.
We explain in Appendix~\ref{appendix:slicesandsetup} how this can be derived from the original lemma.

\begin{lemma}[{Regular Setup Lemma~\cite{ABCM17}}] \label{lem:regular-setup}
	Let $k, d, t_0$ be positive integers,
	$\delta, \mu, \alpha, \eps_k, d_k$ be positive,
	and $r: \NATS \rightarrow \NATS$ and $\eps: \NATS \rightarrow (0,1]$ be functions.
	Suppose that
	\begin{align*}
		1 \leq d \leq k-1,\quad k\geq 3, \quad \text{and} \quad
		\eps_k \ll \alpha, d_k \ll \mu.
	\end{align*}
	Then there exist integers~$t_1$ and~$m_0$ such that the following holds for all $n \ge 2 t_1 m_0$.
	Let $G$ be a graph on $n$ vertices and suppose that $G$ has minimum relative $d$-degree $\overline{\delta}_d(G) \geq \delta +\mu$.
	Then there exists $\mathbf{d} = (d_2, \dotsc, d_{k-1}, d_k)$
	and a representative $(k, m, t, \eps(t_1), \eps_k, r(t_1), \mathbf{d})$-regular setup $(G, G_{\cJ}, \cJ, \cP, R_{d_k})$ with $t_0 \leq t \leq t_1$, $m_0 \leq m$ and $n \leq (1+\alpha)mt$.
	Moreover, there is a $k$-graph $I$ on $\cP$ of edge density at most $\eps_k$ such that $R = R_{d_k} \cup I$ has minimum relative $d$-degree at least $\delta +\mu/2$.
\end{lemma}

\subsection{Oriented setups}\label{sec:oriented-setups}
Finally, in our particular setting of embedding tight paths and cycles, it will be necessary to endow each edge of the reduced graph $R$ (from the definition of a regular setup) with an ordering of its elements.
This ordering will guide in which order we construct tight paths inside a given edge, and how we can join these paths in a coherent way to form longer paths and cycles.

More precisely, an \emph{orientation} of a $k$-graph $H$ is a family of ordered $k$-tuples $\{ \ori{e} \in V(H)^k\colon e \in H \}$, one for each edge in $H$, such that $\ori{e}$ consists of an ordering of the vertices of $e$.
We say that a family of ordered $k$-tuples $\ori{H}$ is an \emph{oriented $k$-graph} if there exists a $k$-graph $H$ such that $\ori{H} = \{ \ori{e} \in V(H)^k \colon e \in H \}$.
Given an oriented $k$-graph $\ori{R}$, we will say that $(G,G_{\cJ},\cJ,\cP,\ori{R})$ is an \emph{oriented $(k,m,t, \eps , \eps_k, r , \mathbf{d})$-regular setup} if $\ori{R}$ is an orientation of $R$ and $(G,G_{\cJ},\cJ,\cP,R)$ is a $(k,m,t, \eps , \eps_k, r , \mathbf{d})$-regular setup.
We will call an oriented setup representative, if the underlying (ordinary) setup is representative.

\section{Bounding the framework threshold -- the proof} \label{sec:framework-threshold-proof}

The goal of this section is to give a formal proof of the Framework Theorem (Theorem~\ref{thm:framework}).
In the next subsection, we introduce some further notations regarding extensible and absorbing paths.
In Subsection~\ref{sec:absorption-cover-lemma-statement}, we state the Absorption Lemma and the Cover Lemma (Lemma~\ref{lem:absorption} and~\ref{lem:cover}.)
Finally, in Subsection~\ref{sec:actual-proof-of-framework-theorem}, we prove Theorem~\ref{thm:framework}.

\subsection{Extensible and absorbing paths}
In the following, we formalise the ideas of absorption and extensibility outlined in Section~\ref{sec:framework-threshold-overview}.

The definition of an absorbing path is fairly straightforward.
Let $G$ be a $k$-graph on $n$ vertices, $S \subseteq V(G)$ and let $P \subseteq H$ be a tight path.
We say that $P$ is \emph{$S$-absorbing in $G$} if there exists a path $P'$ in $G$ with the same $k-1$ starting vertices of $P$, the same $k-1$ ending vertices of $P$, and $V(P') = V(P) \cup S$.
We say $P$ is \emph{$\eta$-absorbing in $G$}, if it is $S$-absorbing in $G$ for every $S$ of size at most $\eta n$ whose size is divisible by $k$ and with $S \cap V(P) \neq \es$.

As explained in the proof sketch, the absorption path $P$ should furthermore extend well into $H$.
In the following, we define an interface appropriate for this requirement.
Recall that given a $\cP$-partite complex $\cJ$ and a $j$-set $X \subseteq \cP$, the $j$-graph $\cJ_X$ corresponds to the $j$-sized edges in $\cJ^{(j)}$ which have index precisely $X$, meaning that each edge in $\cJ_X$ intersects each cluster of $X$ exactly once.
Given an ordering $(X_1, \dotsc, X_j)$ of the elements of $X$, let $\cJ_{(X_1, \dotsc, X_j)}$ correspond to the set of ordered $j$-tuples of vertices of $V(\cJ)$, one for each $j$-set in $\cJ_X$, such that $(x_1, \dotsc, x_j) \in \cJ_{(X_1, \dotsc, X_j)}$ means that $\{ x_1, \dotsc, x_j \} \in \cJ_X$ and $x_i \in X_i$ for each $1 \leq i \leq j$.

\begin{definition}[Extensible paths]\label{def:extensible}
	Let $(G, G_{\cJ},\cJ, \cP, R)$ be a $(k,m,t, \eps, \eps_k, r, \mathbf{d})$-regular setup and $c, \nu > 0$.
	A $(k-1)$-tuple $P$ in $V(G)^{k-1}$ is said to $(c,\nu)$-\emph{extend} \emph{rightwards} to an ordered edge $X=(X_1,\dots,X_k)$ in $R$, if there is a \emph{connection set} $C \subseteq V(G)$ and a \emph{target set} $T \subseteq \cJ_{(X_2,\dots,X_k)}$  with the following properties.
	\begin{enumerate}[(i)]
		\item We have $|T| \ge \nu |\cJ_{(X_2,\dots,X_k)}|$ and
		\item for every $(v_2,\dots,v_k) \in T$, there are at least $c m^{k+1}$ many $(k+1)$-tuples $(w_1,\dots,w_{k},v_1)$ with $v_1\in C\cap X_1$ and $w_i \in C\cap X_i$ for $1 \le i \le k$ so that $P (w_1,\dots ,w_k,v_1,\dots ,v_k) $ is a tight path in $G$.
	\end{enumerate}
\end{definition}

For an illustration of extensible paths, see Figure~\ref{fig:extensible-path}.
Given a tight path $P$ in $V(G)$ and an ordered edge $X$ in $R$,
we say that \emph{$P$ $(c, \nu)$-extends rightwards to $X$} if the $(k-1)$-tuple corresponding $P$'s last $k-1$ vertices $(c, \nu)$-extends rightwards to $X$.
In this context, we will refer to $X$ as the \emph{right extension}.
We define leftward path extensions for $(k-1)$-tuples and for tight paths in an analogous way (this time corresponding to the first $k-1$ vertices of $P$).
Finally, when considering multiple paths, we refer to the union of their connection sets as their \emph{joint connection set}.

\begin{comment}

\begin{figure}
	\centering
	\begin{tikzpicture}[scale=0.5, every node/.style={transform shape}]
	
	\tikzstyle{set} = [minimum size=120,draw=black!80,circle,line width=1,draw opacity=1]
	\tikzstyle{largesubset} = [minimum size=80,draw=black!80,circle,line width=0.5,draw opacity=1, fill=white]
	\tikzstyle{smallsubset} = [minimum size=40,draw=black!80,circle,line width=1,draw opacity=1, fill=VividOrange,fill opacity=0.4]
	\tikzstyle{vertex} = [draw=black!60, fill=black=60!,circle,minimum size = 6,inner sep=0]
	\tikzstyle{edge} = [on background layer,fill opacity=0.8] 
	
	\node[set, label={90:\Huge $X_1$}] (A) at (11,0) {};
	\node[set, label={90:\Huge $X_2$}] (B) at (16.5,0) {};
	\node[set, label={270:\Huge $X_3$}] (C) at (16.5,-5.5) {};
	\node[set, label={270:\Huge $X_1$}] (D) at (11,-5.5) {};
	\node[set, label={270:\Huge $X_2$}] (E) at (5.5,-5.5) {};
	\node[set, label={270:\Huge $X_3$}] (F) at (0,-5.5) {};
	
	\node[smallsubset, label={270:\huge }] (a) at (11,1) {};
	\node[smallsubset, label={90:\huge }] (b) at (16.5,-1) {};
	\node[smallsubset, label={270:\huge }] (c) at (16.5,-5.5) {};
	\node[smallsubset, label={270:\huge }] (d) at (11,-4.5) {};
	\node[largesubset, label={270:\Huge }] (e) at (0,-5) {};
	\node[largesubset, label={90:\Huge }] (f) at (5.5,-6) {};
	
	\node[vertex, label={270:\huge $p_1$}] (p1) at (0,0) {};
	\node[vertex, label={270:\huge $p_2$}] (p2) at (5.5,0) {};
	
	\node[vertex, label={270:\huge $w_1$}] (w1) at (11,1.2) {};
	\node[vertex, label={90:\huge $w_2$}] (w2) at (16.5,-1.5) {};
	\node[vertex, label={270:\huge $w_3$}] (w3) at (16.5,-5.3) {};
	\node[vertex, label={270:\huge $v_1$}] (v1) at (11,-4.3) {};
	\node[vertex, label={270:\huge $v_2$}] (v2) at (5.5,-5.3) {};
	\node[vertex, label={270:\huge $v_3$}] (v3) at (0,-4.3) {};

	\begin{scope}[scale=0.5]
	\triple{(p1)}{(w1)}{(p2)}{6pt}{1pt}{DarkDesaturatedBlue,opacity=0.3}{DarkDesaturatedBlue,opacity=0.2}; 
	\triple{(w2)}{(p2)}{(w1)}{6pt}{1pt}{DarkDesaturatedBlue,opacity=0.3}{DarkDesaturatedBlue,opacity=0.2}; 
	\triple{(w3)}{(w1)}{(w2)}{6pt}{1pt}{DarkDesaturatedBlue,opacity=0.3}{DarkDesaturatedBlue,opacity=0.2}; 
	\triple{(w2)}{(w3)}{(v1)}{6pt}{1pt}{DarkDesaturatedBlue,opacity=0.3}{DarkDesaturatedBlue,opacity=0.2}; 
	\triple{(w3)}{(v2)}{(v1)}{6pt}{1pt}{DarkDesaturatedBlue,opacity=0.3}{DarkDesaturatedBlue,opacity=0.2}; 
	\triple{(v2)}{(v3)}{(v1)}{6pt}{1pt}{DarkDesaturatedBlue,opacity=0.3}{DarkDesaturatedBlue,opacity=0.2}; 
	\end{scope}
	
	\begin{scope}[on background layer]
	\fill[VerySoftBlue, edge] (e.north) to (f.north) -- (f.south) -- (e.south) -- (e.north)-- cycle;
	\end{scope} 
	\end{tikzpicture}
	
	\caption{
		The tuple $P=(p_1,p_2)$ extends rightwards to $(X_1,X_2,X_3)$.
		Hence for every pair of vertices $(v_2,v_3)$ in the target set $T$ (marked in blue), there are many possible choices of vertices $w_1,w_2,w_3,v_1$ in the intersection of the connection set $C$ (marked in orange) with the respective clusters, such that the concatenation $P(w_1,w_2,w_3,v_1,v_2,v_3)$ forms a tight path.
	}
	\label{fig:extensible-path}
\end{figure} 

\end{comment}

Using hypergraph regularity, we can connect extensible paths in $G$.
For this, it is required that the ends of the paths are on a common tight walk in $H$.
For instance, suppose that a tight path $P$ has right extension $X$ and a tight path $Q$ has left extension $Y$.
Then we can connect the two provided that there is a tight walk $W$ starting with the tuple $X$ ending with the tuple $Y$.
Note that it is not enough that $X$ and $Y$ share the same vertices.
For instance, if $X=(a,b,c)$ and $Y=(a,c,b)$, then there is no guarantee that $H$ contains a walk from $X$ to $Y$.
Such a walk does however exist, when $Y$ is a cyclic shift of $X$, that is $(a,b,c)$, $(b,c,a)$ or $(c,a,b)$.
Let us formalise this discussion with a few further definitions.

A \emph{cyclic shift} of a tuple $(v_1, \dots, v_k)$ is any $k$-tuple of the form $(v_i, v_{i+1}, \dots, v_k, v_1, v_2, \dots, v_{i-1})$, for a  $1 \le i \le k$.
Recall the definition of an orientation in Subsection~\ref{sec:oriented-setups}.
Consider a $k$-graph $H$ with orientation $\ori{H}$, and an ordered $k$-tuple $X$ of distinct vertices in $V(H)$.
We say that $X$ is \emph{consistent with $\ori{H}$} if there exists an oriented edge $\ori{e} \in \ori{H}$ which is a cyclic shift of $X$. 
We say that an extensible path is \emph{consistent with $\ori H$}, if its left and right extensions are consistent with an element of $\ori H$.

Let $\ori{H}$ be an orientation of a $k$-graph $H$. 
A closed tight walk $\mathcal{W}$ in $H$ is said to be \emph{compatible} with $\ori{H}$, if $W$ visits every edge of $H$ at least once, and each oriented edge of $\ori{H}$ appears at least once in ${W}$ as a sequence of $k$ consecutive vertices.

As a final piece of notation, we define a set to be sparse with regards to a partition, if its intersection with every cluster is bounded.
This will be useful to ensure that no cluster of a regular slice contains too many vertices of an absorbing path.

\begin{definition}[Sparse]\label{def:sparse-in-cP}
	Consider a partition $\cP$ of a set $V$ such that every part $X \in \cP$ has size $m$.
	For $\alpha >0$, we say that $S \subset V$ is \emph{$\alpha$-sparse in $\cP$}, if $S$ shares at most $\alpha m$ vertices in each part $X \in \cP.$
\end{definition}

\subsection{Absorption and Cover Lemmas}\label{sec:absorption-cover-lemma-statement}
Now we are ready to state the absorption and cover lemma on which the proof of the Framework Theorem (Theorem~\ref{thm:framework}) is based.
We begin by formalising the Absorption Claim (Claim~\ref{cla:absorption}) of the proof sketch in Section~\ref{sec:framework-threshold-overview}.

\begin{restatable}[Absorption Lemma]{lemma}{lemabsorption} \label{lem:absorption}
	Let $k,r,m,t \in \mathbb{N}$, and $d_2,\dots,d_k$, $\eps$, $\eps_k$, $\eta$, $\mu$, $\delta$, $\alpha$, $c$, $\nu$, $\lambda$ be such that
	\begin{align*}
	1/m & \ll 1/r,\eps \ll 1/t, c, \eps_k, d_2,\dots ,d_{k-1}, \\
	c & \ll d_2,\dots ,d_{k-1}, \\
	1/t & \ll \eps_k \ll d_k, \nu \leq 1/k, \text{ and} \\
	c &\ll \eps_k \ll \alpha \ll \eta \ll \lambda \ll \nu  \ll \mu \ll \delta, 1/k.
	\end{align*}
	Let $\mathbf{d} = (d_2, \dotsc, d_k)$ and
	let $\fS = (G,G_{\cJ}, \cJ, \cP, \ori{H})$ be an oriented representative $(k, m, t, \eps, \eps_k, r, \mathbf{d})$-regular setup.
	Suppose that $G$ has  $n \leq (1 + \alpha)mt$ vertices and minimum relative $1$-degree at least $\delta + \mu$.
	Suppose that there exists a closed tight walk which is compatible with the orientation $\ori{H}$ of $H$.
	Moreover, suppose that 
	\begin{enumerate}
		\item [\upshape(F1)] $v(H) \geq (1-\alpha)t$,
		\item [\upshape(F2)] \label{def:framework-connectivity-abs} $H$ is tightly connected, and
		\item [\upshape(F5)] \label{def:framework-outreach-abs} $H$ has minimum relative vertex degree at least $1-\delta$.
	\end{enumerate}
	
	Then there exists a tight path $P$ in $G$ such that
	\begin{enumerate}[\textnormal{(\roman*)}]
		\item $P$ is $(c, \nu)$-extensible and consistent with $\ori{H}$, 
		\item $V(P)$ is $\lambda$-sparse in $\cP$ and $V(P) \cap C =\es$, where $C$ denotes the connection set of $P$ and
		\item $P$ is $\eta$-absorbing in $G$.
	\end{enumerate}
\end{restatable}

The next lemma formalises the Cover Claim (Claim~\ref{cla:cover}) of the proof sketch in Section~\ref{sec:framework-threshold-overview}.

\begin{restatable}[Cover Lemma]{lemma}{lemcover} \label{lem:cover}
	Let $k,r,m,t \in \mathbb{N}$, and $d_2,\dots,d_k$, $\eps$, $\eps_k$, $\alpha$, $\gamma$, $c$, $\nu$, $\lambda$ be such that
	\begin{align*}
	1/m & \ll 1/r,\eps \ll 1/t, c, \eps_k, d_2,\dots ,d_{k-1}, \\
	c & \ll d_2,\dots ,d_{k-1}, \\
	1/t & \ll \eps_k \ll d_k,\nu, \alpha \leq 1/k, \text{and} \\
	\alpha & \ll \lambda \ll \nu \ll \gamma.
	\end{align*}
	Let $\mathbf{d} = (d_2, \dotsc, d_k)$ and let $\fS = (G,G_{\cJ}, \cJ, \mathcal{P}, \ori{H})$ be an oriented $(k, m, t, \eps, \eps_k, r, \mathbf{d})$-regular setup.
	Suppose that $G$ has  $n \leq (1 + \alpha)mt$ vertices.
	Moreover, suppose that 
	\begin{enumerate}
		\item [\upshape(F1)] $v(H) \geq (1-\alpha)t$,
		\item [\upshape(F2)] $H$ is tightly connected,
		\item [\upshape(F3)] $H$ contains a tight closed walk $W$ compatible with $\ori{H}$ whose length is $1 \bmod k$ and
		\item [\upshape(F4)] $H$ is $\gamma$-robustly matchable.
	\end{enumerate}
	Suppose that $P$ is a tight path in $G$ such that
	\begin{enumerate}[\textnormal{(\roman*)}]
		\item $P$ is $(c, \nu)$-extensible and consistent with $\ori{H}$ and
		\item $V(P)$ is $\lambda$-sparse in $\cP$ and $V(P) \cap C =\es$, where $C$ denotes the connection set of $P$.
	\end{enumerate}
	
	Then there is a tight cycle $C$ of length at least $(1-3\alpha)n$ which contains $P$ as a subpath.
	Moreover, the number of uncovered vertices $|V(G) \sm V(C)|$ is divisible by $k$.  
\end{restatable}

We will prove Lemma~\ref{lem:absorption} and~\ref{lem:cover} in Section~\ref{sec:absorbing} and~\ref{sec:connecting}, respectively.

\subsection{Proof of the Framework Theorem}\label{sec:actual-proof-of-framework-theorem}
Now we shall prove our second main result.

\begin{proof}[Proof of Theorem~\ref{thm:framework}]
	Suppose we are given $1 \leq d \leq k$ with $k\geq 3$.
	(The case of $k=2$ is trivially implied by Dirac's theorem.)
	Define $\delta=\hf_d(k)$, and let $\mu > 0$ be arbitrary.
	
	We start by fixing all of the necessary constants and hierarchies.
	We may assume that $\mu$ is sufficiently small compared to $\delta$.
	Start by choosing $\gamma, \nu, \lambda, \eta, \alpha, \eps_k, d_k > 0$ such that
	\begin{align*}
		\eps_k \ll \alpha \ll \eta \ll \lambda \ll \nu \ll \gamma \ll \mu,
		\quad \text{and} \quad
		1/t_0 \ll \eps_k \ll d_k \ll \mu.
	\end{align*}
	We further choose functions $r \colon \NATS \rightarrow \NATS$ and $\eps \colon \NATS \rightarrow [0,1)$ such that for any choice of $t$ respecting the hierarchies of the Absorption Lemma (Lemma~\ref{lem:absorption}) as well as the Cover Lemma (Lemma~\ref{lem:cover}) are satisfied with $r(t)$ and $\eps(t)$ playing the role of $r, \eps$.
	We apply the Regular Setup Lemma (Lemma~\ref{lem:regular-setup})  with input $\eps_k, 1/t_0, r, \eps$ to obtain $t_1, m_0$.
	Choose $c \ll 1/t_1$ and let $r = r(t_1)$ and $\eps = \eps(t_1)$ (note that by choice, $1/r, \eps \ll c$ as well).
	Finally, choose $n_0$ such that $1/n_0 \ll 1/t_1, 1/m_0, c, 1/r, \eps$.
	This concludes the selection of constants.

	Now let $G$ be an arbitrary $k$-graph on $n \ge n_0$ vertices with $\overline{\delta}_d(G) \geq \delta + \mu$.
	Our goal is to show that $G$ contains a tight Hamilton cycle.
	We start by obtaining a regular setup of $G$ which has an $(\alpha, \gamma, \delta)$-Hamilton framework.
	By Lemma~\ref{lem:regular-setup}, there exists $\mathbf{d} = (d_2, \dotsc, d_{k-1})$ and a representative $(k, m , t, \eps, \eps_k, r, d_2, \dotsc , d_k)$-regular setup $(G, G_{\cJ}, \cJ, \cP, R_{d_k})$ with $t_0 \leq t \leq t_1$, $m_0 \leq m$ and $n \leq (1+\alpha)mt$.
	Moreover, there is a $k$-graph $I$ of edge density at most $\eps_k$ such that $R = R_{d_k} \cup I$ has minimum relative $d$-degree at least $\delta +\mu/2$.
	Since $\delta=\hf_d(k)$, the constant hierarchy and the definition of the Hamilton framework threshold (Definition~\ref{def:framework-threshold}), together imply that $R$ contains an $(\alpha, \gamma, \delta)$-Hamilton framework $H$ that avoids all edges of $I$.
	So in particular $H \subset R_{d_k}$.
	
	To navigate between the edges of $H$, we fix an orientation $\ori{H}$ and a compatible walk $W$.
	Since $H$ is an $(\alpha, \gamma, \delta)$-Hamilton framework, $H$ is tightly connected and has a tight closed walk of length $1 \bmod k$.
	Combining this, we obtain a closed tight walk $W$ of length $1 \bmod k$ that visits all edges of $H$.
	We define an orientation $\ori{H}=\{ \ori{e} \in V(H)^k\colon e \in H \}$ of $H$ by choosing for every edge $e$ of $H$ a $k$-tuple (or subpath) $\ori{e}$ in $W$ which contains  the vertices of $e$.
	Note that $W$ is compatible with $\ori{H}$.
	
	Next, we select a tight absorbing path $P$.
	First, note that $1/t_1 \leq d_2, \dotsc, d_{k-1}$, which follows from definition since $\cJ$ is a $(t_0, t_1)$-equitable complex (which in turn follows from the definition of regular setup).
	So together with $t \leq t_1$ and the choice of $r=r(t_1), \eps=\eps(t_1)$ we get $1/r, \eps \ll 1/t, c, d_2, \dotsc, d_{k-1}$.
	Since $H$ is an $(\alpha, \gamma, \delta)$-Hamilton framework the conditions of the Absorption Lemma (Lemma~\ref{lem:absorption}) are satisfied.
	It follows that there exists a tight path $P$ in $G$ such that
	\begin{enumerate}[\textnormal{(\roman*)}]
		\item $P$ is $(c, \nu)$-extensible and consistent with $\ori{H}$, 
		\item $V(P)$ is $\lambda$-sparse in $\cP$ and $V(P) \cap C =\es$, where $C$ denotes the connection set of $P$, and
		\item \label{itm:P-absorbing} $P$ is $\eta$-absorbing in $G$.
	\end{enumerate}
	
	Now we cover most of the graph with a tight cycle $C$ that contains $P$ as a subpath.
	Since $H$ is an $(\alpha, \gamma, \delta)$-Hamilton framework, $W$ has length $1 \bmod k$ and by the choice of $P$, the conditions of the Cover Lemma (Lemma~\ref{lem:cover}) are satisfied.
	It follows that there is a tight cycle $C$ of length at least $(1-3\alpha)n$ which contains $P$ as a subpath.
	Moreover, the number of uncovered vertices $|V(G) \sm V(C)|$ is divisible by $k$. 
	
	We finish by absorbing the remaining vertices into $C$.
	Note that $|V(G) \sm V(C)| \leq 3 \alpha n \leq \eta n$.
	By~\ref{itm:P-absorbing} there is a tight path $P'$ on the vertex set $V(P) \cup (V(G) \sm V(C))$, which has the same $k-1$ starting vertices and the same $k-1$ ending vertices as $P$.
	It follows that $C \cup P' \subset G$ contains a tight Hamilton cycle.
	This finishes the proof.
\end{proof} 

\section{Tools for working with regularity} \label{sec:tools}

In the following, we introduce a few standard results which will help us to work with hypergraph regularity and regular slices.

\subsection{Probabilistic tools}
We will need the following concentration inequalities for random variables.

\begin{lemma}[Chernoff's inequality {\cite[Theorem 2.1]{JLR01}}]\label{lem:che}
	Let $0<\alpha<3/2$ and $X \sim \emph{\text{Bin}}(n,p)$ be a binomial random variable. 
	Then $\Pr\left(|X - np | > \alpha n p \right) < 2e^{-\alpha^2np/3}$.
\end{lemma}

\begin{theorem}[McDiarmid's inequality~\cite{McDiarmid1989}] \label{theorem:mcdiarmid}
	Suppose $X_1, \dotsc, X_m$ are independent Bernoulli random variables and $b_1, \dotsc, b_m \in [0, B]$.
	Suppose $X$ is a real-valued random variable determined by $X_1, \dotsc, X_m$ such that changing the outcome of $X_i$ changes $X$ by at most $b_i$ for all $1 \leq i \leq m$.
	Then, for all $\lambda > 0$, we have
	\[ \Pr \left(|X - \expectation[X] | > \lambda \right) \leq 2 \exp \left(  -\frac{2 \lambda^2}{B \sum_{i=1}^m b_i} \right).  \]
\end{theorem}

\subsection{Slice restriction}
As one would expect, the restriction of a regular complex to a large subset of its vertex set is also a regular complex, with slightly altered regularity constants~\cite{KMO10}.
This is formalised in the following restriction lemma.

\begin{lemma}[Slice Restriction Lemma~{\cite[Lemma 28]{ABCM17}}]\label{lem:regular-slice-restriction}
	Let $k,r,n,t$ be positive integers, and
	$d_2,\ldots,d_k,\eps,\eps_k,\alpha$ be positive constants such that
	$1/d_i\in \mathbb{N}$ for each $2\le i\le k-1$, and such
	that $1/n \ll 1/t$,
	\[1/n \ll
	1/r,\eps\ll\eps_k,d_2,\ldots,d_{k-1}\quad\text{and}\quad\eps_k\ll \alpha.\]
	Let~$G$ be a $k$-graph on~$n$ vertices, and let $\cJ$ be a $(\cdot,\cdot,\eps,\eps_k,r)$-regular slice
	for $G$ with ground partition $\cP$ and density vector $\mathbf{d} =(d_2, \dotsc, d_{k-1})$.
	For each $Z \in \cP$, let $Z' \subseteq Z$ with $|Z'| = \lceil \alpha |Z| \rceil$.
	Let $G' = G[\bigcup_{Z \in \cP} Z']$ be the induced subgraph,
	and let $\cJ' = \cJ[\bigcup_{Z \in \cP} Z']$ be the induced subcomplex.
	Then $\cJ'$ is a $(\cdot,\cdot,\sqrt{\eps},\sqrt{\eps_k},r)$-regular slice for $G'$ with density vector $\mathbf{d}$.
\end{lemma}

\subsection{Counting and extending}

Next, we state a hypergraph counting lemma, which  essentially  states that the number of copies of a given small $k$-graph inside a regular slice is roughly what we would expect if the edges inside a regular slice were chosen at random.
Different versions of this lemma are available in the literature~\cite{ABCM17,CFKO09,Gow07,RS07a}.
We use the version given by Cooley, Fountoulakis, Kühn and Osthus~\cite[Lemma 4]{CFKO09}, expressed in the language of setups.

Let $\mathcal{G}$ be a $\cP$-partite $k$-complex and $X_1, \dotsc, X_s \in \cP$ (possibly with repetition),
and let $\mathcal{H}$ be a $k$-complex on vertices $\{1, \dotsc, s \}$.
We say that an embedding of $\mathcal{H}$ in $\mathcal{G}$ is \emph{partition-respecting}, if $i$ is embedded in $X_i$, for all $1 \leq i \leq s$.
Thus the notion of partition-respecting depends on the labelling of $V(\cH)$ and the clusters $X_1, \dotsc, X_s$, but those will always be clear in the context.
We denote the set of labelled partition-respecting copies of $\mathcal{H}$ in $\mathcal{G}$ by $\mathcal{H}_\mathcal{G}$.
Recall that $e_i(\Hy)$ denotes the number of edges of size $i$ in $\Hy$.

\begin{lemma}[Counting Lemma] \label{lemma:counting}
	Let $k$, $s$, $r$, $m$ be positive integers and let $\beta$, $d_2$, $\dots$, $d_k$, $\eps$, $\eps_{k}$ be positive constants such that $1/d_i\in\mathbb N$ for $2 \leq i \leq k-1$, and such that
	\[1/m \ll 1/r,\eps \ll \eps_k,d_2,\dots,d_{k-1}\quad\textrm{and}\quad\eps_k\ll \beta,d_k,1/s.\]
	Let $H$ be a $k$-graph on $s$ vertices $\{1, \dotsc, s\}$ and let $\mathcal{H}$ be the $k$-complex generated by the down-closure of $H$.
	Let $\mathbf{d} = (d_2, \dotsc, d_k)$, let $(G, G_{\cJ}, \cJ, \cP, R)$ be a $(k, m, \cdot, \eps, \eps_k, r, \mathbf{d})$-regular setup and $\mathcal{G} = \cJ \cup G_{\cJ}$.
	Suppose $X_1, \dotsc, X_s \in \cP$ are such that $i \mapsto X_i$ is a homomorphism from $H$ into $R$.
	Then the number of labelled partition-respecting copies of $\mathcal{H}$ in $\mathcal{G}$ satisfies
	\begin{align*} \label{equation:counting}
	|\cH_\cG| = 
	\left( 1 \pm \beta \right) \left( \prod^{k}_{i=2} d^{e_i(\mathcal{H})}_i \right) m^s.
	\end{align*}
\end{lemma}

The following tool will allow us to extend small subgraphs in a regular slice.
It was originally proved by Cooley, Fountoulakis, Kühn and Osthus~\cite[Lemma 5]{CFKO09}.
Rephrased in the setting of regular setups, the lemma takes the following form.

\begin{lemma}[Extension Lemma]\label{lem:extension}
	Let $k,s,r,m$ be positive integers, where $s'<s$, and let $\beta,d_2,\dots,d_k,\eps,\eps_{k}$ be positive constants such that $1/d_i\in\mathbb N$ for $2 \leq i \leq k-1$, and such that
	\[1/m \ll 1/r,\eps \ll \eps_k,d_2,\dots,d_{k-1}\quad\textrm{and}\quad\eps_k\ll \beta,d_k,1/s.\]
	Suppose $H$ is a $k$-graph on $s$ vertices $\{1, \dotsc, s\}$.
	Let $\mathcal{H}$ be the $k$-complex generated by the down-closure of $H$,
	and let $\mathcal{H'}$ be an induced subcomplex of $\mathcal{H}$ on $s'$ vertices.
	Let $\mathbf{d} = (d_2, \dotsc, d_k)$,
	let $(G, G_{\cJ}, \cJ, \cP, R)$ be a $(k, m, \cdot, \eps, \eps_k, r, \mathbf{d})$-regular setup
	and let $\mathcal{G} = \cJ \cup G_{\cJ}$.
	Suppose	$X_1, \dotsc, X_s \in \cP$ are such that $i \mapsto X_i$ is a homomorphism from $H$ into $R$.
	Then all but at most $\beta|\mathcal{H}_{\mathcal{G}}'|$ labelled partition-respecting copies of $\mathcal{H}'$ in $\mathcal{G}$ extend to \[ (1 \pm \beta) \left( \prod_{i=2}^k d_i^{e_i(\mathcal{H}) - e_i(\mathcal{H}')} \right) m^{s - s'} \] labelled partition-respecting copies of $\mathcal{H}$ in $\mathcal{G}$.
\end{lemma}

We shall also look for structures whose edges lie entirely in the underlying $(k-1)$-complex $\cJ$ of a regular setup.
In this situation we cannot use Lemma~\ref{lemma:counting}, whose input is a regular setup as opposed to an equitable complex and requires $r$ to be large with respect to $\eps_k$.
In contrast, the edges at the $(k-1)$-th levels of $\cJ$ will only be $(d_{k-1},\eps, 1)$-regular with respect to the lower levels.
However, in the upcoming applications the parameter $\eps$, which governs the regularity, is smaller than the underlying densities $d_2, \dotsc, d_{k-1}$ up to the $(k-1)$-th level, in contrast with the situation in regular setups (where $\eps_k$ plays the role of $\eps$ and usually the opposite happens, namely $\eps_k$ could be much smaller than $d_2, \dotsc, d_{k})$.
In this setting we can use a Dense Counting Lemma, as proved by Kohayakawa, Rödl and Skokan~\cite[Corollary 6.11]{KRS2002}.
We state the following variation given by Cooley, Fountoulakis, Kühn and Osthus~\cite[Lemma 6]{CFKO09}.

\begin{lemma}[Dense Counting Lemma] \label{lemma:densecountinglemma}
	Let $k, s, m$ be positive integers and $\eps$, $d_2, \dotsc, d_{k-1}$, $\beta$ be positive constants such that
	\[ 1/m \ll \eps \ll \beta \ll d_2, \dotsc, d_{k-1}, 1/s. \]
	Suppose $H$ is a $(k-1)$-graph on $s$ vertices $\{1, \dotsc, s\}$,
	and let $\mathcal{H}$ be the $(k-1)$-complex generated by the down-closure of $H$.
	Let $\mathbf{d} = (d_2, \dotsc, d_{k-1})$ and let $\mathcal{J}$ be a $(\mathbf{d}, \eps)$-regular $(k-1)$-complex, with ground partition $\cP$ with classes of size $m$ each.
	If $X_1, \dotsc, X_s \in \cP$, then
	\[
		|\mathcal{H}_{\mathcal{J}}| = (1 \pm \beta) \prod_{i=2}^{k-1} d_i^{e_i(\mathcal{H})} m^s.
	\]
\end{lemma}

An important special case of the Dense Counting Lemma gives the number of edges in each layer of a regular slice~{\cite[Fact 7]{ABCM17}}.

\begin{lemma}\label{lemma:countinglevelsslices}
	Suppose $1/m \ll \eps \ll \beta \ll d_2, \dotsc, d_{k-1}, 1/k$ and that $\cJ$ is a $(\cdot, \cdot, \eps)$-equitable $(k-1)$-complex with density vector $(d_2, \dotsc, d_{k-1})$ and clusters of size $m$.
	Let $X$ be a set of at most $k-1$ clusters of $\cJ$.
	Then
	\begin{align*}
	|\cJ_{X}| = \left( 1 \pm \beta \right) \left( \prod^{|X|}_{i=2} d^{\binom{|X|}{i}}_i \right) m^{|X|}.
	\end{align*}
\end{lemma}

Analogously, we also have a Dense Extension Lemma~\cite[Lemma 7]{CFKO09}.
\begin{lemma}[Dense Extension Lemma] \label{lemma:denseextensionlemma}
	Let $k, s, s', m$ be positive integers, where $s' < s$, and $\eps, d_2, \dotsc, d_{k-1}, \beta$ be positive constants such that
	\[ 1/m \ll \eps \ll \beta \ll d_2, \dotsc, d_{k-1}, 1/s. \]
	Let $H$ be a $(k-1)$-graph on $s$ vertices $\{1, \dotsc, s\}$.
	Let $\cH$ be the $(k-1)$-complex generated by the down-closure of $H$,
	and let $\mathcal{H'}$ be an induced subcomplex of $\cH$ on $s'$.
	Let $\mathbf{d} = (d_2, \dotsc, d_{k-1})$ and let $\mathcal{J}$ be a $(\mathbf{d}, \eps)$-regular $(k-1)$-complex, with ground partition $\cP$ with vertex classes of size $m$ each.
	If $X_1, \dotsc, X_s \in \cP$, then all but at most $\beta|\mathcal{H}_{\mathcal{J}}'|$ labelled partition-respecting copies of $\mathcal{H}'$ in $\mathcal{J}$ can be extended into
	\[ (1 \pm \beta) \prod_{i=2}^{k-1} d_i^{e_i(\mathcal{H}) - e_i(\mathcal{H}')} m^{s - s'} \]
	labelled partition-respecting copies of $\mathcal{H}$ in $\mathcal{J}$.
\end{lemma}

\subsection{Long paths}

The next lemma of Allen, Böttcher, Cooley and Mycroft~{\cite[Lemma 40]{ABCM17}} allows us to find long paths (using even most of the vertices of a single cluster) inside a regular slice.

\begin{lemma}\label{lem:sliceconnect1}
	Let $k,r,m,t,B \in \mathbb{N}$, and $d_2,\dots,d_k,\eps,\eps_k, \nu ,\psi$ be such that $1/m \ll 1/t$, and
	\begin{align*}
	{1}/{m}, {1}/{B} & \ll  {1} /{r},\eps \ll d_2,\dots ,d_{k-1}, \eps_k, \text{ and} \\
	\eps_k & \ll \psi, d_k,\nu, {1}/{k}
	\end{align*}
	
	Let $\mathbf{d} = (d_2, \dotsc, d_k)$ and let $\fS = (G, G_{\cJ}, \cJ, \mathcal{P}, H)$ be a $(k, m, t, \eps, \eps_k, r, \mathbf{d})$-regular setup with    $V(G)= V(\cP)$ and $H$ tightly connected.
	Let	$\bvec w $ be a fractional matching of size $\mu = \sum_{e \in E(H)} \bvec w(e)$.
	Also let $X$ and $Y$ be $(k-1)$-tuples of clusters,
	and let $S_X$ and $S_Y$ be subsets of $\cJ_X$ and $\cJ_Y$ of sizes at least $\nu|\cJ_X|$ and $\nu|\cJ_Y|$ respectively.
	Finally, let $W$ be a tight walk in $H$ from $X$ to $Y$ of length at most $t^{2k}$, and let $p$ be the length of $W$. 
	
	Then for any $\ell$ divisible by $k$ with $3k \leq \ell \leq (1-\psi)\mu k m$,
	there is a tight path $P$ in $G$ of length $\ell+p(k+1)$ whose initial
	$(k-1)$-tuple forms an edge of $S_X$ and whose terminal $(k-1)$-tuple forms an edge of
	$S_Y$.
	Furthermore $Q$ uses at most $\mu(Z)m + B$ vertices from any cluster $Z \in \cP$, where $\mu(Z)=\sum_{e \ni Z} \bvec w(e)$ denotes the sum of the weights of the edges of $H$ which contain $Z$.
\end{lemma}

\section{Covering most of the vertices} \label{sec:connecting}
This section is dedicated to finding a cycle that covers most of the vertices of a regular setup.
We will show a few results on finding and connecting paths and then apply these to give a proof of Lemma~\ref{lem:cover}.

Let us begin with the existence of extensible paths.
The following proposition states that most tuples in the complex induced by an edge of the reduced graph of a regular slice also extend to that edge.
Its proof is a straightforward application of the Extension Lemma (Lemma~\ref{lem:extension}).

\begin{proposition} \label{prop:extpath}
	Let $k,m,t,r \in \mathbb{N}$ and $\eps, \eps_k, d_2,d_3, \dotsc, d_k, \beta, c, \nu $ be such that
	\begin{align*}
	1/m \ll 1/r, \eps \ll c \ll\eps_k,d_2, \dotsc, d_{k-1}
	\quad \text{and} \quad
	\eps_k \ll \beta \ll d_k,\nu.
	\end{align*}
	Let $\mathbf{d} = (d_2, \dotsc, d_k)$ and let $(G, G_{\cJ},\cJ, \cP, R)$ be a $(k, m, t, \eps, \eps_k, r, \mathbf{d})$-regular setup.
	Let $Y = (Y_1, \dotsc, Y_k)$ be an ordered edge in $R$.
	Then all but at most $\beta|\mathcal{J}_{(Y_1, \dotsc, Y_{k-1})}|$ many tuples $(v_1,\dotsc, v_{k-1})\in\mathcal{J}_{(Y_1, \dotsc, Y_{k-1})}$ are $(c,\nu)$-extensible both left- and rightwards to $Y$.
\end{proposition}

\begin{proof}
	Let $P = (v_1, v_2, \dots ,v_{3k-1})$ be a tight path on $3k-1$ vertices.
	Partition its vertex set in $k$ clusters $X_1, \dotsc, X_k$ such that $X_i = \{ v_j\colon j \equiv i \bmod k \}$ for all $1 \leq i \leq k$.
	Under this partition, $P$ is a $k$-partite $k$-graph.
	
	We define the necessary complexes to use the Extension Lemma.
	Let $\cH$ be the down-closure of the path $P$, which is a $k$-partite $k$-complex on $3k-1$ vertices.
	Let $S_1 = \{ v_1, \dotsc, v_{k-1} \}$ and $S_2 = \{ v_{2k+1}, \dotsc, v_{3k-1} \}$ correspond to the first $k-1$ and the last $k-1$ vertices of the path $P$, respectively.
	Let $\cH'$ be the induced subcomplex of $\cH$ on $S = S_1 \cup S_2$.
	Thus $\cH'$ is a $k$-partite complex on $2k-2$ vertices.
	Let $\cG = \cJ \cup G_{\cJ}$.
	
	Note that, for every $1\leq j \leq k$,  the $j$-th level $\cH'^{(j)}$  of $\cH'$ corresponds to the union of two vertex-disjoint cliques supported on $S_1$ and $S_2$ respectively (in particular, this implies that $\cH'^{(k)}$ is empty).
	Indeed, both $S_1$ and $S_2$ are covered by the first and the last edge of the path $P$, respectively, thus every $j$-set contained inside one of those sets must be present in $\cH^{(j)}$, and thus in $\cH'^{(j)}$.
	On the other hand, no $j$-edge in $\cH'^{(j)}$ intersects with both $S_1$ and $S_2$.
	This is because no edge of $P$ intersects with both $S_1$ and $S_2$, as there are $k+1$ vertices between the first and last $k-1$ vertices of $P$.
	
	Let $\cH'_\cG$ is the set of labelled partition-respecting copies of $\cH'$ in $\cG$.
	It follows that
	\begin{align}
	|\cH'_{\cG}| = (1 \pm \eps_k ) |\cJ_{(Y_1, \dotsc, Y_{k-1})}|^2, \label{equation:labelledpartition}
	\end{align}
	where the error term $\eps_k$ only accounts for the fact that we do not count intersecting pairs of edges in $\cJ_{(Y_1, \dotsc, Y_{k-1})}$.
	Since $\{ Y_1, \dotsc, Y_k \}$ is an edge of $R$, any function $\phi \colon V(P) \rightarrow V(R)$ such that $\phi(X_i) \subseteq Y_i$  is a hypergraph homomorphism.
	Therefore we can use the Extension Lemma (Lemma~\ref{lem:extension}) with $\beta^2$ playing the role of $\beta$ to deduce that all but at most $\beta^2 |\cH'_{\cG}|$ of the labelled partition-respecting copies of $\cH'$ in $\cG$ extend to at least $cm^{k+1}$ labelled partition-respecting copies of $\cH$ in $\cG$ (where we have used $c \ll d_2, \dotsc, d_{k-1}$).
	For each $e \in \cJ_{(Y_1, \dotsc, Y_{k-1})}$, let $T(e)$ be the number of tuples $e'$ in $\cJ_{(Y_1, \dotsc, Y_{k-1})}$ such that $e \cup e'$ can be extended to at least $c m^{k+1}$ copies of $\cH$ in $\cG$.
	By the previous discussion and inequality~\eqref{equation:labelledpartition}, we have
	\begin{align}
	\sum_{e \in \cJ_{(Y_1, \dotsc, Y_{k-1})}} T(e) \ge (1 - 2\beta^2)|\cJ_{(Y_1, \dotsc, Y_{k-1})}|^2. \label{equation:sumofallextensions}
	\end{align}
	
	Let $S \subseteq \cJ_{(Y_1, \dotsc, Y_{k-1})}$ be the set of $(k-1)$-tuples $e$ which do not $(c, \nu)$-extend leftwards to $Y$, that is, $T(e) < \nu |\cJ_{(Y_1, \dotsc, Y_{k-1})}|$.
	From inequality~\eqref{equation:sumofallextensions} we deduce \[ |S| \leq \frac{2 \beta^2}{1 - \nu} |\cJ_{(Y_1, \dotsc, Y_{k-1})}| \leq \frac{\beta}{2} |\cJ_{(Y_1, \dotsc, Y_{k-1})}|, \]
	where the last inequality follows from $\beta \ll \nu$.
	An analogous (symmetric) argument shows that all but at most $\frac{\beta}{2} |\cJ_{(Y_1, \dotsc, Y_{k-1})}|$ many $(k-1)$-tuples $e$ in $\cJ_{(Y_1, \dotsc, Y_{k-1})}$ are $(c, \nu)$-extensible rightwards to $Y$.
	We deduce that all but at most $\beta |\cJ_{(Y_1, \dotsc, Y_{k-1})}|$ many pairs $e$ in $\cJ_{(Y_1, \dotsc, Y_{k-1})}$ are $(c, \nu)$-extensible both left- and rightwards to $Y$.
\end{proof}

Now that we know how to find extensible paths, we can turn to connecting them.
The following lemma allows us to connect up two extensible paths using either very few or quite a lot of vertices.

\begin{lemma}[Connecting two paths] \label{lem:connecting-two-paths}
	Let $k,r,m,t \in \mathbb{N}$, and $d_2,\dots,d_k$, $\eps$, $\eps_k$, $c$, $\nu$, $\lambda$ be such that
	\begin{align*}
	1/m & \ll 1/r,\eps \ll c \ll \eps_k,d_2,\dots ,d_{k-1}, \\
	\lambda & \ll \nu \ll 1/k, \text{and } \eps_k \ll d_k.
	\end{align*}
	Let $\mathbf{d} = (d_2, \dotsc, d_k)$ and let $\fS = (G, G_{\cJ}, \cJ, \mathcal{P}, H)$ be a $(k, m, t, \eps, \eps_k, r, \mathbf{d})$-regular setup such that $H$ is tightly connected.
	Let $P_1,P_2 \subset G$ be $(c, \nu)$-extensible paths such that $P_1$ extends rightwards to $X$ and $P_2$ extends leftwards to $Y$.
	Suppose furthermore that $P_1$ and $P_2$ are either identical or pairwise disjoint.
	Let $W$ be a tight walk in $H$  of {length~$p \leq t^{2k}$} that starts with $X$ and ends with $Y$.
	Let $C$ be the joint connection set of $P_1$ and $P_2$ and suppose that $C$ is $\lambda$-sparse in $\cP$.
	Let $S \subset V(G)$ be a set which is $\lambda$-sparse in $\cP$, contains $V(P_1) \cup V(P_2)$, and satisfies $C \cap S =\es$.
	Then
	\begin{enumerate}[label=\textnormal{(\roman*)}]
		\item \label{item:sliceconnect-simplepath} there is a tight path $Q$ of length $3k + (p+2)(k+1)$ in $G[V(\cP)]$ such that $P_1QP_2$ is a tight path, which contains no vertices of $S$ and exactly $2(k+1)$ vertices of $C$.
	\end{enumerate}
	Moreover, consider $\psi$ with $\eps_k \ll \psi$.
	Let $\bvec w \in \REALS^{E(H)}$ be a fractional matching of size  $\mu = \sum_{e \in E(H)} \bvec w(e) \geq 4/m$ such that $\sum_{e \ni Z} \bvec w(e) \leq 1-2\lambda$ for every $Z \in \cP$.
	Then 
	\begin{enumerate}[label=\textnormal{(\roman*)},resume]
		\item \label{item:sliceconnect-longpath} there exists a tight path $Q$ of length congruent to $p+2$ modulo $k$ in $G[V(\cP)]$ such that $P_1QP_2$ is a tight path in $G$ which contains no vertices of $S$ and exactly $2(k+1)$ vertices of $C$.
		Moreover, there is a set $U\subset V(\cP)$ of size at most $\psi mt$ such that $U \cup V(Q)$ has exactly $\lceil \sum_{e \ni Z} \bvec w(e) m\rceil$ vertices in each cluster $Z \in \cP$.
	\end{enumerate}
\end{lemma}
\begin{proof}
	
	We start by identifying the right and left target sets of $P_1$ and $P_2$.
	Let us write $X=(X_1,\dots,X_k)$.
	By the definition of extensible paths (Definition~\ref{def:extensible}), there exists a target set $T_1 \subseteq \cJ_{(X_2,\dots,X_k)}$ of size $|T_1| \ge \nu |\cJ_{(X_2,\dots,X_k)}|$ such that, for every $(v_2,\dots,v_k) \in T_1$, there are at least $c m^{k+1}$ many $(k+1)$-tuples $(w_1,\dots,w_{k},v_1)$ with $v_1\in C\cap X_1$ and $w_i \in C\cap X_i$ for $1 \le i \le k$ so that $P_1 (w_1,\dots, w_k,v_1,\dots ,v_k) $ is a tight path in $G$.
	Similarly, let $P_2$ have left extension $Y=(Y_1,\dots,Y_k)$, together with the corresponding target set $T_2 \subset \cJ_{(Y_1,\dots,Y_{k-1})}$.

	In the following, we select a path $Q'$ that starts with an element of the target set $T_1$ and ends with an element of  the target set $T_2$.
	Since $Q'$ shall not intersect with $S'=S\cup C$, we restrict the regular slice $\cJ$ as follows.
	For each $Z \in \cP$, select a subset $Z' \subset Z \sm S'$ of size $m' = (1 - 2\lambda) m$.
	This is possible since both $S$ and $C$ are both ${\lambda}$-sparse in $\cP$.
	Let $\cP'=\{Z'\}_{Z \in \cP}$ and $V(\cP') = \bigcup_{Z \in \cP} Z'$.
	Let $G'=G[V(\cP')]$ (resp. $G'_{\cJ'} = G_{\cJ}[V(\cP')]$) be the induced subgraph of $G$ (resp. $G_{\cJ}$) in $V(\cP')$.
	Let $\cJ' = \cJ[V(\cP')]$ be the respective induced subcomplex.
	By the Slice Restriction Lemma (Lemma~\ref{lem:regular-slice-restriction}), $\fS' := (G',G_{\cJ'}', \cJ', \mathcal{P}', H)$ is a $(k, m', t, \sqrt{\eps}, \sqrt{\eps_k}, r, \mathbf{d})$-regular setup with $H$ as above.
	Since $\lambda \ll \nu \ll 1/k$, we still have $|T_1| \ge \nu^2 |\cJ'_{(X_2,\dots,X_k)}|$ and  $|P^-| \ge \nu^2 |\cJ'_{(X_2,\dots,X_k)}|$.
	This concludes the restriction analysis.
	
	Now we choose a path $Q'$ for case~\ref{item:sliceconnect-longpath} as follows.
	Denote the scaled (to $m'$) size of $\bvec w$ by $\mu' = \mu/(1 - 2\lambda)$.
	Let $B \in \mathbb{N}$ be such that $1/B \ll 1/r, \eps$.
	Let $\ell$ be the largest integer divisible by $k$ with $3k \leq \ell \leq (1- \psi/4) \mu' k m'$.
	(Note that such an $\ell$ exists, since $(1- \psi/4) \mu' m' \geq 3$, where the latter inequality follows from the assumption of $\mu \geq 4/m$.)
	We apply Lemma~\ref{lem:sliceconnect1} with $G'$, $\cJ'$, $W$, $\ell$, $\bvec w$, $\mu'$ and $T_1,T_2$ (playing the role of $S_X,S_Y$)
	to obtain a tight path $Q' \subset G'$ whose initial $(k-1)$-tuple forms an edge of $T_1$ and whose terminal $(k-1)$-tuple forms an edge of $T_2$.
	Furthermore, $Q'$ has $\ell+p(k+1)$ vertices, where $p$ is the length of the tight path $W$ from the assumption.
	Finally, $Q'$ uses at most $\mu(Z)m'/(1-2\lambda) + B = \mu(Z)m + B $ vertices from any cluster $Z \in \cP'$, where $\mu(Z)=\sum_{e \ni Z} \bvec w(e)$ denotes the (unscaled) total weight of edges of $\bvec w$ containing $Z$.
	Note that $\ell \geq (1-\psi/4) \mu k m - k$.
	Since $B \ll \psi  \mu m  k$, it follows that $\sum_{Z \in \cP} \bvec w(e) m - \sum_{Z \in \cP} V(Q' \cap Z) \leq \psi mt/2$.
	Hence, there is a set $U\subset V(\cP)$ of size at most $\psi mt$ such that $U \subset V(Q)$ has $\lceil \sum_{e \ni Z}  w(e)  m \rceil$ vertices in each cluster $Z \in \cP$.
	This gives the desired path $Q'$ for case~\ref{item:sliceconnect-longpath}.
	
	For case~\ref{item:sliceconnect-simplepath}, we can choose a path $Q'$ in the same way.
	The only difference being that we let $\bvec w$ be a single edge of weight $1$ and $\ell=3k$.
	Hence $Q'$ is a path of length $3k+p(k+1)$.
	
	We finish in both cases by using the above detailed property of extensible paths to choose vertices $w_1,\dots,w_{k},v_1$ and $v_k',w_1',\dots,w_{k}'$ in $C$ such that for
		\[Q=(w_1,\dots ,w_{k} ,v_1) Q' (v'_k, w_1', \dots ,w_{k}'),\]
	the concatenation $P_1QP_2$ is a tight path and $Q$ is disjoint from $S$.
	This is possible because $V(S) \cap C =\es$ and $C \cap V(Q')=\es$ by definition of $S'$ and $\cP'$.
	
	Finally, note that the length of $Q$ is the length of $Q'$ plus $2(k+1)$.
	In case~\ref{item:sliceconnect-longpath}, the length of $Q'$ is $\ell + p(k+1)$, and since $\ell$ is divisible by $k$, we deduce the length of $Q$ is congruent to $p+2$ modulo $k$, as required.
\end{proof}

Next, we will show how to connect many paths at once.
For this purpose, will need the following two propositions to bound the length of compatible walks in the reduced graph.

\begin{proposition}\label{prop:closed-walk-bounded-length}
	Let $W$ be a closed tight walk in a $k$-graph $H$ on $t$ vertices which starts and ends with $k$-tuples $X$ and $Y$.
	Then there is a closed tight walk $W'$ of length at most $kt^k$, which starts with $X$ and ends with $Y$.
	Moreover, the length of $W'$ is congruent to the length of $W$ modulo $k$.
\end{proposition}
\begin{proof}
	Suppose that $W$ has length $j \bmod k$ for a $0 \leq j \leq k-1$.
	Let $W'$ be a vertex-minimal closed tight walk from $X$ to $Y$ of size $j \bmod k$.
	We show that no $k$-tuple $Z$ repeats more than $k$ times on $W'$, which implies that $W'$ has length  at most $kt^k$.
	
	For sake of contradiction, assume that $W'$ contains $k+1$ copies of the same $k$-tuple $Z$ and denote by $i_j$ the position in $W'$ where the $j$-th repetition $Z$ begins.
	Clearly, $i_j - i_1 \not \equiv  0 \bmod k$ for every $1 \leq j \leq k+1$, since otherwise we could reduce the length of $W'$ by deleting the vertices from $i_1$ to $i_j-1$ to obtain a shorter walk with the same properties.
	Hence, by the pigeonhole principle, there are indices $1 \leq j < j' \leq k+1$ such that $i_j - i_1 \equiv i_{j'} - i_1 \bmod k$, which is equivalent to $i_j - i_{j'} \equiv 0 \bmod k$.
	It follows that we can reduce the length of $W'$ by deleting the vertices between $i_j$ and $i_{j'}-1$ to obtain a shorter walk with the same properties.
	A contradiction.\footnote{
		A simpler application of the pigeonhole principle shows that no more than $k(k-1)+1$ repetitions can occur, which is also fine for our purposes.
		This problem is also known as the following riddle:
		Given $101$ Dalmatians, is there a non-empty subset of Dalmatians whose number of dots sums to a multiple of $101$?
	}
\end{proof}

Recall that an {orientation} of a $k$-graph $H$ is a family of ordered $k$-tuples $\{ \ori{e} \in V(H)^k\colon e \in H \}$, one for each edge in $H$, such that $\ori{e}$ consists of an ordering of the vertices of $e$.
Moreover, we defined a $k$-tuple $X$ to be {consistent with $\ori{H}$}, if $ \ori{H}$ contains a cyclic shift of $X$.
Finally, we said that a closed tight walk $\mathcal{W}$ in $H$ is {compatible} with $\ori{H}$, if $W$ visits every edge of $H$ at least once, and each oriented edge of $\ori{H}$ appears at least once in ${W}$ as a sequence of $k$ consecutive vertices.

\begin{proposition}\label{prop:closed-walk-compatible}
	Consider $j,k,t \in \NATS$ with $1 \leq j \leq k $.
	Let $W$ be a closed tight walk that is compatible with respect to an orientation $\ori{H}$ of a $k$-graph $H$ on $t$ vertices.
	Let $X$ and $Y$ be  $k$-tuples of vertices in $H$ which are consistent with $\ori{H}$.
	Then there is a tight walk $W'$ of length at most $kt^k$, which starts with $X$ and ends with $Y$.
	Moreover, if $W$ has length $1 \bmod k$, then $W'$ has length $j \bmod k$.
\end{proposition}
\begin{proof}
	By Proposition~\ref{prop:closed-walk-bounded-length} it suffices to show that there exists a tight walk $W'$ such that firstly, $W'$ starts with $X$ and ends with $Y$ and secondly, $W'$ has length $j \bmod k$, if $W$ has length $1\bmod k$.
	
	Since $X$ is consistent with $\cF$, there is a tight walk $W_X$ of length at most $k-1$ from $X$ to an $X'$ in $\ori{H}$.
	Similar, there is a tight walk $W_Y$ of length at most $k-1$ from $Y$ to a $Y'$ in $\ori{H}$.
	Since $W$ is compatible with respect to $\ori{H}$, there is a subwalk $W_{X' Y'} \subset W$ which starts with $X'$ and ends with $Y'$.
	Hence the concatenation $W' = X W_X W_{X' Y'} W_Y Y$ satisfies the first part.
	
	For the second part, note that we can actually choose $W_{X' Y'}$ such that $X W_X W_{X' Y'} W_Y Y$ has length $j \bmod k$ by extending $W_{X' Y'}$ along the same $k$-tuple with copies of $W$, for an appropriate number of times.
	This in turn is possible since any number coprime to $k$ is a generator for the finite cyclic group $\mathbb{Z}/k\mathbb{Z}$.
\end{proof}

This concludes the discussion on compatible walks.
Next, we use these results to connect  many extensible paths.
Given a set of hypergraphs  $\cA$, we denote its joint vertex set by $V(\cA) =\bigcup_{G \in \cA} V(G)$.

\begin{lemma}[Connecting many paths] \label{lem:connecting-many-paths}
	Let $k,r,m,t \in \mathbb{N}$, and $d_2,\dots,d_k$, $\eps$, $\eps_k$, $c$, $\nu$, $\lambda$,  $\zeta$ be such that
	\begin{align*}
	1/m & \ll 1/r,\eps \ll 1/t, \zeta, \eps_k, d_2,\dots ,d_{k-1}, \\
	\zeta & \ll c \ll d_2,\dots ,d_{k-1}, \\
	1/t & \ll \eps_k \ll d_k,\nu \leq 1/k, \text{and} \\
	\lambda & \ll \nu \ll 1/k.
	\end{align*}
	Let $\mathbf{d} = (d_2, \dotsc, d_k)$ and let $(G, G_{\cJ},\cJ, \cP, H)$ be a $(k, m, t, \eps, \eps_k, r, \mathbf{d})$-regular setup, with $H$ tightly connected. 
	Let $\ori{H}$ be an orientation of $H$ with a compatible closed walk $W$.
	Suppose that $\cA$ is a set of pairwise disjoint $(c, \nu)$-extensible paths consistent with $\ori{H}$ and with joint connection set $C$.
	Moreover, assume that
	\begin{enumerate}[\upshape(a)]
		\item $|\cA| \leq \zeta m$,
		\item $V(\cA)$ is ${\lambda}$-sparse in $\cP$ and
		\item $V(\cA)   \cap C = \es$.
	\end{enumerate}
	Consider any two elements $P_1,P_2$ of $\cA$.
	Then there is a tight path $P$ in $G$ such that
	\begin{enumerate}[\upshape(A)]
		\item \label{itm:P-contains-A} $P$ contains every path of $\cA$ as a subpath,
		\item \label{itm:P-starts-with-P1-and-P2} $P$ starts with $P_1$ and ends with $P_2$,
		\item \label{itm:P-A-subset-partition} $V(P) \sm V(\cA) \subset V(\cP)$ and
		\item \label{itm:P-A-intersectio-Z} $V(P) \sm V(\cA)$ intersects in at most $   10k^2\mathcal{A}_Z +t^{t+3k}$ vertices with each cluster $Z \in \cP$, where $\cA_Z$ denotes the number of paths of $\cA$ that intersect with $Z$.
	\end{enumerate} 
\end{lemma}

\begin{proof}
	We start by choosing a set $C'$ of random vertices in $G$ by including every vertex of $V(\cP)$ independently with probability $c$.
	By Chernoff's inequality (Lemma~\ref{lem:che}) and a union bound, it follows that with probability at least $1 - 2 t \exp(-\Omega(m))$ the set $C'$ is $(2c)$-sparse.
	Also by McDiarmid's inequality (Theorem~\ref{theorem:mcdiarmid}), with probability at least $1 - 2 m^{k-1} \exp(-\Omega(m))$ a fixed path of $\cA$ is $(c^{k+2}/2, \nu)$-extensible with connection set $C'$.
	Since $|\cA| \leq \zeta m$ and $1/m \ll 1/t \ll 1/k$, all of these events happen simultaneously with positive probability, and thus we can 
	fix a set $C'$ with these properties.
	
	Now we use the following procedure to merge paths in $\cA$.
	Initiate $S=V(\cA)$.
	While there are two paths $Q_1,Q_2 \in \cA$ such that the extension to the right of $Q_1$ equals the extension to the left of $Q_2$, apply Lemma~\ref{lem:connecting-two-paths}\ref{item:sliceconnect-simplepath} with $p=k$  $c^{k+4}/2$ playing the role of $c$ to obtain a path $Q$ of length $10k^2$ which avoids $S$ and has exactly $2(k+1)$ vertices in $C'$.
	Then add $V(Q)$ to $S$, replace $Q_1,Q_2$ with $Q$ in $\cA$ and delete the $2(k+1)$ vertices used by $Q$ from $C'$.	
	For convenience, let us denote the set of paths after the procedure by $\cA'$ and keep the name $\cA$ for the initial set of paths.
	
	Note that the size of $S$ grows by at most $10k^2 |\cA| \leq 10 k^2 \zeta m \leq \lambda m$
	and we delete at most $2(k+1)|\cA| \leq 2(k+1)\zeta m \leq c^{k+2}m/4$ vertices from $C$ throughout this process, where the last inequality follows from $\zeta \ll c, 1/k$.
	This implies that every path of $\cA$ remains $(c^{k+2}/4, \nu)$-extensible with connecting set $C'$ (updated in each step).
	Hence the conditions of Lemma~\ref{lem:connecting-two-paths}\ref{item:sliceconnect-simplepath} are satisfied in every step and therefore $\cA'$ is well-defined.
	
	We remark that, when the procedure ends, $\cA'$ has size at most $t^t$.
	Moreover, the paths of $\cA'$ inherit from $\cA$ the property of being {consistent with $\ori{H}$} (since each extension of a path in $\cA'$ must be an extension of a path in $\cA$).
	
	To finish, we continue by connecting up the paths of $\cA'$ to the desired path $P$ along the orientation $\ori{H}$.
	As the paths of $\cA'$ are consistent with $\ori{H}$, the right and left extensions of each path in $\cA'$ are contained in the walk $W$.
	Since $W$ is compatible with $\ori{H}$, we can apply Proposition~\ref{prop:closed-walk-compatible} to obtain a tight walk in $H$ of length at most $t^{2k}$ between the right and left end of each path in $\cA'$.
	Using an appropriate restriction (Lemma~\ref{lem:regular-slice-restriction}) and again Lemma~\ref{lem:connecting-two-paths}\ref{item:sliceconnect-simplepath}, we can connect up the paths of $\cA'$ to a single tight path $P$ using at most $t^t \cdot t^{3k}$ further vertices of $V(\cP)$.
	
	By construction, $P$ contains every path in $\cA$ as a subpath and $V(P) \sm V(\cA) \subset V(\cP)$, giving~\ref{itm:P-contains-A} and~\ref{itm:P-A-subset-partition}.
	Moreover, note that $V(\cA') \sm \cA$ intersects in at most $10k^2\mathcal{A}_Z$ vertices with each cluster $Z \in \cP$, where $\cA_Z$ denotes the number of paths of $\cA$ that intersect with $Z$.
	Hence, $P$ also satisfies~\ref{itm:P-A-intersectio-Z}.
	Finally note that, we can easily guarantee that $P$ starts and ends with any two particular paths of $\cA$ by connecting those paths in the very end, which yields~\ref{itm:P-starts-with-P1-and-P2}.
\end{proof}

We conclude this section with the proof of the Cover Lemma.
%We conclude this section with the proof of the Cover Lemma, which we restate here for convenience.
%\lemcover*

\begin{proof}[Proof of Lemma~\ref{lem:cover}]
	Set $P_1 = P$.
	We start by selecting a short path $P_2$, disjoint from $P_1$.
	Suppose that $P_1$ extends to the right with $X$ and to the left with $Y$.
	By Proposition~\ref{prop:extpath} there exists a path $P_2$ of length $k-1$ which $(c, \nu)$-extends both right- and leftwards to $Y$.
	Moreover, by the Slice Restriction Lemma (Lemma~\ref{lem:regular-slice-restriction}), we can assume that $V(P_1)$ is disjoint of $V(P_2)$ and $C_2$, where $C_2$ is the connection set of $P_2$ (see the proof of Lemma~\ref{lem:connecting-two-paths} for details).
	Next, we use McDiarmid's inequality (Theorem~\ref{theorem:mcdiarmid}) to  pick a $\lambda$-sparse vertex set $C'$ such $P_1$, $P_2$ are $(c^{k+2}/2, \nu)$-extensible paths with joint connection set $C'$ (see the proof of Lemma~\ref{lem:connecting-many-paths} for details).
	 
	In the following, we will find two tight paths $Q_1$ and $Q_2$ such that $P_1Q_1P_2Q_2$ forms the desired tight cycle.
	The purpose of the first path is to flatten out the intersection of $S$ with the clusters of $\cP$, such that $V(Q_1) \cup S$ has approximately the same number of vertices in every cluster.
	The second path will cover most of the remaining vertices and adjust the length modulo $k$.
	
	To obtain the first path, we restrict our scope, similarly as it was done in the proof of Lemma~\ref{lem:connecting-two-paths}.
	Let $S_1 = V(P_1) \cup V(P_2)$.
	Choose $\kappa > 0$ such that $\lambda \ll \kappa \ll \gamma$.
	For each $Z \in \cP$, select a subset $Z' \subset Z$ of size $m' = \kappa m$, such that $Z'$ contains $Z \cap S_1$.
	This is possible since $S_1$ is ${(2\lambda)}$-sparse in $\cP$ and $2\lambda \ll \kappa$.
	Let $\cP'=\{Z'\}_{Z \in \cP}$ and $V(\cP') = \bigcup_{Z \in \cP} Z'$.
	Let $G'=G[V(\cP')]$ and $G'_{\cJ'} = G_{\cJ}[V(\cP')]$ be the corresponding induced subgraphs and $\cJ' = \cJ[V(\cP')]$ be the induced subcomplex.
	By the Slice Restriction Lemma (Lemma~\ref{lem:regular-slice-restriction}), $\fS' = (G', G'_{\cJ'},\cJ', \mathcal{P}', H)$ is a $(k, m', t, \sqrt{\eps}, \sqrt{\eps_k}, r, d_2,\dots,d_k)$-regular setup, as desired.
	
	Now we define a fractional matching that complements the discrepancy of $S$ in the clusters $\cP$.
	Consider $\bvec b \in \REALS^{V(H)}$ by setting 
	$\bvec b({Z'}) =  |{Z'} \sm S_1|/|Z'|$
	for every $Z \in V(H)$.
	Recall that $|S_1 \cap Z|\leq 2 \lambda m$, $|Z'| = \kappa m$ and $\lambda \ll \kappa, \gamma$.
	It follows that
	\begin{align*}
	1-\gamma \leq 1-\frac{2\lambda}{\kappa} \leq 1-  \frac{|S_1|}{|Z'|} \leq \bvec b \leq 1.
	\end{align*}
	Since $H$ is $\gamma$-robustly matchable, there is a fractional matching $\bvec w$ such that $\sum_{e \ni Z'} \bvec w (e)=\bvec b(Z)$ for every cluster $Z \in \cP'$.
	Consider $\psi > 0$ with  $\eps_3 \ll \psi \ll \alpha$.
	By Lemma~\ref{lem:connecting-two-paths}~\ref{item:sliceconnect-longpath} there exists a tight path $Q_1$ in $G'$ such that $P_1Q_1P_2$ is a tight path in $G$ which contains no vertices of $S_1$ and exactly $2(k+1)$ vertices of $C'$.
	Moreover, there is a set $U\subset V(\cP)$ of size at most $\psi mt$ such that $U \cup V(Q_1)$ has $\lceil \sum_{e \ni Z}  \bvec w(e) m\rceil$ vertices in each cluster $Z \in \cP'$.
	By definition of $\bvec b$, this (crudely) translates to $(1\pm\psi)(\kappa m - |Z \cap S'|)$ vertices in each $Z \in V(H)$.
	In other words, 
	\begin{equation}\label{equ:PQP-balanced}
	\text{the set $V(P_1Q_1P_2) \cup U$ has $(1\pm\psi) \kappa m $ vertices in each cluster $Z \in V(H)$,}
	\end{equation} 
	which suffices for our purposes.
	
	We now turn to the second path $Q_2$.
	Note that $P_1Q_1P_2$ has left right extensions $X$ and $Y$, which in particular are consistent with $\ori{H}$.
	Since $W$ is compatible with $\ori{H}$, we can apply Proposition~\ref{prop:closed-walk-compatible} to obtain a tight walk $W'$ in $H$ of length $p \leq t^{2k}$ starting with $X$ and ending with $Y$.
	Moreover, since $W$ has length coprime to $k$, we can choose $W'$ such that
	\begin{align}\label{equ:cover-almost-everything-divisibility}
		p+2 = |V(G) \sm V(P_1Q_1P_2)| \mod k.
	\end{align}
	
	Next, let $S_2 = V(P_1Q_1P_2)$ and $C'' = C' \sm S_2$.
	We define $\bvec c \in \REALS^{V(H)}$ by setting $\bvec c(Z) = (m - |Z\cap S_2|)/m$ for all $Z \in V(H)$.
	Note that $1-\gamma \leq 1-\kappa-\psi \leq \bvec c \leq 1$.
	Since $H$ is {$\gamma$-robustly} matchable, there is a fractional matching $\bvec w$ such that $\sum_{e \ni Z} \bvec w(e)=\bvec b(Z)$ for every $Z \in \cP$.
	By Lemma~\ref{lem:connecting-two-paths}~\ref{item:sliceconnect-longpath}  (this time applied with $\lambda =\psi/2$) there exists a tight path $Q_2$ in $G$ of length congruent with $p+2$ modulo $k$ which contains no vertices of $S_2$ and exactly $2(k+1)$ vertices of $C''$ such that $F=P_1Q_1P_2Q_2$ is a tight cycle.
	Moreover, there is a set $U'\subset V(\cP)$ of size at most $\psi mt$ such that $U' \cup V(Q_2)$ has $\lceil \sum_{e \ni Z}  (1-\kappa) m\rceil$ vertices in each cluster $Z \in \cP$.
	
	Note that the cycle $F$ contains all vertices of $G$, except for those in $V(G) \sm V(\cP)$, $\bigcup_{Z \in V(\cP) \sm V(H)}Z$, $U$ and $U'$.
	By assumption we have $|V(G) \sm V(\cP)| \leq \alpha n$, $|\bigcup_{Z \in V(\cP) \sm V(H)}Z| \leq \alpha mt$, and $|U|,|U'| \leq \psi mt$.
	Together with assertion~\eqref{equ:PQP-balanced} this yields that the cycle $F$ covers all but $\alpha n+ \alpha mt +2\psi mt    \leq 3\alpha n$ vertices.
	Moreover, as $Q_2$ has length congruent to $p+2 \bmod k$ and by equation~\eqref{equ:cover-almost-everything-divisibility}, it follows that $|V(G) \sm V(F)|$ is divisible by $k$.
	This concludes the proof.
\end{proof}

\section{Attaching vertices to a regular complex} \label{sec:preabsorbing}

In this section we show a few technical results, which will help us to attach vertices to regular complexes.
This will be helpful to construct the absorbing paths later on.
Recall that in a representative regular setup (Definition~\ref{def:setup}) the neighbourhoods of every vertex are well-represented in the regular complexes of the slice.
More precisely, given a $k$-graph $G$ with large minimum $1$-degree, and (a dense subgraph of) the reduced graph of a regular slice of $G$, one can show that each vertex of $G$ is joined to a large number of edges which are supported in the corresponding $(k-1)$-tuples of cluster of the regular complex.
In the following, we formalise this observation.

Recall that, given a subset $X \subseteq \cP$ of at most $k-1$ clusters of $\cJ$, $\cJ_X$ is an $|X|$-partite $|X|$-graph containing all the edges of the $|X|$-th level of $\cJ$ which intersect each cluster of $X$ exactly once.
Let $G$ be a $k$-graph, $\cJ$ a regular slice with cluster set $\cP$.
Given $v \in V(G)$ and $\delta > 0$, we define
\begin{align}
N_{\cJ}(v, \delta) = \{ X \subseteq \cP\colon |X| = k-1,~ |N_G(v; \cJ_X)| > \delta |\cJ_X| \}, \label{equation:tupleneighbours}
\end{align}
so $N_{\cJ}(v, \delta)$ corresponds to the sets $X$ of $k-1$ clusters in $\cP$ such that $v$ has at least a  proportion of $\delta$ of the $(k-1)$-edges of $\cJ_X$ as neighbours.

\begin{lemma} \label{lemma:neighboursincomplex}
	Let $k,r,m,t \in \mathbb{N}$, and $d_2,\dots,d_k,\eps,\eps_k,\mu,\delta$ be such that
	\begin{align*}
	{1}/{m} \ll {1}/{r},\eps \ll \eps_k,d_2,\dots ,d_{k-1}, \quad  \eps_k \ll d_k \leq 1/k, \quad \text{and} \quad  \eps_k \ll \mu \ll \delta,
	\end{align*}
	Let $\mathbf{d} = (d_2, \dotsc, d_k)$ and let $(G, G_{\cJ},\cJ, \cP, R)$ be a representative $(k, m, t, \eps, \eps_k, r, \mathbf{d})$-regular setup.
	Suppose that $G$ has minimum relative $1$-degree at least $\delta+\mu$.
	Then, for all $v \in V(G)$,
	\begin{align*}
		\left|N_{\cJ}\left(v, \frac{\mu}{3}\right)\right| \ge \left( \delta + \frac{\mu}{4} \right) \binom{t}{k-1}.
	\end{align*}
\end{lemma}

\begin{proof}
	Let $v \in V(G)$ be arbitrary.
	The minimum degree condition implies that $\reldeg_G(v) \ge \delta + \mu$.
	Since $\mathfrak{S}$ is representative (part \ref{item:regsetup-rootedrepresent} of Definition~\ref{def:setup}),  we have $|\reldeg_G(v) - \reldeg_G(v; \cJ)| < \eps_k$ and therefore
	\begin{align}
	\deg_{G}(v;\cJ^{(k-1)}) \ge  \left(\delta + \mu - \eps_k\right) |\cJ^{(k-1)}| \ge \left(\delta + \frac{2}{3} \mu \right) |\cJ^{(k-1)}|, \label{equation:dgvJk-1}
	\end{align} where we have used that $\eps_k \ll \mu$ for the last inequality.
	
	Recall that, for any set $X \subseteq \cP$ of $k-1$ clusters, $\cJ_X$ corresponds to those $(k-1)$-edges of $\cJ^{(k-1)}$ which are $X$-partite.
	Define $d_X = \prod_{i=2}^{k-1} d_i^{\binom{k-1}{i}}$.
	By the corollary (Lemma~\ref{lemma:countinglevelsslices}) of the Dense Counting Lemma and  $\eps \ll \eps_k$, we get
	\[ |\cJ_X| = (1 \pm \eps_k) d_X m^k. \]
	Using this we can further deduce, by summing over all the $\binom{t}{k-1}$ sets of $k-1$ clusters of $\cP$, that
	\[ |\cJ^{(k-1)}| = (1 \pm \eps_k) \binom{t}{k-1} d_X m^k. \]
	Moreover, letting $X$ range over all the $(k-1)$-sets of clusters of $\cP$, we get from inequality~\eqref{equation:dgvJk-1}
	\begin{align*}
	\sum_{X} |N_{G}(v;\cJ_X) | = \deg_{G}(v;\cJ^{(k-1)}) \ge \left(\delta + \frac{2}{3}\mu \right) |\cJ^{(k-1)}|.
	\end{align*}
	Now we use $N_{\cJ}(v, \mu/3)$ to estimate
		\begin{align*}
	& \left(\delta + \frac{2}{3}\mu \right) |\cJ^{(k-1)}| \\
	& \leq \sum_{X} |N_{G}(v;\cJ_X) |
	\leq \sum_{X \in N_{\cJ}(v, \mu/3)} |\cJ_X| + \sum_{X \notin N_{\cJ}(v, \mu/3)} \frac{\mu}{3} |\cJ_X| \\
	& \leq \left(\left|N_{\cJ}\left(v, \frac{\mu}{3} \right)\right| + \frac{\mu}{3} \left( \binom{t}{k-1} - \left|N_{\cJ}\left(v, \frac{\mu}{3} \right)\right| \right) \right) (1 + \eps_k) d_X m^k \\
	& \leq \left( \left(1 - \frac{\mu}{3}\right) \left|N_{\cJ}\left(v, \frac{\mu}{3} \right)\right| + \frac{\mu}{3} \binom{t}{k-1} \right) \frac{1 + \eps_k}{1 - \eps_k}\frac{|\cJ^{(k-1)}|}{\binom{t}{k-1}} \\
	& \leq \left( \left|N_{\cJ}\left(v, \frac{\mu}{3} \right)\right| + \frac{\mu}{3} \binom{t}{k-1} \right)(1 +2 \eps_k) \frac{|\cJ^{(k-1)}|}{\binom{t}{k-1}}.
	\end{align*}
	Since $\eps_k \ll \mu$, this can be rearranged to get
	\begin{align*}
		\left|N_{\cJ}\left(v, \frac{\mu}{3} \right)\right| \ge \left( \delta + \frac{\mu}{4} \right) \binom{t}{k-1},
	\end{align*}
	as desired.
\end{proof}

The following lemma states that the outcome of Lemma~\ref{lemma:neighboursincomplex} is robust under the deletion of a small amount of vertices.

\begin{lemma} \label{lemma:resilientneighboursincomplex}
	Let $k,r,m,t \in \mathbb{N}$, and $d_2,\dots,d_k,\eps,\eps_k,\mu,\lambda$ be such that
	\begin{align*}
		1/m \ll 1/r,\eps \ll \eps_k,d_2,\dots ,d_{k-1}, \quad \eps_k \ll d_k \leq 1/k, \quad \text{and} \quad  \eps_k \ll \lambda \ll \mu,
	\end{align*}
	Let $\mathbf{d} = (d_2, \dotsc, d_k)$ and let $(G, G_{\cJ},\cJ, \cP, R)$ be a $(k, m, t, \eps, \eps_k, r, \mathbf{d})$-regular setup.
	Let $T \subseteq V(G)$ such that $|Z_1 \cap T| = |Z_2 \cap T| \leq \lambda m$ for each $Z_1, Z_2 \in \cP$.
	Let $Z' = Z \setminus T$ for each $Z \in \cP$, and let $\cJ' = \cJ[\bigcup Z']$ be the induced subcomplex.
	Then, for every $v \in V(G)$,
	\begin{align*}
	|N_{\cJ}(v, 2 \mu)| \leq |N_{\cJ'}(v, \mu)|. 
	\end{align*}
\end{lemma}

\begin{proof}
	Fix $v \in V(G)$ and a $(k-1)$-set $X \in N_{\cJ}(v, 2 \mu)$ and recall that $|N_G(v, \cJ_X)| > 2 \mu |\cJ_X|$ as defined in inequality~\eqref{equation:tupleneighbours}.
	Write $X = \{ X_1, \dotsc, X_{k-1} \}$ and let $X' = \{ X'_1, \dotsc, X'_{k-1} \}$ be the corresponding clusters of $X$ in the complex $\cJ'$.
	To prove the lemma, it is enough to show that $X' \in N_{\cJ'}(v, \mu)$.
	
	Let $\eps \ll \beta \ll \eps_k$ and $d_X = \prod_{i=2}^{k-1} d_i^{\binom{k-1}{i}}$.
	By the corollary (Lemma~\ref{lemma:countinglevelsslices}) of the Dense Counting Lemma, we have
	\[|\cJ_X| = (1 \pm \beta) d_X m^{k-1}   \]
	and therefore
	\begin{align}\label{equ:NvJX-geq-dXmk}
	|N_G(v, \cJ_X)| > 2 \mu |\cJ_X| \ge 2 \mu (1 - \beta) d_X m^{k-1}. 
	\end{align}
	Now let $m' = |X_1 \setminus T|$, so that we  have $|Z'| = m'$ for each $Z \in \cP$, and note that $m' \ge (1 - \lambda) m$.
	By the Slice Restriction Lemma (Lemma~\ref{lem:regular-slice-restriction}), $\cJ'$ is a $(\cdot, \cdot, \sqrt{\eps}, \sqrt{\eps_k}, r)$-regular slice.
	In particular, by applying Lemma~\ref{lemma:countinglevelsslices} in this restricted complex (and as $\sqrt{\eps} \ll \beta$), we obtain that
	\begin{align*}
	(1 + \beta) d_X (m')^{k-1}   &\ge|\cJ'_{X'}|  
	\geq (1 - \beta) d_X (m')^{k-1}  \\
	& \ge (1 - \beta) (1 - \lambda)^{k-1} d_X m^{k-1}.
	\end{align*}
	Combining this with inequality~\eqref{equ:NvJX-geq-dXmk}, results in
	\begin{align*}
	|N_G(v, \cJ'_{X'})|
	& \ge |N_G(v, \cJ_X)| - (|\cJ_X| - |\cJ'_{X'}|) \\
	& \ge (1 - \beta) \left(  2 \mu - ( 1 - (1 - \lambda)^{k-1}) \right) d_X m^{k-1} \\
	& \ge \mu (1 + \beta) d_X m^{k-1} \ge \mu |\cJ'_{X'}|,
	\end{align*}
	where in the second to last inequality we have used that $\beta \ll \eps_k \ll \lambda \ll \mu$.
	Thus $X' \in N_{\cJ'}(v, \mu)$, as required.
\end{proof}

In a (sufficiently long) $k$-uniform tight cycle, the link graph of every vertex corresponds to a $(k-1)$-uniform tight path.
Because of this, we will look for tight paths in the neighbourhoods of vertices inside of a regular complex.
The next lemma states that by looking at a $\mu$-fraction of the $(k-1)$-edges of the regular complex we will find lots of tight paths, even a positive proportion of the paths that the complex already has.
Its proof follows from various applications of the dense counting and extension lemmas.

\begin{lemma} \label{lemma:sidorenkito}
	Let $1/m \ll \eps \ll d_2, \dotsc, d_{k-1}, 1/k, \mu$ and $k \ge 3$.
	Suppose $\cJ$ is a $(\cdot, \cdot, \eps)$-equitable complex with density vector $\mathbf{d} = (d_2, \dotsc, d_{k-1})$ and ground partition $\cP$, with vertex classes of size $m$ each.
	Let $W = \{ W_1, \dotsc, W_{k-1} \} \subseteq \cP$.
	Let $S \subseteq \cJ_{W}$ have size at least $\mu |\cJ_W|$.
	Let $Q$ be a $(k-1)$-uniform tight path on $2k-2$ vertices $v_1 \dotsc v_{{2k-2}}$ whose vertices are partitioned in clusters $\{X_1, \dotsc, X_{k-1} \}$ such that $v_i, v_{i+k-1} \in X_i$ for all $1 \leq i \leq k-1$, and let $\mathcal{Q}$ be the downward-closed \mbox{$(k-1)$-complex} generated from $Q$.
	Recall that $\mathcal{Q}_\mathcal{J}$ is the set of labelled partition-respecting copies of $\mathcal{Q}$ in $\cJ$,
	and let $\mathcal{Q}_\mathcal{S} \subseteq \mathcal{Q}_\mathcal{J}$ correspond to those copies whose edges in the $(k-1)$-th level are in $S$.
	Then \[ |\mathcal{Q}_S| \ge \frac{1}{2} \left( \frac{\mu}{8k} \right)^{k+1} |\mathcal{Q}_\mathcal{J}|. \]
\end{lemma}

\begin{proof}
	The proof consists of three steps, which we sketch now.
	First, we begin by using the dense versions of the counting and extension lemmas to count the number of various hypergraphs inside $\cJ$, including partite edges and tight paths on $2k-2$ vertices.
	In the second step, we remove a few edges from $S$ which have undesirable properties.
	This gives a subgraph $S' \subseteq S$ whose behaviour is somewhat better controlled.
	For instance we know that every partite $(k-1)$-set is contained in either none or many edges of $S'$.  
	Finally, we give an iterative procedure which returns a tight path using the edges of $S'$.
	Using the codegree properties of $S'$ we obtain that there are a lot of choices in every step of the procedure, and each distinct choice yields a different tight path which uses only edges of $S' \subseteq S$.
	Combining this with the counting done in the first step, the lemma follows.
	
	\medskip
	\emph{Step 1: Counting subgraphs.}
	We begin by using the dense versions of the counting and extension lemma to count different hypergraphs inside $\cJ$.
	Let $\beta > 0$ such that $\eps \ll \beta \ll d_2, \dotsc, d_{k-1}, 1/k, \mu$.
	Define
	\begin{align*}
	d_a = \prod_{i=2}^{k-2} d_i^{\binom{k-2}{i}}
	\quad \text{and} \qquad
	d_b = \prod_{i=2}^{k-1} d_i^{\binom{k-2}{i-1}}.
	\end{align*}
	We will express the densities of the subgraphs we will count in terms of $d_a$ and $d_b$.
	Recall that $|\cJ_W|$ denotes the number of edges in the $(k-1)$-th level of $\cJ$.
	Let $W' = W \setminus \{W_{k-1}\}$ and let $\cJ_{W'}$ be the set of edges in the $(k-2)$-th level of $\cJ$ and contained in $\{W_1, \dotsc, W_{k-2}\}$.
	Let $\mathcal{Q}_\mathcal{J}$ be as in the statement.
	Then the Dense Counting Lemma (Lemma~\ref{lemma:densecountinglemma}) implies that
	\begin{align}
	|\cJ_W| & = (1 \pm \beta) d_a d_b m^{k-1}, \label{equation:sidorenkito-jw} \\
	|\cJ_{W'}| & = (1 \pm \beta) d_a m^{k-2} \text{ and} \label{equation:sidorenkito-jww} \\
	|\mathcal{Q}_\mathcal{J}| & = (1 \pm \beta) d_a d_b^{k} m^{2k - 2}. \label{equation:sidorenkito-pj}
	\end{align}
	Since $S \subseteq \cJ_W$ satisfies $|S| \ge \mu |\cJ_W|$, we deduce from inequality~\eqref{equation:sidorenkito-jw} and $\beta \ll 1$ that
	\begin{align}
	|S| & \ge (1 - \beta) \mu d_a d_b m^{k-1}. \label{equation:sidorenkito-jws}
	\end{align}
	Let $B_{W'} \subseteq \cJ_{W '}$ be the set of $(k-2)$-edges which \emph{do not} extend to $(1 \pm \beta) d_b m$ copies of a $(k-1)$-edge in $\cJ_W$.
	The Dense Extension Lemma (Lemma~\ref{lemma:densecountinglemma}) implies that
	\begin{align}
	|B_{W'}| & \leq \beta |\cJ_{W'}|. \label{equation:sidorenkito-smallbad}
	\end{align}
	
	\medskip
	\emph{Step 2: Removing edges of $S$.}
	In the following, we select a subset $S' \subseteq S$ with well-behaved properties.
	We begin by deleting from $S$ all those edges which contain a $(k-2)$-set from $B_{W'}$, to obtain a set $S''$.
	The number of edges deleted is at most
	\begin{align*}
	|B_{W'}|m
	& \leq \beta |\cJ_{W'}| m
	\leq \beta (1 + \beta) d_a m^{k-1}
	\leq |S|/3,
	\end{align*}
	where we have used inequality~\eqref{equation:sidorenkito-jww},~\eqref{equation:sidorenkito-jws},~\eqref{equation:sidorenkito-smallbad} and $\beta \ll d_b, \mu$.
	Thus we deduce that $|S''| \ge 2|S|/3$.
	
	Next, we delete edges from $S''$ using the following procedure:
	If there is any partite $(k-2)$-set $T$ in $\cJ$ which lies in less than $\mu d_b m / (4k)$ edges of $S''$, then delete all edges of $S''$ which contain $T$.
	Iterate until no further deletions are possible.
	We let $S'$ be the resulting set of edges at the end of this procedure.
	Similar to the counting in inequality~\eqref{equation:sidorenkito-jws}, we see that the number of partite $(k-2)$-sets supported in the clusters of $W$ is $(k-1)(1\pm\beta) d_a m^{k-2}$.
	Thus the number of edges deleted in the procedure is at most
	\begin{align*}
	(k-1)(1+\beta) d_a m^{k-2}\frac{\mu d_b m}{4k} &
	\leq (1 + \beta) \frac{1}{4} \mu d_a d_b m^{k-1}
	\leq |S|/3,
	\end{align*}
	where the last inequality follows from inequality~\eqref{equation:sidorenkito-jws} and $\beta \ll 1$.
	This yields $|S'| \ge |S''| - |S|/3 \ge |S|/3$.
	By construction, the set $S'$ has the following crucial property: each partite $(k-2)$-set in $W_1, \dotsc, W_{k-1}$, is either contained in zero edges of $S'$ or in at least $\mu d_b m / (4k)$ edges of $S'$.
	
	\medskip
	\emph{Step 3: Finding many tight paths in $\mathcal{P}_S$.}
	Finally, we use the properties of $S'$ to construct many labelled partition-respecting paths in $\mathcal{Q}_S$.
	We construct paths in the following way:
	\begin{enumerate}[(i)]
		\item Select $T = \{ x_1, \dotsc, x_{k-2} \} \in \cJ_{W'}$ which is contained in at least $(\mu/4) d_b m$ edges in $S'$.
		\item Then, choose $x_{k-1}$ such that $\{ x_{2}, \dotsc, x_{k-1} \} \in S'$ and $x_{k-1}$ is not in $T$.
		\item Finally, for each $i = k, \dotsc, 2k-2$, choose $x_{i}$ such that $\{ x_{i-k+3}, \dotsc, x_{i} \} \in S'$ and $x_i$ is not in $\{ x_1, \dotsc, x_{i-1} \}$.
	\end{enumerate}
	This will construct a tight path on $2k-2$ vertices such that each edge in the $(k-1)$-th level is in $S'$, and thus in $S$.
	Thus it constructs a path in $\mathcal{Q}_S$.
	By counting how many choices we have at every step we will get a lower bound on the size of $|\mathcal{Q}_S|$.
	
	\begin{enumerate}[(i)]
		\item We have at least $(\mu/13) d_a m^{k-2}$ options for $T$ in the first step.
		To see this, let $G \subseteq \cJ_{W'}$ be the set of bad choices, namely the $(k-2)$-sets which are contained in less than $(\mu/4) d_b m$ edges in $S'$.
		Note that by construction every edge in $S'$ contains exactly one set of $\cJ_{W'}$ that is not in $B_{W'}$.
		Thus we have \[ \frac{|S|}{3} \leq |S'| = \sum_{T \in \cJ_{W'}} \deg_{S'}(T) \leq |G| \frac{\mu}{4} d_b m + (|\cJ_{W'}| - |G|) d_b m (1 + \beta), \]
		then rearranging and using inequalities~\eqref{equation:sidorenkito-jww},~\eqref{equation:sidorenkito-jws} and $\beta \ll \mu$ gives that $|G| \leq (1-\beta) (1 - \mu/12) d_a m^{k-2}$, and therefore the number of good choices is at least $|\cJ_{W'}| - |G| \ge (\mu/13) d_a m^{k-2}$, as required.
		
		\item We have at least $(\mu/4) d_b m$ choices for $x_{k-1}$, by the choice of $T$.
		
		\item We have at least $\mu d_b m / (8k)$ choices for $x_i$ in each step $i = k, \dotsc, 2k-2$.
		Indeed, by construction $\{ x_{i - k + 1}, \dotsc, x_{i-1} \}$ is a $(k-2)$-set contained in an edge of $S'$.
		By the crucial property of $S'$, there exist at least $\mu d_b m / (4k)$ vertices $x'_i$ which joined to the $(k-2)$-set form an edge in $S'$.
		Among those possibilities, surely at least $\mu d_b m / (8k)$ are different from all of our previous choices and thus are valid.
	\end{enumerate}
	
	Since each different choice in any step gives a different path, we deduce that the number of paths in $\mathcal{Q}_S$ must be at least \[ \left( \frac{\mu}{13} d_a m^{k-2} \right) \left( \frac{\mu}{4} d_b m \right) \left( \frac{\mu}{8k} d_b m \right)^{k-1} \ge \left( \frac{\mu}{8k} \right)^{k+1} d_a d_b^{k} m^{2k-2} \ge \frac{1}{2} \left( \frac{\mu}{8k} \right)^{k+1} |\mathcal{Q}_\cJ|, \]
	where the last inequality follows from inequality~\eqref{equation:sidorenkito-pj} and $\beta \ll \mu, 1/k$.
\end{proof}

\section{Building the absorbing path} \label{sec:absorbing}

The aim of this section is to show the Absorption Lemma (Lemma~\ref{lem:absorption}).
The standard approach to prove such a lemma for a graph $G$ can be sketched as follows.
We begin by defining absorbing gadgets, which are small subgraphs of $G$ suitable to absorb a particular set $T$ of $k$ vertices.
Next, we show that, for every possible $k$-set $T$, the absorber gadgets which are suitable to absorb $T$ are numerous.
Because of this, we can extract a small family of vertex-disjoint paths which, for every $k$-set of vertices $T$, contains many absorbing gadgets which are suitable for $T$.
Such a family is obtained with high probability by including each gadget in the family independently at random.
Connecting all these gadgets up yields the desired absorbing path $P$.

Our proof follows the above outlined steps together with a few modifications owing to the setting of regular setups.
For instance, we have to ensure that the gadgets are well-integrated in the regular complexes and extend along the reduced graph.
Moreover, due to constrains in the constant hierarchy, we will only by able to find and connect $\Omega(n/t)$ absorbing gadgets at a time, where $n$ is the number of vertices of $G$ and $t$ is the number of clusters.
We therefore have to iterate the process of selecting and connecting gadgets about $1/t$ times to obtain an absorbing path that can absorb $\Omega(n)$ vertices.

The rest of this section is organised as follows.
In the following, subsection, we introduce the definition of absorbing gadgets in the setting of regular setups and give some bounds on their number.
In Subsection~\ref{sec:locally-counting-absorbers}, we count how many of these gadgets are suitable to absorb a particular set of $k$ vertices.
In Subsection~\ref{sec:short-absorbing-paths}, it is shown that one can select a set of well-distributed gadgets, which is able to absorb a small number of arbitrary sets of $k$ vertices.
Finally, in Subsection~\ref{sec:proof-absorption-lemma}, we give a proof of the Absorption Lemma.

\subsection{Globally counting absorbing gadgets}

In this subsection we define absorbing gadgets, explain how to integrate them into a regular setup and estimate their number under suitable degree conditions.

We begin by defining the basic absorbing gadget which will be used to build a full absorbing path for $k$-sets of vertices.
In the following, when concatenating two paths $P$ and $Q$ with a vertex $b$, we omit the parenthesis in $P(b)Q$ and write $PbQ$ instead.
\begin{definition}[Absorbing gadget]\label{def:absorbing-gadget}
	Let $T = \{ t_1, \dotsc, t_k \}$ be a $k$-set of vertices of $G$.
	We say that a subgraph $F \subseteq G$ is an \emph{absorbing gadget for $T$} if $F = A \cup B \cup C \cup \bigcup_{i=1}^k (P_i \cup Q_i)$, where
	\begin{enumerate}[(i)]
		\item $A, B, C, P_1, Q_1,\dots,P_k,Q_k$ are pairwise disjoint tight paths and also disjoint of $T$,
		\item $A, B, C$ have $k$ vertices each, and $AC$ and $ABC$ are tight paths and
		\item for $B = (b_1, \dotsc, b_k)$ and every  $1 \leq i \leq k$, the paths $P_i$ and $Q_i$ have $k-1$ vertices each, and both $P_i b_i Q_i$ and $P_i t_i Q_i$ form tight paths on $2k-1$ vertices.
	\end{enumerate}
\end{definition}
Note that an absorbing gadget $F$ spans $2k+k(2k-1) = k(2k+1)$ vertices.
See Figure~\ref{figure:absorber-full} for a drawing of this structure for $3$-uniform graphs.

To absorb $T$ using an absorbing gadget for $X$, suppose we can find a tight path $P$ which contains as subpaths $AC$ and $P_i b_i Q_i$ for all $1 \leq i \leq k$.
If the need to absorb $T$ arises, then we substitute the segment $AC$ in $P$ by $ABC$ and for all $1 \leq i \leq k$, we substitute the segment $P_i b_i Q_i$ by $P_i t_i Q_i$, to obtain a path $P'$.
Thus $P'$ will be a path with the same ends as $P$, but $V(P') = V(P) \cup T$.
This absorption process is visualised in Figure~\ref{figure:absorber-before} and~\ref{figure:absorber-after}.

%\begin{comment}

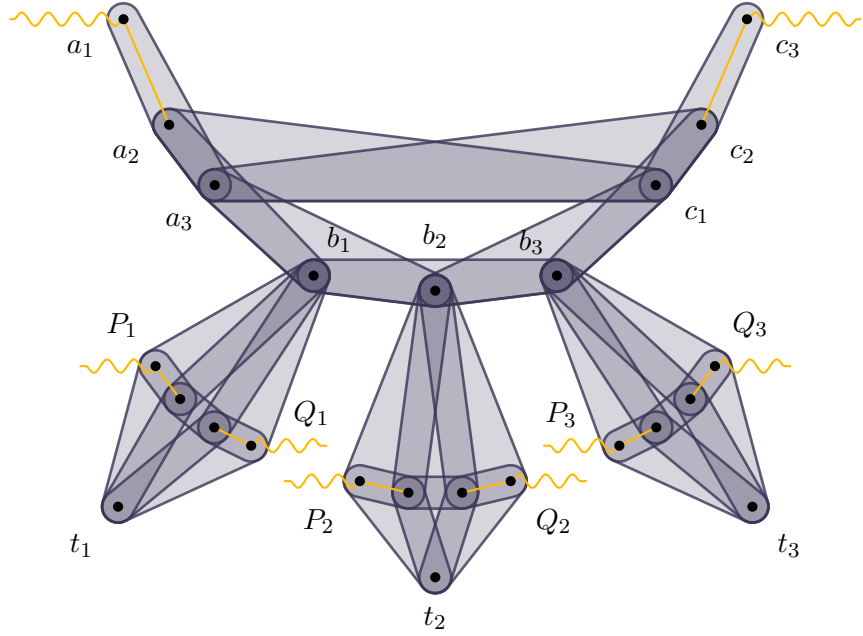
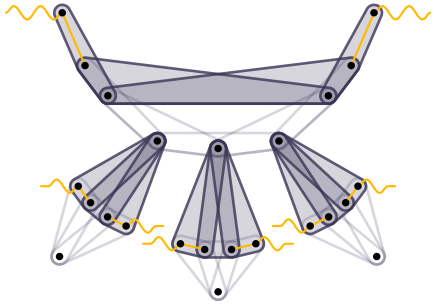
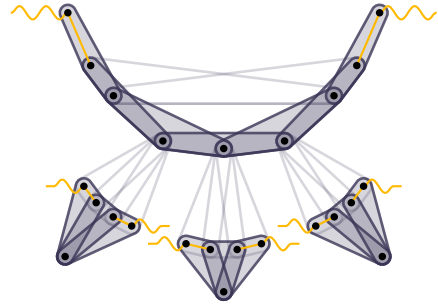
\begin{figure}
	\begin{subfigure}[b]{\textwidth}
		\centering
		\begin{tikzpicture}[thick,scale=1]
		\coordinate (b1) at (-1.6,0);
		\coordinate (b2) at (0,-0.2);
		\coordinate (b3) at (1.6,0);
		
		\coordinate (a1) at (-3.5-0.6,4-0.6);
		\coordinate (a2) at (-3.5,2.6-0.6);
		\coordinate (a3) at (-3.5+0.6,1.8-0.6);
		
		\coordinate (c1) at (3.5-0.6,1.8-0.6);
		\coordinate (c2) at (3.5,2.6-0.6);
		\coordinate (c3) at (3.5+0.6,4-0.6);
		
		\coordinate (x2) at (0,-4);
		\coordinate (p21) at (-110:2.9);
		\coordinate (p22) at (-97:2.9);
		\coordinate (p23) at (-83:2.9);
		\coordinate (p24) at (-70:2.9);
		
		\coordinate (x1) at ($(b1)+(-90-40:4)$);
		\coordinate (p11) at ($(b1)+(-110-40:2.4)$);
		\coordinate (p12) at ($(b1)+(-97-40:2.4)$);
		\coordinate (p13) at ($(b1)+(-83-40:2.4)$);
		\coordinate (p14) at ($(b1)+(-70-40:2.4)$);
		
		\coordinate (x3) at ($(b3)+(-90+40:4)$);
		\coordinate (p31) at ($(b3)+(-110+40:2.4)$);
		\coordinate (p32) at ($(b3)+(-97+40:2.4)$);
		\coordinate (p33) at ($(b3)+(-83+40:2.4)$);
		\coordinate (p34) at ($(b3)+(-70+40:2.4)$);
		
		\coordinate (a1') at ($(a1)-(1.5,0)$);
		\coordinate (c3') at ($(c3)+(1.5,0)$);
		\coordinate (p11') at ($(p11)-(1,0)$);
		\coordinate (p14') at ($(p14)+(1,0)$);
		\coordinate (p21') at ($(p21)-(1,0)$);
		\coordinate (p24') at ($(p24)+(1,0)$);
		\coordinate (p31') at ($(p31)-(1,0)$);
		\coordinate (p34') at ($(p34)+(1,0)$);
		
		\triple{(a3)}{(a2)}{(a1)}{6pt}{1pt}{DarkDesaturatedBlue,opacity=0.8}{DarkDesaturatedBlue,opacity=0.2};
		\triple{(a3)}{(a2)}{(b1)}{6pt}{1pt}{DarkDesaturatedBlue,opacity=0.8}{DarkDesaturatedBlue,opacity=0.2};
		\triple{(b2)}{(b1)}{(a3)}{6pt}{1pt}{DarkDesaturatedBlue,opacity=0.8}{DarkDesaturatedBlue,opacity=0.2};
		\triple{(b3)}{(b2)}{(b1)}{6pt}{1pt}{DarkDesaturatedBlue,opacity=0.8}{DarkDesaturatedBlue,opacity=0.2};
		\triple{(b3)}{(b2)}{(c1)}{6pt}{1pt}{DarkDesaturatedBlue,opacity=0.8}{DarkDesaturatedBlue,opacity=0.2};
		\triple{(b3)}{(c2)}{(c1)}{6pt}{1pt}{DarkDesaturatedBlue,opacity=0.8}{DarkDesaturatedBlue,opacity=0.2};
		\triple{(c3)}{(c2)}{(c1)}{6pt}{1pt}{DarkDesaturatedBlue,opacity=0.8}{DarkDesaturatedBlue,opacity=0.2};
		\triple{(a3)}{(a2)}{(c1)}{6pt}{1pt}{DarkDesaturatedBlue,opacity=0.8}{DarkDesaturatedBlue,opacity=0.2};
		\triple{(a3)}{(c2)}{(c1)}{6pt}{1pt}{DarkDesaturatedBlue,opacity=0.8}{DarkDesaturatedBlue,opacity=0.2};
		
		\triple{(b2)}{(p22)}{(p21)}{6pt}{1pt}{DarkDesaturatedBlue,opacity=0.8}{DarkDesaturatedBlue,opacity=0.2};
		\triple{(b2)}{(p23)}{(p22)}{6pt}{1pt}{DarkDesaturatedBlue,opacity=0.8}{DarkDesaturatedBlue,opacity=0.2};
		\triple{(b2)}{(p24)}{(p23)}{6pt}{1pt}{DarkDesaturatedBlue,opacity=0.8}{DarkDesaturatedBlue,opacity=0.2};
		\triple{(x2)}{(p21)}{(p22)}{6pt}{1pt}{DarkDesaturatedBlue,opacity=0.8}{DarkDesaturatedBlue,opacity=0.2};
		\triple{(x2)}{(p22)}{(p23)}{6pt}{1pt}{DarkDesaturatedBlue,opacity=0.8}{DarkDesaturatedBlue,opacity=0.2};
		\triple{(x2)}{(p23)}{(p24)}{6pt}{1pt}{DarkDesaturatedBlue,opacity=0.8}{DarkDesaturatedBlue,opacity=0.2};
		
		\triple{(b1)}{(p12)}{(p11)}{6pt}{1pt}{DarkDesaturatedBlue,opacity=0.8}{DarkDesaturatedBlue,opacity=0.2};
		\triple{(b1)}{(p13)}{(p12)}{6pt}{1pt}{DarkDesaturatedBlue,opacity=0.8}{DarkDesaturatedBlue,opacity=0.2};
		\triple{(b1)}{(p14)}{(p13)}{6pt}{1pt}{DarkDesaturatedBlue,opacity=0.8}{DarkDesaturatedBlue,opacity=0.2};
		\triple{(x1)}{(p11)}{(p12)}{6pt}{1pt}{DarkDesaturatedBlue,opacity=0.8}{DarkDesaturatedBlue,opacity=0.2};
		\triple{(x1)}{(p12)}{(p13)}{6pt}{1pt}{DarkDesaturatedBlue,opacity=0.8}{DarkDesaturatedBlue,opacity=0.2};
		\triple{(x1)}{(p13)}{(p14)}{6pt}{1pt}{DarkDesaturatedBlue,opacity=0.8}{DarkDesaturatedBlue,opacity=0.2};
		
		\triple{(b3)}{(p32)}{(p31)}{6pt}{1pt}{DarkDesaturatedBlue,opacity=0.8}{DarkDesaturatedBlue,opacity=0.2};			\triple{(b3)}{(p33)}{(p32)}{6pt}{1pt}{DarkDesaturatedBlue,opacity=0.8}{DarkDesaturatedBlue,opacity=0.2};
		\triple{(b3)}{(p34)}{(p33)}{6pt}{1pt}{DarkDesaturatedBlue,opacity=0.8}{DarkDesaturatedBlue,opacity=0.2};
		\triple{(x3)}{(p31)}{(p32)}{6pt}{1pt}{DarkDesaturatedBlue,opacity=0.8}{DarkDesaturatedBlue,opacity=0.2};
		\triple{(x3)}{(p32)}{(p33)}{6pt}{1pt}{DarkDesaturatedBlue,opacity=0.8}{DarkDesaturatedBlue,opacity=0.2};
		\triple{(x3)}{(p33)}{(p34)}{6pt}{1pt}{DarkDesaturatedBlue,opacity=0.8}{DarkDesaturatedBlue,opacity=0.2};
		
		\node at (a1) [label={[label distance=3pt]210:$a_1$}] {};
		\node at (a2) [label={[label distance=3pt]210:$a_2$}] {};
		\node at (a3) [label={[label distance=3pt]240:$a_3$}] {};
		\node at (b1) [label={[label distance=1pt]80:$b_1$}] {};
		\node at (b2) [label={[label distance=8pt]90:$b_2$}] {};
		\node at (b3) [label={[label distance=0pt]100:$b_3$}] {};
		\node at (c1) [label={[label distance=3pt]330:$c_1$}] {};
		\node at (c2) [label={[label distance=3pt]330:$c_2$}] {};
		\node at (c3) [label={[label distance=3pt]330:$c_3$}] {};
		\node at (x1) [label={[label distance=3pt]-130:$t_1$}] {};
		\node at (x2) [label={[label distance=3pt]-90:$t_2$}] {};
		\node at (x3) [label={[label distance=3pt]-50:$t_3$}] {};
		
		\node at (p11) [label={[label distance=3pt]110:$P_1$}] {};
		\node at (p14) [label={[label distance=8pt]15:$Q_1$}] {};
		\node at (p31) [label={[label distance=8pt]165:$P_3$}] {};
		\node at (p34) [label={[label distance=3pt]70:$Q_3$}] {};
		\node at (p21) [label={[label distance=3pt]230:$P_2$}] {};
		\node at (p24) [label={[label distance=3pt]-50:$Q_2$}] {};
		
		\path [draw=PureOrange,snake it] (a1) -- (a1');
		\path [draw=PureOrange] (a1) -- (a2);
		\path [draw=PureOrange,snake it] (c3) -- (c3');
		\path [draw=PureOrange] (c2) -- (c3);
		
		\path [draw=PureOrange,snake it] (p11) -- (p11');
		\path [draw=PureOrange] (p11) -- (p12);
		\path [draw=PureOrange,snake it] (p21) -- (p21');
		\path [draw=PureOrange] (p21) -- (p22);
		\path [draw=PureOrange,snake it] (p31) -- (p31');
		\path [draw=PureOrange] (p31) -- (p32);
		\path [draw=PureOrange,snake it] (p14) -- (p14');
		\path [draw=PureOrange] (p13) -- (p14);
		\path [draw=PureOrange,snake it] (p24) -- (p24');
		\path [draw=PureOrange] (p23) -- (p24);
		\path [draw=PureOrange,snake it] (p34) -- (p34');
		\path [draw=PureOrange] (p33) -- (p34);
		
		\draw (a1') node{};
		\draw (c3') node{};
		
		\tikzstyle{every node}=[circle,draw,fill,inner sep=0pt,minimum width=3pt]
		
		\draw (c1) node{};
		\draw (c2) node{};
		\draw (c3) node{};
		\draw (b1) node{};
		\draw (b2) node{};
		\draw (b3) node{};
		\draw (a1) node{};
		\draw (a2) node{};
		\draw (a3) node{};
		\draw (p21) node{};
		\draw (p22) node{};
		\draw (p23) node{};
		\draw (p24) node{};
		\draw (p11) node{};
		\draw (p12) node{};
		\draw (p13) node{};
		\draw (p14) node{};
		\draw (p31) node{};
		\draw (p32) node{};
		\draw (p33) node{};
		\draw (p34) node{};
		\draw (x1) node{};
		\draw (x2) node{};
		\draw (x3) node{};
		\end{tikzpicture}
		\caption{The full absorbing gadget.}
		\label{figure:absorber-full}
	\end{subfigure}
	\par\bigskip
	\begin{subfigure}[b]{0.45\textwidth}
		\centering
		\begin{tikzpicture}[thick,scale=0.5]
		\coordinate (b1) at (-1.6,0);
		\coordinate (b2) at (0,-0.2);
		\coordinate (b3) at (1.6,0);
		
		\coordinate (a1) at (-3.5-0.6,4-0.6);
		\coordinate (a2) at (-3.5,2.6-0.6);
		\coordinate (a3) at (-3.5+0.6,1.8-0.6);
		
		\coordinate (c1) at (3.5-0.6,1.8-0.6);
		\coordinate (c2) at (3.5,2.6-0.6);
		\coordinate (c3) at (3.5+0.6,4-0.6);
		
		\coordinate (x2) at (0,-4);
		\coordinate (p21) at (-110:2.9);
		\coordinate (p22) at (-97:2.9);
		\coordinate (p23) at (-83:2.9);
		\coordinate (p24) at (-70:2.9);
		
		\coordinate (x1) at ($(b1)+(-90-40:4)$);
		\coordinate (p11) at ($(b1)+(-110-40:2.4)$);
		\coordinate (p12) at ($(b1)+(-97-40:2.4)$);
		\coordinate (p13) at ($(b1)+(-83-40:2.4)$);
		\coordinate (p14) at ($(b1)+(-70-40:2.4)$);
		
		\coordinate (x3) at ($(b3)+(-90+40:4)$);
		\coordinate (p31) at ($(b3)+(-110+40:2.4)$);
		\coordinate (p32) at ($(b3)+(-97+40:2.4)$);
		\coordinate (p33) at ($(b3)+(-83+40:2.4)$);
		\coordinate (p34) at ($(b3)+(-70+40:2.4)$);
		
		\coordinate (a1') at ($(a1)-(1.5,0)$);
		\coordinate (c3') at ($(c3)+(1.5,0)$);
		\coordinate (p11') at ($(p11)-(1,0)$);
		\coordinate (p14') at ($(p14)+(1,0)$);
		\coordinate (p21') at ($(p21)-(1,0)$);
		\coordinate (p24') at ($(p24)+(1,0)$);
		\coordinate (p31') at ($(p31)-(1,0)$);
		\coordinate (p34') at ($(p34)+(1,0)$);
		
		\triple{(a3)}{(a2)}{(a1)}{6pt}{1pt}{DarkDesaturatedBlue,opacity=0.8}{DarkDesaturatedBlue,opacity=0.2};
		\triple{(a3)}{(a2)}{(b1)}{6pt}{1pt}{DarkDesaturatedBlue,opacity=0.2}{DarkDesaturatedBlue,opacity=0};
		\triple{(b2)}{(b1)}{(a3)}{6pt}{1pt}{DarkDesaturatedBlue,opacity=0.2}{DarkDesaturatedBlue,opacity=0};
		\triple{(b3)}{(b2)}{(b1)}{6pt}{1pt}{DarkDesaturatedBlue,opacity=0.2}{DarkDesaturatedBlue,opacity=0};
		\triple{(b3)}{(b2)}{(c1)}{6pt}{1pt}{DarkDesaturatedBlue,opacity=0.2}{DarkDesaturatedBlue,opacity=0};
		\triple{(b3)}{(c2)}{(c1)}{6pt}{1pt}{DarkDesaturatedBlue,opacity=0.2}{DarkDesaturatedBlue,opacity=0};
		\triple{(c3)}{(c2)}{(c1)}{6pt}{1pt}{DarkDesaturatedBlue,opacity=0.8}{DarkDesaturatedBlue,opacity=0.2};
		\triple{(a3)}{(a2)}{(c1)}{6pt}{1pt}{DarkDesaturatedBlue,opacity=0.8}{DarkDesaturatedBlue,opacity=0.2};
		\triple{(a3)}{(c2)}{(c1)}{6pt}{1pt}{DarkDesaturatedBlue,opacity=0.8}{DarkDesaturatedBlue,opacity=0.2};
		
		\triple{(b2)}{(p22)}{(p21)}{6pt}{1pt}{DarkDesaturatedBlue,opacity=0.8}{DarkDesaturatedBlue,opacity=0.2};
		\triple{(b2)}{(p23)}{(p22)}{6pt}{1pt}{DarkDesaturatedBlue,opacity=0.8}{DarkDesaturatedBlue,opacity=0.2};
		\triple{(b2)}{(p24)}{(p23)}{6pt}{1pt}{DarkDesaturatedBlue,opacity=0.8}{DarkDesaturatedBlue,opacity=0.2};
		\triple{(x2)}{(p21)}{(p22)}{6pt}{1pt}{DarkDesaturatedBlue,opacity=0.2}{DarkDesaturatedBlue,opacity=0};
		\triple{(x2)}{(p22)}{(p23)}{6pt}{1pt}{DarkDesaturatedBlue,opacity=0.2}{DarkDesaturatedBlue,opacity=0};
		\triple{(x2)}{(p23)}{(p24)}{6pt}{1pt}{DarkDesaturatedBlue,opacity=0.2}{DarkDesaturatedBlue,opacity=0};
		
		\triple{(b1)}{(p12)}{(p11)}{6pt}{1pt}{DarkDesaturatedBlue,opacity=0.8}{DarkDesaturatedBlue,opacity=0.2};
		\triple{(b1)}{(p13)}{(p12)}{6pt}{1pt}{DarkDesaturatedBlue,opacity=0.8}{DarkDesaturatedBlue,opacity=0.2};
		\triple{(b1)}{(p14)}{(p13)}{6pt}{1pt}{DarkDesaturatedBlue,opacity=0.8}{DarkDesaturatedBlue,opacity=0.2};
		\triple{(x1)}{(p11)}{(p12)}{6pt}{1pt}{DarkDesaturatedBlue,opacity=0.2}{DarkDesaturatedBlue,opacity=0};
		\triple{(x1)}{(p12)}{(p13)}{6pt}{1pt}{DarkDesaturatedBlue,opacity=0.2}{DarkDesaturatedBlue,opacity=0};
		\triple{(x1)}{(p13)}{(p14)}{6pt}{1pt}{DarkDesaturatedBlue,opacity=0.2}{DarkDesaturatedBlue,opacity=0};
		
		\triple{(b3)}{(p32)}{(p31)}{6pt}{1pt}{DarkDesaturatedBlue,opacity=0.8}{DarkDesaturatedBlue,opacity=0.2};
		\triple{(b3)}{(p33)}{(p32)}{6pt}{1pt}{DarkDesaturatedBlue,opacity=0.8}{DarkDesaturatedBlue,opacity=0.2};
		\triple{(b3)}{(p34)}{(p33)}{6pt}{1pt}{DarkDesaturatedBlue,opacity=0.8}{DarkDesaturatedBlue,opacity=0.2};
		\triple{(x3)}{(p31)}{(p32)}{6pt}{1pt}{DarkDesaturatedBlue,opacity=0.2}{DarkDesaturatedBlue,opacity=0};
		\triple{(x3)}{(p32)}{(p33)}{6pt}{1pt}{DarkDesaturatedBlue,opacity=0.2}{DarkDesaturatedBlue,opacity=0};
		\triple{(x3)}{(p33)}{(p34)}{6pt}{1pt}{DarkDesaturatedBlue,opacity=0.2}{DarkDesaturatedBlue,opacity=0};
		
		\path [draw=PureOrange,snake it] (a1) -- (a1');
		\path [draw=PureOrange] (a1) -- (a2);
		\path [draw=PureOrange,snake it] (c3) -- (c3');
		\path [draw=PureOrange] (c2) -- (c3);
		
		\path [draw=PureOrange,snake it] (p11) -- (p11');
		\path [draw=PureOrange] (p11) -- (p12);
		\path [draw=PureOrange,snake it] (p21) -- (p21');
		\path [draw=PureOrange] (p21) -- (p22);
		\path [draw=PureOrange,snake it] (p31) -- (p31');
		\path [draw=PureOrange] (p31) -- (p32);
		\path [draw=PureOrange,snake it] (p14) -- (p14');
		\path [draw=PureOrange] (p13) -- (p14);
		\path [draw=PureOrange,snake it] (p24) -- (p24');
		\path [draw=PureOrange] (p23) -- (p24);
		\path [draw=PureOrange,snake it] (p34) -- (p34');
		\path [draw=PureOrange] (p33) -- (p34);
		
		\draw (a1') node{};
		\draw (c3') node{};
		
		\tikzstyle{every node}=[circle,draw,fill,inner sep=0pt,minimum width=2pt]
		
		\draw (c1) node{};
		\draw (c2) node{};
		\draw (c3) node{};
		\draw (b1) node{};
		\draw (b2) node{};
		\draw (b3) node{};
		\draw (a1) node{};
		\draw (a2) node{};
		\draw (a3) node{};
		\draw (p21) node{};
		\draw (p22) node{};
		\draw (p23) node{};
		\draw (p24) node{};
		\draw (p11) node{};
		\draw (p12) node{};
		\draw (p13) node{};
		\draw (p14) node{};
		\draw (p31) node{};
		\draw (p32) node{};
		\draw (p33) node{};
		\draw (p34) node{};
		\draw (x1) node{};
		\draw (x2) node{};
		\draw (x3) node{};
		\end{tikzpicture}
		\caption{Before the absorption.}
		\label{figure:absorber-before}
	\end{subfigure}
	\begin{subfigure}[b]{0.45\textwidth}
		\centering
		\begin{tikzpicture}[thick,scale=0.5]
		\coordinate (b1) at (-1.6,0);
		\coordinate (b2) at (0,-0.2);
		\coordinate (b3) at (1.6,0);
		
		\coordinate (a1) at (-3.5-0.6,4-0.6);
		\coordinate (a2) at (-3.5,2.6-0.6);
		\coordinate (a3) at (-3.5+0.6,1.8-0.6);
		
		\coordinate (c1) at (3.5-0.6,1.8-0.6);
		\coordinate (c2) at (3.5,2.6-0.6);
		\coordinate (c3) at (3.5+0.6,4-0.6);
		
		\coordinate (x2) at (0,-4);
		\coordinate (p21) at (-110:2.9);
		\coordinate (p22) at (-97:2.9);
		\coordinate (p23) at (-83:2.9);
		\coordinate (p24) at (-70:2.9);
		
		\coordinate (x1) at ($(b1)+(-90-40:4)$);
		\coordinate (p11) at ($(b1)+(-110-40:2.4)$);
		\coordinate (p12) at ($(b1)+(-97-40:2.4)$);
		\coordinate (p13) at ($(b1)+(-83-40:2.4)$);
		\coordinate (p14) at ($(b1)+(-70-40:2.4)$);
		
		\coordinate (x3) at ($(b3)+(-90+40:4)$);
		\coordinate (p31) at ($(b3)+(-110+40:2.4)$);
		\coordinate (p32) at ($(b3)+(-97+40:2.4)$);
		\coordinate (p33) at ($(b3)+(-83+40:2.4)$);
		\coordinate (p34) at ($(b3)+(-70+40:2.4)$);
		
		\coordinate (a1') at ($(a1)-(1.5,0)$);
		\coordinate (c3') at ($(c3)+(1.5,0)$);
		\coordinate (p11') at ($(p11)-(1,0)$);
		\coordinate (p14') at ($(p14)+(1,0)$);
		\coordinate (p21') at ($(p21)-(1,0)$);
		\coordinate (p24') at ($(p24)+(1,0)$);
		\coordinate (p31') at ($(p31)-(1,0)$);
		\coordinate (p34') at ($(p34)+(1,0)$);
		
		\triple{(a3)}{(a2)}{(a1)}{6pt}{1pt}{DarkDesaturatedBlue,opacity=0.8}{DarkDesaturatedBlue,opacity=0.2};
		\triple{(a3)}{(a2)}{(b1)}{6pt}{1pt}{DarkDesaturatedBlue,opacity=0.8}{DarkDesaturatedBlue,opacity=0.2};
		\triple{(b2)}{(b1)}{(a3)}{6pt}{1pt}{DarkDesaturatedBlue,opacity=0.8}{DarkDesaturatedBlue,opacity=0.2};
		\triple{(b3)}{(b2)}{(b1)}{6pt}{1pt}{DarkDesaturatedBlue,opacity=0.8}{DarkDesaturatedBlue,opacity=0.2};
		\triple{(b3)}{(b2)}{(c1)}{6pt}{1pt}{DarkDesaturatedBlue,opacity=0.8}{DarkDesaturatedBlue,opacity=0.2};
		\triple{(b3)}{(c2)}{(c1)}{6pt}{1pt}{DarkDesaturatedBlue,opacity=0.8}{DarkDesaturatedBlue,opacity=0.2};
		\triple{(c3)}{(c2)}{(c1)}{6pt}{1pt}{DarkDesaturatedBlue,opacity=0.8}{DarkDesaturatedBlue,opacity=0.2};
		\triple{(a3)}{(a2)}{(c1)}{6pt}{1pt}{DarkDesaturatedBlue,opacity=0.2}{DarkDesaturatedBlue,opacity=0};
		\triple{(a3)}{(c2)}{(c1)}{6pt}{1pt}{DarkDesaturatedBlue,opacity=0.2}{DarkDesaturatedBlue,opacity=0};
		
		\triple{(b2)}{(p22)}{(p21)}{6pt}{1pt}{DarkDesaturatedBlue,opacity=0.2}{DarkDesaturatedBlue,opacity=0};
		\triple{(b2)}{(p23)}{(p22)}{6pt}{1pt}{DarkDesaturatedBlue,opacity=0.2}{DarkDesaturatedBlue,opacity=0};
		\triple{(b2)}{(p24)}{(p23)}{6pt}{1pt}{DarkDesaturatedBlue,opacity=0.2}{DarkDesaturatedBlue,opacity=0};
		\triple{(x2)}{(p21)}{(p22)}{6pt}{1pt}{DarkDesaturatedBlue,opacity=0.8}{DarkDesaturatedBlue,opacity=0.2};
		\triple{(x2)}{(p22)}{(p23)}{6pt}{1pt}{DarkDesaturatedBlue,opacity=0.8}{DarkDesaturatedBlue,opacity=0.2};
		\triple{(x2)}{(p23)}{(p24)}{6pt}{1pt}{DarkDesaturatedBlue,opacity=0.8}{DarkDesaturatedBlue,opacity=0.2};
		
		\triple{(b1)}{(p12)}{(p11)}{6pt}{1pt}{DarkDesaturatedBlue,opacity=0.2}{DarkDesaturatedBlue,opacity=0};
		\triple{(b1)}{(p13)}{(p12)}{6pt}{1pt}{DarkDesaturatedBlue,opacity=0.2}{DarkDesaturatedBlue,opacity=0};
		\triple{(b1)}{(p14)}{(p13)}{6pt}{1pt}{DarkDesaturatedBlue,opacity=0.2}{DarkDesaturatedBlue,opacity=0};
		\triple{(x1)}{(p11)}{(p12)}{6pt}{1pt}{DarkDesaturatedBlue,opacity=0.8}{DarkDesaturatedBlue,opacity=0.2};
		\triple{(x1)}{(p12)}{(p13)}{6pt}{1pt}{DarkDesaturatedBlue,opacity=0.8}{DarkDesaturatedBlue,opacity=0.2};
		\triple{(x1)}{(p13)}{(p14)}{6pt}{1pt}{DarkDesaturatedBlue,opacity=0.8}{DarkDesaturatedBlue,opacity=0.2};
		
		\triple{(b3)}{(p32)}{(p31)}{6pt}{1pt}{DarkDesaturatedBlue,opacity=0.2}{DarkDesaturatedBlue,opacity=0};
		\triple{(b3)}{(p33)}{(p32)}{6pt}{1pt}{DarkDesaturatedBlue,opacity=0.2}{DarkDesaturatedBlue,opacity=0};
		\triple{(b3)}{(p34)}{(p33)}{6pt}{1pt}{DarkDesaturatedBlue,opacity=0.2}{DarkDesaturatedBlue,opacity=0};
		\triple{(x3)}{(p31)}{(p32)}{6pt}{1pt}{DarkDesaturatedBlue,opacity=0.8}{DarkDesaturatedBlue,opacity=0.2};
		\triple{(x3)}{(p32)}{(p33)}{6pt}{1pt}{DarkDesaturatedBlue,opacity=0.8}{DarkDesaturatedBlue,opacity=0.2};
		\triple{(x3)}{(p33)}{(p34)}{6pt}{1pt}{DarkDesaturatedBlue,opacity=0.8}{DarkDesaturatedBlue,opacity=0.2};
		
		\path [draw=PureOrange,snake it] (a1) -- (a1');
		\path [draw=PureOrange] (a1) -- (a2);
		\path [draw=PureOrange,snake it] (c3) -- (c3');
		\path [draw=PureOrange] (c2) -- (c3);
		
		\path [draw=PureOrange,snake it] (p11) -- (p11');
		\path [draw=PureOrange] (p11) -- (p12);
		\path [draw=PureOrange,snake it] (p21) -- (p21');
		\path [draw=PureOrange] (p21) -- (p22);
		\path [draw=PureOrange,snake it] (p31) -- (p31');
		\path [draw=PureOrange] (p31) -- (p32);
		\path [draw=PureOrange,snake it] (p14) -- (p14');
		\path [draw=PureOrange] (p13) -- (p14);
		\path [draw=PureOrange,snake it] (p24) -- (p24');
		\path [draw=PureOrange] (p23) -- (p24);
		\path [draw=PureOrange,snake it] (p34) -- (p34');
		\path [draw=PureOrange] (p33) -- (p34);
		
		\draw (a1') node{};
		\draw (c3') node{};
		
		\tikzstyle{every node}=[circle,draw,fill,inner sep=0pt,minimum width=2pt]
		
		\draw (c1) node{};
		\draw (c2) node{};
		\draw (c3) node{};
		\draw (b1) node{};
		\draw (b2) node{};
		\draw (b3) node{};
		\draw (a1) node{};
		\draw (a2) node{};
		\draw (a3) node{};
		\draw (p21) node{};
		\draw (p22) node{};
		\draw (p23) node{};
		\draw (p24) node{};
		\draw (p11) node{};
		\draw (p12) node{};
		\draw (p13) node{};
		\draw (p14) node{};
		\draw (p31) node{};
		\draw (p32) node{};
		\draw (p33) node{};
		\draw (p34) node{};
		\draw (x1) node{};
		\draw (x2) node{};
		\draw (x3) node{};
		\end{tikzpicture}
		\caption{After the absorption.}
		\label{figure:absorber-after}
	\end{subfigure}
	\caption{An absorbing gadget for $T = \{t_1, t_2, t_3\}$.
		In this case, $k = 3$.
		The paths $AC = (a_1, a_2, a_3, c_1, c_2, c_3)$ and $ABC = (a_1, a_2, a_3, b_1, b_2, b_3, c_1, c_2, c_3)$ are shown, as well as the paths $P_i, Q_i$ for all $1 \leq i \leq 3$.
		The orange pairs with wiggly lines attached indicate extensible pairs.
		In practice, the absorbing gadget will be embedded in a single path, with all the different substructures connected via the extensible pairs.
		Initially, the paths $AC$ as well as $P_1 b_1 Q_1$, $P_2 b_2 Q_2$ and $P_3 b_3 Q_3$ form part of a single path, as seen in Figure~\ref{figure:absorber-before}.
		To absorb $T$, we substitute each $P_i b_i Q_i$ for $P_i t_i Q_i$, and $AC$ by $ABC$, as seen in Figure~\ref{figure:absorber-after}.
	}
	\label{figure:absorber}
\end{figure}

%\end{comment}

Next, we relate absorbing structure to the setting of regular slices.

\begin{definition}[$\fS$-gadget, extensible] \label{definition:absorberinslice}
	Suppose $F = A \cup B \cup C \cup \bigcup_{i=1}^k (P_i \cup Q_i)$ is an absorbing gadget with $A = (a_1 ,\dots, a_k)$, $B = (b_1,\dots, b_k)$, $C = (c_1,\dots, c_k)$ and $P_i = (p_{i,1},\dots, p_{i,k-1})$ and $Q_i = (q_{i,1} ,\dots, q_{i,k-1})$ for all $1 \leq i \leq k$.
	Suppose that $\eps, \eps_k, d_2, \dotsc, d_k, c, \nu >0$.
	Let $\mathbf{d} = (d_2, \dotsc, d_k)$ and suppose $\fS = (G, G_\cJ, \cJ, \cP, \ori{H})$ is an oriented $(k,m,t, \eps , \eps_k, r , \mathbf{d})$-regular setup.  
	We say that $F$ is an \emph{$\fS$-gadget} if
	\begin{enumerate}[(G1)]
		\item \label{item:gadget-bulk} there exists an oriented edge $Y = (Y_1, \dotsc, Y_k) \in \ori{H}$ such that $a_i, b_i, c_i \in Y_i$, for all $1 \leq i \leq k$,
		\item \label{item:gadget-order} for all $1 \leq i \leq k$, there exists an ordered $(k-1)$-tuple of clusters $W_i = (W_{i,1}, \dotsc, W_{i,k-1})$ which is consistent with $\ori{H}$ and such that $W_i \cup \{ Y_i \}$ is an (unordered) edge in $H$,
		\item \label{item:gadget-neighbours} $p_{i,j}, q_{i,j} \in W_{i,j}$, for all $1 \leq i \leq k$, $1 \leq j \in k-1$ and
		\item \label{item:gadget-subgraphs} $F \subseteq G_{\cJ}$.
	\end{enumerate}
	We will further say that $F$ is \emph{$(c, \nu)$-extensible} if the following also holds:
	\begin{enumerate}[(G1),resume]
		\item \label{item:gadget-extensible} The path $AC$ is $(c, \nu, V(G))$-extensible both left- and rightwards to the ordered tuple $(Y_1, \dotsc, Y_k)$
		and for each $1 \leq i \leq k$, the path $P_i b_i Q_i$ is $(c, \nu, V(G))$-extensible leftwards to $(W_{i,1}, \dotsc, W_{i,k-1}, Y_i)$ and rightwards to $(Y_i, W_{i,1}, \dotsc, W_{i,k-1})$.
	\end{enumerate}
\end{definition}

Note that under this definition, the absorbing gadget is not associated with any $k$-set: the whole point is that a single absorbing gadget will be useful for many $k$-sets at once.
So the absorbers in this definition should be understood as potential absorbers which moreover lie in the correct locations inside the regular slice.
The latter will allow us to join the gadgets into a single path.
We note that gadgets are required to lie inside the subgraph $G_{\cJ} \subseteq G$ which will be convenient whenever we use the regularity techniques to find them.
Finally, it is worth remarking that the sets to be absorbed by an absorbing gadget will not necessarily be contained in the union of the clusters of $\cJ$.

We now count how many extensible $\fS$-gadgets are inside a given regular setup.
It will be natural to count the gadgets in three steps: first we will count the ways to choose the clusters in $H$ which will house the absorbers (corresponding to the clusters $Y, W_i$ in Definition~\ref{definition:absorberinslice}).
Then we will count the actual gadgets in $G_\cJ$ (corresponding to the choice of $A, B, C, P_i, Q_i$) which can be found inside each those clusters.
Lastly, we show that the number of non-extensible gadgets only form a small proportion of the total of gadgets.

Recall that an {orientation} of a $k$-graph $H$ is a family of ordered $k$-tuples $\{ \ori{e} \in V(H)^k\colon e \in H \}$, one for each edge in $H$, such that $\ori{e}$ consists of an ordering of the vertices of $e$.
Moreover, we defined a $k$-tuple $X$ to be {consistent with $\ori{H}$}, if $ \ori{H}$ contains a cyclic shift of $X$.
Finally, we said that closed tight walk {$W$} in $H$ is {compatible} with $\ori{H}$, if $W$ visits every edge of $H$ at least once, and each oriented edge of $\ori{H}$ appears at least once in ${W}$ as a sequence of $k$ consecutive vertices.

The following definition corresponds to the structures that we look in the reduced graph.

\begin{definition}[Reduced gadget] \label{definition:reducedgadget}
	A \emph{reduced gadget} is a $k$-graph $L$ consisting on $k^2$ labelled vertices $Y \cup W_1 \cup \dotsb \cup W_k$, where $Y = \{ Y_1, \dotsc, Y_k \}$, and $W_i = \{ W_{i,1}, \dotsc, W_{i,{k-1}} \}$ for all $1 \leq i \leq k$, and $k+1$ edges given by $Y$ and $W_i \cup \{ Y_i \}$ for all $1 \leq i \leq k$.
	We will refer to $Y$ as the \emph{core} edge of the reduced gadget, and to $W_1, \dotsc, W_{k}$ as the \emph{peripheral sets} of the reduced gadget.
	
	Given an oriented $k$-graph $\ori{H}$, a \emph{reduced gadget in $\ori{H}$} is a copy of $L$ in $H$ such that the labelling of its core edge $Y$ coincides with the orientation of that edge in $\ori{H}$, and such that for each $1 \leq i \leq k$, the $k$-tuple $(W_{i,1}, \dotsc, W_{i,k-1}, Y_i)$ is consistent with $\ori{H}$.
	We let $\mathscr{L}_{\ori{H}}$ be the set of all the reduced gadgets in $\ori{H}$.
\end{definition}

The next definition capture the whole set of gadgets we can find in a regular setup by looking at all of the possible reduced gadgets.

\begin{definition} \label{definition:globalgadgets}
	Let $\fS = (G, G_\cJ, \cJ, \cP, \ori{H})$ be an oriented regular setup,
	let $c, \nu > 0$
	and let $L \in \mathscr{L}_{\ori{H}}$ be a reduced gadget in $\ori{H}$.
	We define the following sets:
	\begin{enumerate}
		\item $\sF_L$ is the set of $\fS$-gadgets which use precisely the clusters of $L$ as in Definition~\ref{definition:absorberinslice},
		\item $\sF^\ext_L \subseteq \sF_L$ is the set of $\fS$-gadgets which are $(c, \nu)$-extensible,
		\item $\sF$ is the set of all $\fS$-gadgets, and
		\item $\sF^\ext$ is the set of all $(c, \nu)$-extensible $\fS$-gadgets.
	\end{enumerate}
	We remark that the definitions of $\sF_L$, $\sF^\ext_L$, $\sF$ and $\sF^\ext$ also depend on $c, \nu, G, G_{\cJ}, \cJ, \cP$ and $\ori{H}$, but we suppress them from the notation since they will always be clear in the context.
\end{definition}

The following lemma estimates the sizes of $|\sF_L|$ and $|\sF^\ext_L|$ for a given choice of a reduced gadget $L$, and essentially says that those two values are always close to each other.

\begin{lemma} \label{lemma:countinggadgets-extensibleornot}
	Let $k,r,m,t \in \mathbb{N}$, and $d_2,\dots,d_k,\eps,\eps_k, c,\nu,\beta$ be such that
	\begin{align*}
	1/m & \ll 1/r,\eps \ll 1/t, c, \eps_k, d_2,\dots ,d_{k-1}, \\
	c & \ll d_2,\dots ,d_{k-1}, \\
	1/t & \ll \eps_k \ll \beta, d_k \leq 1/k, \quad \text{and} \quad \eps_k \ll \nu.
	\end{align*}
	Let $\mathbf{d} = (d_2, \dotsc, d_k)$,
	let $(G, G_\cJ,\cJ, \cP, \ori{H})$ be an oriented $(k,m,t, \eps , \eps_k, r , \mathbf{d})$-regular setup 
	and let $L \in \mathscr{L}_{\ori{H}}$ be a reduced gadget in $\ori{H}$.
	Let $\cF$ be the $k$-complex corresponding to the down-closure of $k$-graph $F$ as in Definition~\ref{definition:absorberinslice}.
	Then
	\begin{align}
	|\sF_L| = (1 \pm \beta) \left( \prod_{i=1}^{k} d_i^{e_i(\cF)} \right) m^{k(2k+1)}
	\label{equation:countinggadgets-extensibleornot-A}
	\end{align}
	and
	\[ |\sF_L \setminus \sF^\ext_L| \leq \beta |\sF_L|. \]
\end{lemma}

\begin{proof}
	Let $Y = (Y_1, \dotsc, Y_k) \in \ori{H}$ denote the ordered core edge of $L$ and, for each $1 \leq i \leq k$, let $W_i = \{ W_{i,1}, \dotsc, W_{i,k-1} \}$ be the peripheral sets of $L$, ordered  such that $(W_{i,1}, \dotsc, W_{i,k-1}, Y_i)$ is consistent with $\ori{H}$.
	Note that $|V(F)| = k(2k+1)$.
	The bounds on $|\sF_L|$ are given by a straightforward application of the Counting Lemma (Lemma~\ref{lemma:counting}) inside the clusters $Y \cup W_1 \cup \dotsb \cup W_k$, and thus we omit more details.
	
	To prove that $|\sF_L \setminus \sF^\ext_L| \leq \beta |\sF_L|$, we will estimate the number of non-extensible gadgets in $\sF_L$.
	We sketch the proof first.
	Note that $F \in \sF_L$ is extensible if and only if a set of $k+1$ paths contained in $F$ (corresponding to $AC$ and $P_i b_i Q_i$ for $1 \leq i \leq k$) are extensible with respect to certain edges of the reduced graph $H$.
	This means that a set of $2(k+1)$ many $(k-1)$-tuples (two for each path, corresponding to its endpoints) are extensible according to certain edges of the reduced graph.
	To prove the result we show that the proportion of $F \in \cF^\ast_L$ which are spoiled by a single non-extensible $(k-1)$-tuple corresponding to $F$ is small.
	This will follow from an application of Proposition~\ref{prop:extpath} and the Dense Extension Lemma.
	Repeating this bound over all the $(k-1)$-sets that we need to care about, we get the result.
	Now we turn to the details.
	
	We begin by showing that the number of copies $F \in \sF_L$ such that the vertices corresponding to the starting end of $AC$ are not extensible is a small proportion of $|\sF_L|$.
	
	Let $Y' = (Y_1, \dotsc, Y_{k-1})$ and denote the ordered tuples in the $(k-1)$-th level of $\cJ$ in the clusters $\{Y_1, \dotsc, Y_{k-1} \}$  by $\cJ_{Y'} = \cJ_{(Y_1, \dotsc, Y_{k-1})}$.
	Let $d_{Y'} = \prod_{i=2}^{k-1} d_i^{\binom{k-1}{i}}$.
	By Lemma~\ref{lemma:countinglevelsslices} we have
	\begin{align}
	|\cJ_{Y'}| = (1 \pm \beta) d_{Y'} m^{k-1}. \label{equation:countinggadgets-extensibleornot-j}
	\end{align}
	
	Let $\beta_1$ be such that $\eps_k \ll \beta_1 \ll \beta, d_k, 1/k$.
	Let $B_1 \subseteq \cJ_{Y'}$ be the set of $(k-1)$-tuples which are not $(c, \nu, V(G))$-extensible leftwards to $(Y_1, \dotsc, Y_k)$.
	Using Proposition~\ref{prop:extpath} with $\beta_1$ playing the role of $\beta$, we deduce that
	\begin{align}
	|B_1| \leq \beta_1 |\cJ_{Y'}|. \label{equation:countinggadgets-extensibleornot-b1}
	\end{align}
	
	Now we prepare the setting for an application of the Dense Extension Lemma.
	Let $\beta_2$ be such that $\eps \ll \beta_2 \ll \eps_k, d_2, \dotsc, d_{k-1}$.
	Let $\phi\colon V(F) \rightarrow L$ be the obvious hypergraph homomorphism associated with $F$ and the reduced gadget $L$.
	Let $Z \subseteq V(F)$ correspond to the first $k-1$ vertices of the path $AC$ in $F$ (so $Z$ corresponds to the vertices $\{a_1, \dotsc, a_{k-1}\}$).
	Let $\mathcal{F}^-$ be the $(k-1)$-complex generated by the down-closure of $F$ and then removing the $k$-th layer.
	Let $\mathcal{Z} = \mathcal{F}^{-}[Z]$ be the induced subcomplex of $\mathcal{F}^{-}$ in $Z$.
	Note that $\mathcal{Z}$ consists of the down-closure of a $(k-1)$-edge and $\phi(a_i) = Y_i$ for all $1 \le i\le k-1$.
	Thus the labelled partition-respecting copies of $\mathcal{Z}$ in $\cJ$ correspond exactly to $\cJ_{Y'}$.
	Define \[d_{F^{-} \setminus \mathcal Z} = \prod_{i=2}^{k-1} d_i^{e_i(\mathcal{F}^-) - e_i(\mathcal{Z})}.\]
	Let $B_2 \subseteq \cJ_{Y'}$ be the set of $(k-1)$-tuples which do not extend to \[ (1 \pm \beta_2) d_{F^{-} \setminus \mathcal Z} m^{|V(F)| - (k-1)} \] labelled partition-respecting copies of $\mathcal{F}^-$ in $\cJ$.
	An application of the Dense Extension Lemma (Lemma~\ref{lemma:denseextensionlemma}) with $\beta_2$ playing the role of $\beta$ gives that
	\begin{align}
	|B_2| \leq \beta_2 |\cJ_{Y'}|. \label{equation:countinggadgets-extensibleornot-b2}
	\end{align}
	
	From $\beta_2 \ll d_2, \dotsc, d_{k-1}, \eps_k$, we deduce
	\begin{align}
	\beta_1 (1 + \beta_2) d_{F^{-} \setminus \mathcal Z} + \beta_2 \leq 3 \beta_1 d_{F^{-} \setminus \mathcal Z}, \label{equation:countinggadgets-extensibleornot-dversusb1}
	\end{align}
	and from inequality~\eqref{equation:countinggadgets-extensibleornot-A} together with the definition of $d_{Y'}$ and $d_{F^{-} \setminus \mathcal Z}$ we get (since $|V(F)| = k(2k+1)$)
	\begin{align}
	|\sF_L| = (1 \pm \beta) d_k^{e_k(\cF)} d_{F^{-} \setminus \mathcal Z} d_{Y'} m^{|V(F)|}. \label{equation:countinggadgets-extensibleornot-dFdZ}
	\end{align}
	
	Let $\cG = \cJ \cup G_{\cJ}$.
	Say that a labelled partition-respecting copy of $\cF$ in $\cG$ is \emph{good} if the vertices corresponding to $\{a_1, \dotsc, a_{k-1}\}$ are \emph{not} in $B_1 \cup B_2$.
	We will find an upper bound on the number of copies which are not good.
	For every $Z$ in $\cJ_{Y'}$, let $N^*(Z)$ be the number of labelled partition-respecting copies of $\cF$ in $\cG$ which extend $Z$.
	For every $Z \in \cJ_{Y'}$, $N^*(Z)$ is between $0$ and $m^{|V(F)| - (k-1)}$.
	Crucially, for every $Z \in \cJ_{Y'} \setminus B_2$ we also have $N^*(Z) \leq (1 + \beta_2) d_{F^{-} \setminus \mathcal Z} m^{|V(F)| - (k-1)}$ (because each labelled partition-respecting copy of $F$ is supported in a labelled partition-respecting copy of $\mathcal{F}^-$).
	Therefore, the number of copies which are not good is at most
	\begin{align*}
	\sum_{Z \in B_1 \cup B_2} N^*(Z)
	& = \sum_{Z \in B_1 \setminus B_2} N^*(Z) + \sum_{Z \in B_2} N^*(Z) \\
	& \leq \left[ |B_1| (1 + \beta_2) d_{F^{-} \setminus \mathcal Z} + |B_2| \right] m^{|V(F)| - (k-1)} \\
	& \stackrel{\eqref{equation:countinggadgets-extensibleornot-b1},~ \eqref{equation:countinggadgets-extensibleornot-b2}}{\leq}
	\left[ \beta_1 (1 + \beta_2) d_{F^{-} \setminus \mathcal Z} + \beta_2 \right] |\cJ_{Y'}| m^{|V(F)| - (k-1)} \\
	& \stackrel{\eqref{equation:countinggadgets-extensibleornot-dversusb1}}{\leq}
	3 \beta_1  d_{F^{-} \setminus \mathcal Z} |\cJ_{Y'}| m^{|V(F)| - (k-1)} \\
	& \stackrel{\eqref{equation:countinggadgets-extensibleornot-j}}{\leq}
	3 \beta_1 (1 + \beta) d_{F^{-} \setminus \mathcal Z} d_{Y'} m^{|V(F)|} \\
	& \stackrel{\eqref{equation:countinggadgets-extensibleornot-dFdZ}}{\leq} \frac{3 \beta_1(1 + \beta)}{(1 - \beta) d_k^{e_k(F)}} |\sF_L| \leq \frac{\beta}{2k+2} |\sF_L|,
	\end{align*}
	where we used the indicated inequalities and $\beta_1 \ll \beta, d_k, 1/k$ in the last step.
	
	Essentially the same argument shows that if we define (in the obvious way) good tuples for any $(k-1)$-set of vertices of $F$, the number of copies of $F$ which are not good with respect to that $(k-1)$-set is at most $\beta |\sF_L| / ( 2k+2 )$ as well.
	As sketched before, $F \in \sF_L \setminus \sF^\ext_L$ implies that $F$ is not good with respect to one of $2k+2$ many $(k-1)$-sets.
	Thus we get
	\begin{align*}
	|\sF_L \setminus \sF^\ext_L|
	\leq (2k+2) \beta |\sF_L| / ( 2k+2 )
	= \beta |\sF_L|,
	\end{align*}
	as desired.
\end{proof}

The following lemma bounds the sizes of $\mathscr{L}_{\ori{H}}$, $\sF^\ext_L$ and $\sF^\ext$.

\begin{lemma} \label{lemma:countinggadgets-global}
	Let $k,r,m,t \in \mathbb{N}$, and $d_2,\dots,d_k,\eps,\eps_k,c,\nu,\beta, \mu$ be such that
	\begin{align*}
	1/m & \ll 1/r,\eps \ll 1/t, c, \eps_k, d_2,\dots ,d_{k-1}, \\
	c & \ll d_2,\dots ,d_{k-1}, \\
	1/t & \ll \eps_k \ll \beta, d_k \leq 1/k, \quad \text{and} \quad \eps_k \ll \nu, \mu.
	\end{align*}
	Let $\mathbf{d} = (d_2, \dotsc, d_k)$, and
	let $(G, G_\cJ,\cJ, \cP, \ori{H})$ be an oriented $(k,m,t, \eps , \eps_k, r , \mathbf{d})$-regular setup.
	Suppose that $H$ has minimum relative $1$-degree at least $\mu$.
	Then
	\[ \frac{\mu^{k+1}}{2} \binom{t}{k} \binom{t}{k-1}^k \leq |\mathscr{L}_{\ori{H}}| \leq \binom{t}{k} \binom{t}{k-1}^k. \]
	Let $\cF$ be the $k$-complex corresponding to the down-closure of the $k$-graph $F$ as in Definition~\ref{definition:absorberinslice}.
	For each reduced gadget $L \in \mathscr{L}_{\ori{H}}$ in $\ori{H}$, we have
	\[ |\sF^\ext_L| = (1 \pm \beta) \left( \prod_{i=1}^{k} d_i^{e_i(\cF)} \right) m^{k(2k+1)} \]
	and moreover
	\[ |\sF^\ext| = \left( 1  \pm \beta \right) \left( \prod_{i=1}^{k} d_i^{e_i(\cF)} \right) m^{k(2k+1)} |\mathscr{L}_{\ori{H}}|. \]
\end{lemma}

\begin{proof}
	The lower bound on $|\mathscr{L}_{\ori{H}}|$ follows from choosing any edge in $\ori{H}$ to play the role of $Y$ and then greedily choosing $(k-1)$-tuples of neighbours, one for each cluster in $Y$, to play the roles of $W_i$ for all $1 \leq i \leq k$.
	The choice of $Y$ can be done in at least $\mu \binom{t}{k}$ ways, since $H$ has minimum relative $1$-degree at least $\mu$ and thus edge density at least $\mu$.
	Owing to the minimum $1$-degree, each choice of $W_i$ can in principle be done in $\mu \binom{t-1}{k-1}$ ways,
	but we need to subtract from the possible choices those that would reuse clusters already used in another tuple $W_i$.
	We just hide those error terms using $1/t \ll \mu, 1/k$ together with the factor of $1/2$ which appears in the lower bound.
	The upper bound is very crude, just assuming we can choose any $k$-set as $Y$ and any other $(k-1)$-set as $W_i$ for all $i$.
	
	Choose $\beta'$ such that $\eps_k \ll \beta' \ll \beta, d_k, 1/k$.
	The bounds on $|\sF^\ext_L|$ follows from applying Lemma~\ref{lemma:countinggadgets-extensibleornot} (with $\beta'$ in place of $\beta$).
	We get $|\sF_L| \ge |\sF^\ext_L| \ge (1 - \beta') |\sF_L|$, and using $\beta' \ll \beta$ gives the desired bounds.
	It is easy to see that
	\[ \sF^\ext = \bigcup_{L \in \mathscr{L}_{\ori{H}}} \sF^\ext_L,\]
	and the union is disjoint.
	Then the bounds on $|\sF^\ext|$ follow from taking the disjoint union of $\sF^\ext_L$ over all the possible choices of $L$.
\end{proof}

\subsection{Locally counting absorbing gadgets}\label{sec:locally-counting-absorbers}
In this subsection, we count the number of absorbing gadgets which are suitable to absorb a given $k$-set $X$.
We will use the representativity of regular setups (and slices) to argue that we can find well-supported  edges inside a regular complex, which then can be completed to many $\fS$-gadgets.

Suppose $\eps, \eps_k, d_2, \dotsc, d_k, c, \nu > 0$ and let $\mathbf{d} = (d_2, \dotsc, d_{k-1})$.
Suppose $\fS = (G,G_\cJ,\cJ, \cP, \ori{H})$ is an oriented $(k,m,t, \eps , \eps_k, r , \mathbf{d})$-regular setup.
Given a $k$-set $T \subseteq V(G)$, we let $\sF_T \subseteq \sF$ be the set of $\fS$-gadgets which are suitable for absorbing $T$ (as in Definition~\ref{def:absorbing-gadget}).
We also define $\sF_T^\ext$ as the number of these $\fS$-gadgets that are $(c, \nu)$-extensible, namely $\sF_T^\ext = \sF_T \cap \sF^\ext$.
Finally, recall that $N_{\cJ}(v, \delta) = \{ X \subseteq \cP\colon |X| = k-1,~ |N_G(v; \cJ_X)| > \delta |\cJ_X| \}$.

The next lemma states that for any $k$-set $T$, a lower bound on the number of extensible $\fS$-gadgets $\sF^\ext$ translates into a lower bound for  the number of extensible $\fS$-gadgets suitable for $T$.

\begin{lemma} \label{lemma:countinggadgets-local}
	Let $k,r,m,t \in \mathbb{N}$, and $d_2,\dots,d_k,\eps,\eps_k, c, \nu,\mu, {\theta}$ be such that 
	\begin{align*}
	1/m & \ll 1/r,\eps \ll 1/t, c, \eps_k, d_2,\dots ,d_{k-1}, \\
	c & \ll d_2,\dots ,d_{k-1}, \\
	1/t & \ll \eps_k \ll d_k \leq 1/k, \text{and} \\
	\eps_k & \ll \nu \ll  {\theta} \ll \mu \ll 1/k.
	\end{align*}
	Let $\mathbf{d} = (d_2, \dotsc, d_k)$, and
	let $(G, G_\cJ,\cJ, \cP, \ori{H})$ be an oriented $(k,m,t, \eps , \eps_k, r , \mathbf{d})$-regular setup.
	Suppose that $H$ has edge density at least $\mu$, and for every $v \in V(G)$ and every $Z \in \cP$, we have $|N_{\cJ}(v, \mu) \cap N_H(Z) | \ge \mu \binom{t}{k-1}$.
	Let $T \subseteq V(G)$ be a $k$-set.
	Then
	\[ |\sF_T^\ext| \ge  {\theta} |\sF^\ext|. \]
\end{lemma}

The rest of this subsection is dedicated to the proof of Lemma~\ref{lemma:countinggadgets-local}.
We will proceed in three steps.
First, we will show that there are many reduced gadgets in $\ori{H}$ which are suitable for $T$, roughly meaning that these reduced gadgets connect properly with the neighbours of $T$ in $\cJ$.
Secondly, we show that each reduced gadget which is good for $X$ yields many $\fS$-gadgets in $\sF_T$.
Combined with the first step, this yields that $\sF_T$ is large.
In the last step we will sort out the non-extensible gadgets and this will yield the required bound.

We proceed with the first step, making our definitions precise.
Suppose $T = \{t_1, \dotsc, t_k\}$ is a $k$-set in $V(G)$.
Given the family $\mathscr{L}_{\ori{H}}$ of all reduced absorbers in $\ori{H}$ and $\mu > 0$,
define the set $\mathscr{L}_{\ori{H}, T, \mu} \subseteq \mathscr{L}_{\ori{H}}$ of \emph{reduced $(T, \mu)$-absorbers} as the set of absorbers $Y \cup W_1 \cup \dotsb \cup W_k$ such that $W_i \subseteq N_{\cJ}(t_i, \mu)$ for all $1 \leq i \leq k$.
Equivalently, this means that $i$-th peripheral set $W_i$ lies inside $N_{\cJ}(t_i, \mu)$.
We begin by showing that $|\mathscr{L}_{\ori{H}, T, \mu}|$ is large.

\begin{lemma} \label{lemma:countinggadgets-local-reduced}
	Let $k,r,m,t \in \mathbb{N}$, and $d_2,\dots,d_k,\eps,\eps_k,c,\nu,\mu, \theta$ be such that  
	\begin{align*}
	1/m & \ll 1/r,\eps \ll 1/t, c, \eps_k, d_2,\dots ,d_{k-1}, \\
	c & \ll d_2,\dots ,d_{k-1}, \\
	1/t& \ll \eps_k \ll d_k \leq 1/k, \text{ and} \\
	\eps_k & \ll \nu \ll  {\theta} \ll \mu \ll 1/k.
	\end{align*}
	Let $\mathbf{d} = (d_2, \dotsc, d_k)$ and
	let $(G, G_\cJ,\cJ, \cP, \ori{H})$ be an oriented $(k,m,t, \eps , \eps_k, r , \mathbf{d})$-regular setup.
	Suppose that $H$ has edge density at least $\mu$,
	and that for every $v \in V(G)$ and every $Z \in \cP$, we have $|N_{\cJ}(v, \mu) \cap N_H(Z) | \ge \mu \binom{t}{k-1}$.
	Let $T \subseteq V(G)$ be a $k$-set.
	Then
	\[ |\mathscr{L}_{\ori{H}, T, \mu}| \ge  {\theta} |\mathscr{L}_{\ori{H}}|. \]
\end{lemma}

\begin{proof}
	Let $T = \{ t_1, \dotsc, t_k \}$.
	We show a procedure to construct reduced absorbers in $\ori{H}$ which are also reduced $(T, \mu)$-absorbers,
	then we count how many choices we have at every step.
	Since each possible sequence of choices will give a different reduced $(T, \mu)$-absorber, the desired lower bound will follow.
	
	Let $Y = (Y_1, \dotsc, Y_k)$ be an arbitrary oriented edge in $\ori{H}$.
	By assumption, for every $1 \leq i \leq k$, 
	we have $| N_H(Y_i) \cap N_{\cJ}(t_i, \mu) | \ge \mu \binom{t}{k-1}$.
	Thus we can greedily choose pairwise disjoint $(k-1)$-sets of clusters $W_1, \dotsc, W_k$ such that $W_i \in N_H(Y_i) \cap N_{\cJ}(t_i, \mu)$ for all $1 \leq i \leq k$.
	By construction, $Y \cup W_1 \cup \dotsb \cup W_k$ is a reduced $(T, \mu)$-absorber.
	
	Now we count how many options we have at every step.
	Since $H$ has edge density at least $\mu$, there are at least $\mu \binom{t}{k}$ possible choices for $Y \in H$, each one of them carrying an unique orientation in $\ori{H}$.
	Each $W_i$ is chosen among a set of size $\mu \binom{t}{k-1}$ and we only need to avoid selecting $W_i$ which intersects one of the previously selected sets.
	Since $1/t \ll \mu$, for sure we have at least $(\mu/2) \binom{t}{k-1}$ possible choices at each one of the $k$ steps.
	Thus the total number of reduced $(T, \mu)$-absorbers found in this way is at least
	\[ \mu \binom{t}{k} \left( \frac{\mu}{2} \binom{t}{k-1} \right)^k \ge {\theta} \binom{t}{k} \binom{t}{k-1}^{k} \ge {\theta} |\mathscr{L}_{\ori{H}}|,  \]
	where the second to last inequality follows from ${\theta} \ll \mu$,
	and the last bound follows from Lemma~\ref{lemma:countinggadgets-global}.
\end{proof}

Now we proceed with the second step, showing that every reduced $(T, \mu)$-gadget $L$ yields many $\fS$-gadgets in $\sF_L \cap \sF_T$.

\begin{lemma} \label{lemma:countinggadgets-local-yield}
	Let $k,r,m,t \in \mathbb{N}$, and $d_2,\dots,d_k,\eps,\eps_k,c, \nu,\mu, {\theta}$ be such that  
	\begin{align*}
	1/m & \ll 1/r,\eps \ll 1/t, c, \eps_k, d_2,\dots ,d_{k-1}, \\
	c & \ll d_2,\dots ,d_{k-1}, \\
	1/t & \ll \eps_k \ll d_k \leq 1/k, \text{ and} \\
	\eps_k &  \ll \nu \ll  {\theta} \ll \mu \ll 1/k.
	\end{align*}
	Let $\mathbf{d} = (d_2, \dotsc, d_k)$ and
	let $(G, G_\cJ,\cJ, \cP, \ori{H})$ be an oriented $(k,m,t, \eps , \eps_k, r , \mathbf{d})$-regular setup.
	Let $T \subseteq V(G)$ be a $k$-set and let $L \in \mathscr{L}_{\ori{H}}$ be a reduced $(T, \mu)$-gadget in $\ori{H}$.
	Then
	\[ |\sF_L \cap \sF_T| \ge  {\theta} |\sF_L|. \]
\end{lemma}

\begin{proof}
	Let $T = \{ t_1, \dotsc, t_k \}$ and $L = Y \cup W_1 \cup \dotsb \cup W_k$.
	Similarly as the proof of Lemma~\ref{lemma:countinggadgets-local-reduced}, we show a procedure to construct gadgets in $\sF_L \cap \sF_T$.
	Then we count how many possibilities we have in each step.
	
	Let $1 \leq i \leq k$.
	We begin by choosing the paths $P_i, Q_i$ in $W_i = \{ W_{i,1}, \dotsc, W_{i,k-1} \}$.
	Let $\mathcal{Q}_{W_i}$ be the set of $(k-1)$-uniform tight paths $(v_1 v_2, \dots ,v_{2k-2})$ such that $v_j, v_{j+(k-1)} \in W_{i,j}$ for all $1 \leq j \leq k-1$ and such that its down-closure is in $\cJ$.
	Let $\mathcal{Q}_{W_i, t_i} \subseteq \mathcal{Q}_{W_i}$ be the set of those paths whose edges in the $(k-1)$-th level are in $N_G(t_i)$.
	
	Since $L$ is a reduced $(T, \mu)$-gadget, we have $W_i \in N_H(Y_i) \cap N_{\cJ}(t_i, \mu)$.
	In particular, this implies that $|N_G(t_i, \cJ_{W_i})| \ge \mu |\cJ_{W_i}|$.
	By Lemma~\ref{lemma:sidorenkito}, we get that
	\begin{align}
	|\mathcal{Q}_{W_i, t_i}| \ge \frac{1}{2} \left( \frac{\mu}{8k} \right)^{k+1} |\mathcal{Q}_{W_i}|. \label{equation:countinggadgets-local-yield-Pwx}	\end{align}
	This gives a lower bound on the possible choices in each $W_i$.
	Now our task is to show that from these choices we can build many absorbers which are suitable for $X$.
	
	We prepare the setup to apply the Extension Lemma, to show that most of the previous choices can be extended to many full absorbers in $\sF_L$.
	Let $\phi: V(F) \rightarrow V(L)$ be the obvious hypergraph homomorphism which labels the copies of $F$ in $\sF_L$.
	Let $Z \subseteq V(F)$ be equal to $\bigcup_{i=1}^k V(P_i) \cup V(Q_i)$, that is $Z$ corresponds to the union of the vertices of all the paths $P_i$ and $Q_i$, for all $1 \leq i \leq k$.
	Thus $|Z| = 2k(k-1)$.
	Let $\mathcal{Z} = \mathcal{F}[Z]$ be the induced subcomplex of $\mathcal{F}$ in $Z$.
	Note that $\mathcal{Z}$ consists of (the down-closure of) $k$ vertex-disjoint $(k-1)$-uniform tight paths on $2k-2$ vertices each, and the labelling $\phi$ indicates that the $i$-th path lies precisely in $\cQ_{W_i}$.
	Let $\cG = \cJ \cup G_\cJ$ and let $\mathcal{Z}_\cG$ (respectively $\mathcal{F}_\cG$) be the set of labelled partition-respecting copies of $\mathcal{Z}$ (respectively $\mathcal{F}$) in $\cG$.
	Let $\mathcal{F}$ be the $k$-complex generated by the down-closure of $F$.
	Let $\mathcal{Z} = \mathcal{F}[Z]$ be the induced subcomplex of $\mathcal{F}$ in $Z$.
	Note that $\mathcal{Z}$ consists of (the down-closure of) $k$ many vertex-disjoint $(k-1)$-uniform tight paths on $2k-2$ vertices each, and the labelling $\phi$ indicates that the $i$-th path lies precisely in $\cQ_{W_i}$.
	Let $\cG = \cJ \cup G_\cJ$ and let $\mathcal{Z}_\cG$ (respectively $\mathcal{F}_\cG$) be the set of labelled partition-respecting copies of $\mathcal{Z}$ (respectively $\mathcal{F}$) in $\cG$.
	Let $\beta_1$ be such that $\eps \ll \beta_1 \ll d_2, \dotsc, d_{k-1}, \eps_k$ and define $d_{\mathcal{Z}} = \prod_{i=2}^{k-1} d_i^{e_i(\mathcal{Z})}$.
	Then we have
	\begin{align}
	|\mathcal{Z}_\cG| = \prod_{i=1}^k |\cQ_{W_i}| = (1 \pm \beta_1) d_{\mathcal{Z}} m^{|V(Z)|}, \label{equation:countinggadgets-local-yield-JZ}
	\end{align}
	where the first equality follows from the structural description of $\mathcal{Z}$ and the second one from an application of the Dense Counting Lemma (Lemma~\ref{lemma:densecountinglemma}).
	
	Let $\mathcal{Z}_{\cG, T} \subseteq \mathcal{Z}_\cG$ be the labelled partition-respecting copies of $\mathcal{Z}$ whose $i$-th path lies in $\cQ_{W_i, t_i}$.
	Then inequality~\eqref{equation:countinggadgets-local-yield-Pwx} implies that
	\begin{align}
	|\mathcal{Z}_{\cG, T}| \ge \prod_{i=1}^k |\cQ_{W_i,t_i}| \ge \left( \frac{1}{2} \left( \frac{\mu}{8 k} \right)^{k+1} \right)^{k} \prod_{i=1}^k |\cQ_{W_i}| \ge 3  {\theta} |\mathcal{Z}_\cG|,\label{equation:countinggadgets-local-yield-JXZ}
	\end{align}
	where the last inequality follows from inequality~\eqref{equation:countinggadgets-local-yield-JZ} and $ {\theta} \ll \mu , 1/k$.
	
	Let $\beta_2$ be such that $\eps_k \ll \beta_2 \ll  {\theta}, d_k, 1/k$.
	Define $d_{\mathcal{F}-\mathcal{Z}} = \prod_{i=2}^{k} d_i^{e_i(\mathcal{F}) - e_i(\mathcal{Z})}$.
	Let $B \subseteq \mathcal{Z}_\cG$ be the set of labelled partition-respecting copies of $\mathcal{Z}$ which do not extend to $(1 \pm \beta_2) d_{\mathcal{F}-\mathcal{Z}} m^{|V(F)| - |V(Z)|}$ many labelled partition-respecting copies of $\mathcal{F}$ in $\cG$.
	An application of the Extension Lemma (Lemma~\ref{lem:extension}) with $\beta_2$ playing the role of $\beta$ gives that
	\begin{align}
	|B| \leq \beta_2 |\mathcal{Z}_\cG| \leq  {\theta} |\mathcal{Z}_\cG|, \label{equation:countinggadgets-local-yield-B}
	\end{align}
	where we used $\beta_2 \ll  {\theta}$ in the last inequality.
	
	From inequality~\eqref{equation:countinggadgets-extensibleornot-A}, $\eps_k \ll \beta_2 \ll d_k, 1/k$ and the definitions of $d_{\mathcal{F}-\mathcal{Z}}$ and $d_{\mathcal{Z}}$ we get
	\begin{align} 
	|\sF_L| = (1 \pm \beta_2) d_{\mathcal{F}-\mathcal{Z}} d_{\mathcal{Z}} m^{|V(F)|}. \label{equation:countinggadgets-local-yield-dFdZ}
	\end{align}
	
	We can now estimate $|\sF_L \cap \sF_T|$ by noting that any labelled partition-respecting copy of $\mathcal{F}$ in $\cG$ which extends a $Z \in \mathcal{Z}_{\cG, T}$ yields exactly one gadget in $\sF_L \cap \sF_T$.
	Also, by definition, every $Z \in \mathcal{Z}_{\cG,T} \setminus B$ extends to at least $(1 - \beta_2) d_{\mathcal{F}-\mathcal{Z}} m^{|V(F)| - |V(Z)|}$ such copies.
	Hence, we have
	\begin{align*}
	|\sF_L \cap \sF_T|
	& \ge |\mathcal{Z}_{\cG,T} \setminus B| (1 - \beta_2) d_{\mathcal{F}-\mathcal{Z}} m^{|V(F)| - |V(Z)|} \\
	& \ge (|\mathcal{Z}_{\cG,T}| - |B|) (1 - \beta_2) d_{\mathcal{F}-\mathcal{Z}} m^{|V(F)| - |V(Z)|} \\
%	& \stackrel{\eqref{equation:countinggadgets-local-yield-JXZ}}{\geq} (3  {\theta} |\mathcal{Z}_{\cG}| - |B|) (1 - \beta_2) d_{\mathcal{F}-\mathcal{Z}} m^{|V(F)| - |V(Z)|} \\
	& \stackrel{\eqref{equation:countinggadgets-local-yield-JXZ},~\eqref{equation:countinggadgets-local-yield-B}}{\geq} 2  {\theta} |\mathcal{Z}_{\cG}| (1 - \beta_2) d_{\mathcal{F}-\mathcal{Z}} m^{|V(F)| - |V(Z)|} \\
	& \stackrel{\eqref{equation:countinggadgets-local-yield-JZ}}{\geq} 2  {\theta} (1 - \beta_1 ) (1 - \beta_2)
	d_{\mathcal{Z}} m^{|V(Z)|}
	d_{\mathcal{F}-\mathcal{Z}} m^{|V(F)| - |V(Z)|} \\
	& \ge 2  {\theta} (1 - 2\beta_2) d_{\mathcal{Z}} d_{\mathcal{F}-\mathcal{Z}} m^{|V(F)|} \\&
	\stackrel{\eqref{equation:countinggadgets-local-yield-dFdZ}}{\geq} 2  {\theta} \frac{1 - 2\beta_2}{1 + \beta_2} |\sF_L|
	\ge  {\theta} |\sF_L|,
	\end{align*}
	where we used the indicated inequalities and $\beta_2 \ll  {\theta}$ in the last step.
\end{proof}

\begin{proof}[Proof of Lemma~\ref{lemma:countinggadgets-local}]
	Let $ {\theta}'$ be such that $ {\theta} \ll  {\theta}' \ll \mu$.
	By Lemma~\ref{lemma:countinggadgets-local-yield}  with $ {\theta}'$ playing the role of $ {\theta}$, we have for each reduced $(T, \mu)$-gadget $L \in \mathscr{L}_{\ori{H}}$,
	\[ |\sF_T \cap \sF_L| \ge  {\theta}' |\sF_L|. \]
	Let $\beta$ be such that $\eps_k \ll \beta \ll d_k, 1/k,  {\theta}'$.
	Lemma~\ref{lemma:countinggadgets-extensibleornot} implies that $|\sF_L \setminus \sF^\ext_L| \leq \beta |\sF_L| \leq ( {\theta}'/2) |\sF_L|$.
	It follows that
	\[ |\sF^\ext_T \cap \sF_L| \ge |\sF_T \cap \sF_L| - |\sF_L \setminus \sF^\ext_L| \ge ( {\theta}'/2) |\sF_L| .\]
	By Lemma~\ref{lemma:countinggadgets-local-reduced} with $ {\theta}'$ playing the role of $ {\theta}$, we get that $|\mathscr{L}_{\ori{H}, T, \mu}| \ge  {\theta}' |\mathscr{L}_{\ori{H}}|$.
	Combining everything and using $ {\theta} \ll  {\theta}'$, we get
	\[ |\sF^\ext_T|
	\ge \sum_{L \in \mathscr{L}_{\ori{H}, T, \mu}} |\sF^\ext_T \cap \sF_L|
	\ge \frac{ {\theta}'}{2} \sum_{L \in \mathscr{L}_{\ori{H}, T, \mu}} |\sF_L|
	\ge  {\theta} |\sF^\ext|, \]
	where we used Lemma~\ref{lemma:countinggadgets-global} to bound $|\sF^\ext|$.
\end{proof}

\subsection{Short absorbing paths}\label{sec:short-absorbing-paths}

The results of the previous subsection imply that, in the setting of the Absorption Lemma (Lemma~\ref{lem:absorption}), every $k$-set $T$ is suitable for many extensible $\fS$-gadgets.
In this subsection, we show that one can select a well-distributed subset $\cA$ of $O(n/t)$ extensible $\fS$-gadgets such that every $k$-set $T$ is suitable for  $\Omega(n/t)$ of the elements of $\cA$.
In the next subsection, we connect $O(t)$ of such subsets $\cA$ up to obtain the absorbing path that meets the requirement of Lemma~\ref{lem:absorption}.

\begin{lemma}[Sampling absorbers] \label{lem:sampling-absorbers}
	Let $k,r,m,t \in \mathbb{N}$ and $d_2,\dots,d_k$, $\eps$, $\eps_k$, $c$, $\nu$, $\mu$, $\alpha$, $ {\theta}$, $\zeta$ be such that
	\begin{align*}
	1/m & \ll 1/r,\eps \ll 1/t, \zeta, \eps_k, d_2,\dots ,d_{k-1}, \\
	\zeta & \ll c \ll d_2,\dots ,d_{k-1}, \\
	1/t & \ll \eps_k \ll d_k, \nu \leq 1/k, \text{and} \\
	 {c} & \ll \eps_k \ll \alpha \ll  {\theta} \ll \mu  \ll 1/k.
	\end{align*}
	Let $\mathbf{d} = (d_2, \dotsc, d_k)$ and
	let $(G, G_\cJ,\cJ, \cP, \ori{H})$ be an oriented $(k,m,t, \eps , \eps_k, r , \mathbf{d})$-regular setup.
	Suppose that $G$ has $n$ vertices and $n \leq (1 + \alpha)mt$.
	Suppose that $H$ has edge density at least $\mu$ 
	and that for every $v \in V(G)$ and every $Z \in V(H)$, we have $|N_{\cJ}(v, \mu) \cap N_H(Z) | \ge \mu \binom{t}{k-1}$.
	Then there exists a set $\cA$ of pairwise disjoint $\fS$-gadgets which are $(c, \nu)$-extensible and such that
	\begin{enumerate}[\upshape (a)]
		\item \label{itm:sampling-absorbers-A-small} $|\cA| \leq  \zeta m$,
		\item \label{itm:sampling-absorbers-A-absorbing}  $|\cA \cap \sF^\ext_T| \geq \zeta {\theta} m$ for every $k$-set $T$ in $V(G)$ and
		\item \label{itm:sampling-absorbers-A-sparse}  $V(\cA)$ is $(4{\zeta}/t)$-sparse in $\cP$.
	\end{enumerate}	 
\end{lemma}

\begin{proof} 
	Let $\beta > 0$ such that $\eps_k \ll \beta \ll d_k$.
	Let $F$ be the $k$-graph as in Definition~\ref{definition:absorberinslice} and let $\cF$ be the $k$-complex generated by its down-closure.
	Let $d_F = \prod^{k}_{i=2} d_i^{e_i(\cF)}$.
	By Lemma~\ref{lemma:countinggadgets-global}, the number $|\sF^\ext|$ of $\fS$-gadgets which are $(c, \nu)$-extensible satisfies
	\begin{align}
	|\sF^\ext|
	& \leq (1 + \beta)  d_F m^{k(2k+1)} \binom{t}{k} \binom{t}{k-1}^k
	\leq  d_F m^{k(2k+1)} t^{k^2}, \label{equation:absorberiterable-Fupper} \\
	|\sF^\ext|
	& \ge \frac{\mu^{k+1}}{2} (1 - \beta) d_F m^{k(2k+1)} \binom{t}{k} \binom{t}{k-1}^k, \nonumber \\
	& \ge \frac{\mu^{k+1}}{2 k^k (k-1)^{k^2}} d_F m^{k(2k+1)} t^{k^2} \ge  16 \theta^{1/2} d_F m^{k(2k+1)} t^{k^2}, \label{equation:absorberiterable-Flower}
	\end{align}
	where the second to last inequality follows as $1/t \ll \eps_k \ll \beta \ll d_k$,
	and the last inequality follows from ${\theta} \ll \mu, 1/k$.
	Similarly, Lemma~\ref{lemma:countinggadgets-global} allows us to bound, for each reduced gadget $L \in \mathscr{L}_{\ori{H}}$, the number of absorbing gadgets contained in $L$ by
	\begin{align}
	|\sF^\ext_{L}|
	& \leq 2 d_F m^{k(2k+1)}.
	\label{equation:absorberiterable-FLupper}
	\end{align}
	By Lemma~\ref{lemma:countinggadgets-local} (applied with $\theta^{1/2}$ playing the role of $\theta$) we also know that, for each $k$-set $T$ of vertices in $G$, the set $\sF^\ext_T$ of $(c, \nu)$-extensible $\fS$-gadgets associated with $T$ satisfies
	\begin{align}
	|\sF^\ext_T|
	& \ge \theta^{1/2} |\sF^\ext|
	\ge  6 \theta d_F m^{k(2k+1)} t^{k^2}, \label{equation:absorberiterable-FXlower}
	\end{align}
	where we used inequality~\eqref{equation:absorberiterable-Flower} for the last step.
	
	Next, we choose a random set $\sF' \subset \sF^\ext_T$ and show that with high probability, $\sF'$ satisfies properties~\ref{itm:sampling-absorbers-A-small}--\ref{itm:sampling-absorbers-A-sparse} of the lemma statement.
	Some of the chosen gadgets might overlap, but it is not likely that many pairs have this property.
	Hence we can simply delete the overlapping pairs while maintaining properties~\ref{itm:sampling-absorbers-A-small}--\ref{itm:sampling-absorbers-A-sparse}.
	
	To put this plan into action, let \[p = \frac{\zeta m}{2d_F m^{k(2k+1)} t^{k^2}}\]
	and define a subset $\sF' \subset \sF^\ext$ by including each $\fS$-gadget in $\sF'$ independently at random, each with probability $p$.
	By inequality~\eqref{equation:absorberiterable-Fupper}, we can bound the expectation 
	\begin{align}\label{equ:absorberiterable-exp-F}
	\expectation[|\sF'|] = p |\sF^\ext| \leq \frac{\zeta m}{2}.
	\end{align}
	
	For each $L \in \mathscr{L}_{\ori{H}}$, let $\sF'_L = \sF' \cap \sF^\ext_L$, and
	for each $k$-set $T$ of $V(G)$, let $\sF'_T = \sF' \cap \sF^\ext_T$.
	Note that for all $k$-sets $T$ of $V(G)$, we have 
	\begin{align}\label{equ:absorberiterable-exp-X}
	\expectation[|\sF'_T|] = p |\sF^\ext_T| \ge 3 \theta \zeta m.
	\end{align}
	by inequality~\eqref{equation:absorberiterable-FXlower}.
	For each $Z \in \cP$, let $Z'=V(\sF') \cap Z$.
	Note that any fixed cluster of $\cP$ is contained in at most $t^{k^2-1}$ reduced gadgets.
	So by inequality~\eqref{equation:absorberiterable-FLupper}, there are at most $2 d_F m^{k(2k+1)}t^{k^2-1}$ $\fS$-gadgets with vertices in that cluster.
	Also note that each $\fS$-gadget has at most $3$ vertices in the same cluster.
	Hence, we can bound
	\begin{align}\label{equ:absorberiterable-exp-L}
	\expectation[|Z'|] \leq 6 d_F m^{k(2k+1)}t^{k^2-1} p = \frac{3 \zeta m}{t}
	\end{align}
	for every cluster $Z \in \cP$.
	Since $\sF'$, $\sF'_L$ and $\sF'_T$ are binomially distributed, Chernoff's bound (Lemma~\ref{lem:che}) together with the inequalities~\eqref{equ:absorberiterable-exp-F},~\eqref{equ:absorberiterable-exp-X} and~\eqref{equ:absorberiterable-exp-L} gives that, for every $k$-set $T$ of $V(G)$ and every $Z \in \cP$,
	\begin{align*}
	\Pr\left(|\sF'| > \zeta m\right) &< \exp(-\Omega(m)) < \frac{1}{4},\\
	\Pr\left(|\sF'_T|  < 2 \theta \zeta m \right) &< \exp(-\Omega(m)) \leq \frac{1}{4n^k} \text{ and }
	\\ \Pr\left(|Z'|  >  \frac{4 \zeta m}{t} \right) &< \exp(-\Omega(m)) \leq \frac{1}{4t}.
	\end{align*}
	
	Next, we show that it is not likely that too many of the selected gadgets overlap.
	More precisely, we say that two distinct $\fS$-gadgets (or reduced gadgets) \emph{overlap}, if they share at least one vertex.
	Note that there are at most $k^4 t^{2k^2-1}$ pairs of overlapping reduced gadgets.
	Consequently, there at most $(k(2k+1))^2m^{2k(2k+1)-1} k^4 t^{2k^2-1}$ pairs of overlapping $\fS$-gadgets.
	Denote by $P$ the random variable that counts the number of those pairs which are both in $\sF'$.
	It follows that 
	\begin{align*}
	\expectation[P] &\leq (k(2k+1))^2m^{2k(2k+1)-1} k^4 t^{2k^2-1} p^2 
	\\&= (k(2k+1))^2m^{2k(2k+1)-1} k^4 t^{2k^2-1} \left(\frac{\zeta m}{2d_F m^{k(2k+1)} t^{k^2}}\right)^2
	\\&= \frac{(k(2k+1))^2 \cdot k^4 \zeta^2 {m} }{4d_F^2  t }
	\leq {\frac{\zeta \theta m}{4}},
	\end{align*}
	where the last inequality is due to $\zeta \ll d_2, \dotsc, d_{k-1}, d_k, 1/k, \theta$.
	By Markov's inequality, we can bound
	\begin{align*}
	\Pr\left(P > \zeta \theta m \right) \leq \frac{1}{4}.
	\end{align*}
	
	Hence, with positive probability, $\sF'$ satisfies
	\begin{enumerate}[(a$'$)]
		\item $|\sF'| \leq \zeta m$,
		\item $|\sF'_T| \geq 2 \zeta \theta m$ for every $k$-set $T$ of $V(G)$,
		\item $|V(\sF')|$ is $(4 \zeta /t)$-sparse in $\cP$ and
		\item at most $\zeta \theta m$ pairs of $\fS$-gadgets in $\sF'$ overlap.
	\end{enumerate}
	Fix such a set $\sF'$ and delete one gadget from every overlapping pair in it, to obtain $\mathcal{A} \subseteq \sF^\ext$.
	We deleted at most $\zeta \theta m$ gadgets in total, so it follows that $\mathcal{A}$ satisfies~\ref{itm:sampling-absorbers-A-small}--\ref{itm:sampling-absorbers-A-sparse}, as desired.
\end{proof}

\subsection{Proof of the Absorption Lemma}\label{sec:proof-absorption-lemma}

Now we are ready to give a proof of the Absorption Lemma.
%Now we are ready to give a proof of the Absorption Lemma, which we restate here together with the definition of an absorbing path.
%Indeed, consider a $k$-graph $G$ on $n$ vertices, $S \subseteq V(G)$ and let $P$ be a tight path in $G$.
%Recall that we defined a path $P$ to be $S$-absorbing in $G$ if there exists a path $P'$ in $G$ with the same $k-1$ starting vertices of $P$, the same $k-1$ ending vertices of $P$ and $V(P') = V(P) \cup S$.
%We also said that $P$ is $\eta$-absorbing in $G$, if it is $S$-absorbing in $G$ for every $S$ of size at most $\eta n$ whose size is divisible by $k$ and with $S \cap V(P) \neq \es$.

%\lemabsorption*

\begin{proof}[Proof of Lemma~\ref{lem:absorption}]
	We begin by using the degree conditions of $G$ and $H$ to ensure that we can apply our earlier developed tools.
	Since $G$ has minimum relative $1$-degree at least $\delta + \mu$ and $\fS$ is a representative setup, Lemma~\ref{lemma:neighboursincomplex} implies that for every $v \in V(G)$, we have
	\begin{align}
	\left|N_{\cJ}\left(v, \frac{\mu}{3}\right)\right| \ge \left( \delta + \frac{\mu}{4} \right) \binom{t}{k-1}.
	\label{equation:absorbinglemma-degreeisrepresented}
	\end{align}
	
	Let $\zeta>0$ with $1/r,\eps \ll \zeta \ll c$, and let $\theta > 0$ with $\eta \ll \theta \ll \mu, 1/k$.
	The idea is to construct the path $P$ iteratively.
	More precisely, let $N = \lceil \eta t/(\theta \zeta) \rceil$.
	We show the following claim.
	\begin{claim}\label{cla:absorbinglemma}
		For each $0 \leq j \leq N$, there is a tight path $P_j \subseteq G$ with the following properties:
		\begin{enumerate}[\upshape (i$'$)]
			\item \label{itm:absorbing-path-proof-absorbing} $P_j$ is $(j \theta \zeta / t)$-absorbing in $G$,
			\item \label{itm:absorbing-path-proof-Pj-extending} $P_j$ is $(c, \nu)$-extensible and consistent with $\ori{H}$, 
			\item \label{itm:absorbing-path-proof-sparse-in-cP} $V(P_j)$ is $(100k^2 j \zeta/t)$-sparse in $\cP$ and $V(P_j) \cap C_j=\es$, where $C_j$ denotes the connection set of $P_j$.
		\end{enumerate}
	\end{claim}
	\begin{proofclaim}
		We can take $P_0$ to be the empty path.
		Now let us assume that $P_j$ satisfies the above conditions for $0 \leq j < N$.
		We will use Lemma~\ref{lem:sampling-absorbers} to choose new gadgets and Lemma~\ref{lem:connecting-many-paths} to connect everything up.
		
		As a first step, we restrict the graph $G$ to ensure that the new set of gadgets has an extension set that is not already occupied with the vertices of $V(P_j)$.
		For each $Z \in \cP$, select a subset $Z' \subset Z \sm V(P_j) $ of size $m' = (1 - \lambda) m$.
		This is possible by~\ref{itm:absorbing-path-proof-sparse-in-cP} and as $100k^2 j \zeta/t \leq (2\eta t/(\zeta \theta))(100k^2 \zeta/t) \leq \lambda$ which follows from $j \leq N$ and $\zeta \ll\eta \ll \lambda, \theta, 1/k$.
		Note that $m' \geq n/(2t)$ by the assumption of $n \leq (1 + \alpha)mt$.
		Let $\cP'=\{Z'\}_{Z \in \cP}$ and $V(\cP') = \bigcup_{Z' \in \cP'} Z'$.
		Let $G'_{\cJ'} = G_{\cJ}[V(\cP')]$ and $\cJ' = \cJ[V(\cP')]$ be the induced subgraphs and subcomplex, respectively.
		By the Slice Restriction Lemma (Lemma~\ref{lem:regular-slice-restriction}), $\fS' = (G, G'_{\cJ'},\cJ', \mathcal{P}', H)$ is a $(k, m', r, \sqrt{\eps}, \sqrt{\eps_k}, r, \mathbf{d})$-regular setup, as desired.
		
		Now we select a set $\cA'$ of fresh $\fS'$-gadgets in $G'_{\cJ'}$.
%		Given a $k$-set $T \subseteq V(G')$, we let $\sF_T$ be the set of $\fS'$-gadgets which are $(c, \nu)$-extensible and suitable for absorbing $T$.
		By Lemma~\ref{lemma:resilientneighboursincomplex} (with $\mu/6$ in place of $\mu$) we have, for every $v \in V(G)$,  that $|N_{\cJ'}(v, \mu/6)| \ge |N_{\cJ}(v, \mu/3)| \ge \left( \delta + \frac{\mu}{4} \right) \binom{t}{k-1}$, where the last inequality is due to inequality~\eqref{equation:absorbinglemma-degreeisrepresented}.
		Together with the minimum $1$-degree of $H$, it follows that for every $v \in V(G)$ and every $Z \in V(H)$, we have
		\[ |N_{\cJ'}(v, \mu/6) \cap N_H(Z) | \ge \frac{\mu}{6} \binom{t}{k-1}. \]
		Thus we can invoke Lemma~\ref{lem:sampling-absorbers} with
		\begin{center}
			\begin{tabular}{ c  *{6}{|c} }
				object/parameter &$G'$& $\cJ'$ & $m'$  & $\mu/6$    & $2\zeta$  &   $4c$\\
				\hline
				playing the role of &$G$& $\cJ$ & $m$   & $\mu$ & $\zeta$  &   $c$
			\end{tabular}
		\end{center} 
		to obtain a set $\cA'$ of pairwise disjoint $\fS'$-gadgets which are $(4c, \nu)$-extensible and such that
		\begin{enumerate}[(a)]
			\item \label{itm:absorbing-path-proof-size-A'} $|\cA'| \leq 2\zeta m'${,}
			\item  \label{itm:absorbing-path-proof-absorbing-A'} $|\cA' \cap \sF_T| \geq 2 \zeta \theta m'$ for every $k$-set $T$ in $V(G)$ and
			\item \label{itm:absorbing-path-proof-sparse-A'} $V(\cA')$ is $(8 \zeta/t)$-sparse in $\cP'$.
		\end{enumerate}
		
		Next, we connect up the individual paths of the absorbing gadgets and also $P_j$.
		To start, let us prepare the input for Lemma~\ref{lem:connecting-many-paths}.
		Recall that, by Definition~\ref{definition:absorberinslice}~\ref{item:gadget-extensible}, each  $\fS'$-gadget in $\cA'$ consists of $k+1$ tight paths which are $(4c, \nu)$-extensible in $\fS'$.
		Let $\cA$ be the union of these paths for every gadget in $\cA'$ and with $P_j$.
		Also set  $C_{j+1} = V(G) \sm V(\cA)$.
		We claim that $\cA$ is a set of pairwise disjoint tight paths in $G$ such that 
		\begin{enumerate}[(a$'$)]  
			\item \label{itm:absorbing-path-proof-size-A} $|\cA| \leq (k+1)2\zeta m' +1$,
			\item \label{itm:absorbing-path-proof-sparse-A} $V(\cA)$ is $({100}k^2 j \zeta/t + {8} \zeta /t)$-sparse in $\cP$ and $V(\cA) \cap C_{j+1}=\es$ and
			\item \label{itm:absorbing-path-proof-exensible-A} every path in $\cA\sm \{P_j\}$ is $(2c, \nu, C_{j+1})$-extensible in $\fS$ and consistent with $\ori{H}$.
			Moreover, $P_j$ is $(c, \nu, C_{j+1})$-extensible in $\fS$ and consistent with $\ori{H}$.
		\end{enumerate}
		Indeed, part~\ref{itm:absorbing-path-proof-size-A} follows from~\ref{itm:absorbing-path-proof-size-A'} and the addition of $P_j$.
		Part~\ref{itm:absorbing-path-proof-sparse-A} follows from~\ref{itm:absorbing-path-proof-sparse-in-cP},~\ref{itm:absorbing-path-proof-sparse-A'} and the definition of $C_{j+1}$ (note that this implies, using the hierarchy and the bounds on $j \leq N$, that $V(\cA)$ is $\lambda$-sparse in $\cP'$ as well).
		Finally, part~\ref{itm:absorbing-path-proof-exensible-A} follows from~\ref{itm:absorbing-path-proof-Pj-extending} and~\ref{itm:absorbing-path-proof-sparse-A'} as ${8 \zeta m} /t \leq 2cm$.
		In particular, $P_j$ is $(c,\nu)$-extensible by~\ref{itm:absorbing-path-proof-Pj-extending}, while all other paths go from $(4c,\nu)$-extensible in $\fS'$ to $(2c,\nu)$-extensible in $\fS$.
		Moreover, the consistency with $\ori{H}$ is given by the consistency of $P_j$ and the definition of the $\fS'$-gadgets.
%		 (so we go from $(4c,\nu)$-extensible in $\fS'$ to $(2c,\nu)$-extensible in $\fS$).
		
		Having established~\ref{itm:absorbing-path-proof-size-A}--\ref{itm:absorbing-path-proof-exensible-A}, we apply Lemma~\ref{lem:connecting-many-paths} with
		\begin{center}
			\begin{tabular}{ c  *{7}{|c} }
				object/parameter &$G'$& $\cJ'$ & $m'$ & $c$ & $C_{j+1}$ & $(k+1)4\zeta$ & {$\lambda$} \\
				\hline
				playing the role of &$G$& $\cJ$ & $m$ & $c$ & $C$ & $\zeta$ & $\lambda$
			\end{tabular}
		\end{center} 
		to obtain a tight path $P_{j+1}$ such that
		\begin{enumerate}[\upshape(A)]
			\item \label{itm:absorbing-path-proof-Pcontains} $P_{j+1}$ contains every path of $\cA$ as subpath,
			\item \label{itm:absorbing-path-proof-Pends} $P_{j+1}$ starts and ends with two paths different from $P_j$, 
			\item \label{itm:absorbing-path-proof-P-A-subset-partition} $V(P_{j+1}) \sm V(\cA) \subset V(\cP')$ and
			\item \label{itm:absorbing-path-proof-P-A-intersectio-Z} $V(P_{j+1}) \sm V(\cA)$ intersects in at most $   10k^2\mathcal{A}_Z +t^{t+3k}$ vertices with each cluster $Z \in \cP$, where $\cA_Z$ denotes the number of paths of $\cA$ that intersect with $Z$.
		\end{enumerate} 
		
		We claim that $P_{j+1}$ satisfies the conditions~\ref{itm:absorbing-path-proof-absorbing}--\ref{itm:absorbing-path-proof-sparse-in-cP} with $j+1$ in place of $j$.
		We begin by showing~\ref{itm:absorbing-path-proof-sparse-in-cP}.
		Note that for every cluster $Z \in \cP$, the number of paths of $\cA$ that intersect with $Z$ is bounded by $8\zeta m/t+1$.
		This follows by~\ref{itm:absorbing-path-proof-sparse-A'} and the definition of $\cA$.
		Hence~\ref{itm:absorbing-path-proof-P-A-intersectio-Z} implies that $V(P_{j+1}) \sm V(\cA)$ intersects in at most $100k^2 \zeta m/t$ vertices with each cluster $Z \in \cP$.
		Together with~\ref{itm:absorbing-path-proof-sparse-in-cP}, it follows that $\cA$ is $(100k^2 (j+1)\zeta m/t)$-sparse in $\cP$, which gives~\ref{itm:absorbing-path-proof-sparse-in-cP}.
		
		Next, we deduce~\ref{itm:absorbing-path-proof-Pj-extending}.
		As noted before, $V(P_{j+1}) \sm V(\cA)$ intersects in at most $100k^2 \zeta m/t \leq cm/4$ vertices with each cluster $Z \in \cP$, where the bound follows from $\zeta \ll c$.
		Moreover, by~\ref{itm:absorbing-path-proof-exensible-A} we have $V(\cA) \cap C_{j+1}=\es$.
		Hence we obtain~\ref{itm:absorbing-path-proof-Pj-extending} after deleting the vertices of $\cP_{j+1}$ from $C_{j+1}$. (So we go from $(2c,\nu)$-extensible as in~\ref{itm:absorbing-path-proof-exensible-A} to $(c,\nu)$-extensible after the deletion).
		Note that at this point it was crucial that $P_{j+1}$ starts and ends with two paths different from $P_j$, which was asserted in~\ref{itm:absorbing-path-proof-Pends}.\footnote{In fact, the only reason why we construct the absorbing paths iteratively is that we need to refresh the connection set after adding about $cn$ absorbing gadgets.}
		
		Finally, we claim that $P_{j+1}$ is $((j+1)\zeta \theta/t)$-absorbing in $G$, which gives~\ref{itm:absorbing-path-proof-absorbing}.
		To this end, let $S \subseteq V(G)$ be any set of size divisible by $k$ and at most $(j+1)\zeta \theta n / t$.
		Partition $S \setminus V(P_{j+1})$ into two sets $S_1$ and $S_2$, such that both $|S_1|, |S_2|$ are divisible by $k$ and $|S_1|$ is maximal such that $|S_1| \leq j \zeta \theta n / t$.
		Since $P_j$ is $(j \zeta \theta/ t)$-absorbing in $G$ and is a subpath of $P_{j+1}$, there exists a path $P'_j$ with the same endpoints as $P_j$ such that $S_1 \subseteq V(P'_j)$, so it remains to cover $S_2$.
		By the choice of $S_1$, we have that $|S_2| \leq \zeta \theta n / t + k \leq 2 \zeta^3 n / t \leq 2 (1 + \alpha) \zeta^3 m \leq 5 \zeta^3 m / 2$.
		Therefore we can partition $S_2$ into $\ell \leq 5 \zeta \theta m / (2k) \leq 2 \zeta \theta m'$ sets of size $k$ each, let $T_1, \dotsc, T_{\ell}$ be those sets.
		By~\ref{itm:absorbing-path-proof-absorbing-A'}, for each $1 \leq i \leq \ell$ we have $|\sF_{T_i} \cap \cA'| \ge \ell$.
		Thus we can associate each $T_i$ with a different gadget $F_i \in \cA'$.
		Each gadget $F_i$ yields a collection of $k+1$ many paths $P_{i,1}, \dotsc, P_{i,k+1}$, all of which are in $\cA$, and we can substitute those paths for a collection of different paths $P'_{i,1}, \dotsc, P'_{i,k+1}$ with the same endpoints such that $T_i \cup V(F_i) = V(P'_{i,1}) \cup \dotsb \cup V(P'_{i,k+1})$ (as illustrated in Figure~\ref{figure:absorber-after}).
		Since $P_j$ and each of the paths $P_{i,j}$ are subpaths of $P_{j+1}$,  the aforementioned substitutions result in a tight path $P'_{j+1}$ with the same endpoints as $P_{j+1}$ and whose vertex set is exactly $V(P_{j+1}) \cup T$, as desired.
	\end{proofclaim}
	To finish, note that $P_N$ has the desired properties.
	Indeed, by the choice of $N = \lceil \eta t / (\theta \zeta) \rceil$ we have $N \zeta \theta / t \ge \eta$, so $P_N$ is $\eta$-absorbing in $G$.
	Moreover, since $N(100k^2\zeta /t) \leq 200 k^2 \eta / \theta \leq \lambda$ (by the assumption of $\zeta \ll \eta \ll \lambda, \theta, 1/k$), $V(P_N)$ is $\lambda$-sparse in $\cP$.
\end{proof}

\section{Conclusion} \label{sec:discussion}
Rödl, Ruciński and Szemerédi determined the $d$-degree threshold of $k$-uniform Hamilton cycles whenever $d=k-1$.
Here we extended this result to $d= k-2$ and proved a general upper bound of $\hc_d(k) \leq 1 - 1 /(2(k-d))$.
Our proofs take place in the environment of frameworks and vicinities, which offer a novel perspective on these problems and could lead to further progress on the study of thresholds for tight Hamilton cycles.
To conclude this paper, we discuss some conjectures and problems in this strand of research.

In light of our work, the most basic question is how the minimum $d$-degree threshold of $k$-uniform tight Hamilton cycles behaves for $k-d\geq3$.
The work of Han and Zhao~\cite{HZ16} together with Theorem~\ref{thm:main-simple-general} shows that there are constants $c, C > 0$ such that $1 - c(k-d)^{-1/2} \leq \hc_d(k) \leq 1- C (k-d)^{-1}$ for all $k > d \ge 1$.
In our view, the left side of these two inequalities is more likely to reflect the truth.
In order to understand the asymptotic behaviour of the tight Hamilton cycle threshold it would be interesting to gauge whether $\hc_d(k) \leq 1-C' (k-d)^{-1/2}$ for a $C' > 0$.

A more challenging task consists in determining exact bounds.
In context of this, let us review the construction of Han and Zhao~\cite{HZ16}.

\begin{construction}\label{const:tight-general}
	For $1 \leq d \leq k-2$, let $\ell=k-d$.
	Choose $0 \leq j \leq k$  such that $(j-1)/k < \lceil \ell/2 \rceil /(\ell+1) < (j+1)/k$.
	Let $H$ be a $k$-graph on $n$ vertices and a subset of vertices  $X$ with $|X|= \lceil \ell/2 \rceil n /(\ell+1)$, such that $H$ contains precisely the edges $S$ for which $|S \cap X| \neq j$.
\end{construction}
Consider a $k$-graph $H$ as in the Construction~\ref{const:tight-general}.
By design, there is no tight walk between edges $S$ of type $|S \cap X| > j$ and edges $S'$ of type $|S' \cap X| < j$.
On the other hand, a simple averaging argument shows that a tight Hamilton cycle must contain an edge of each type.
Hence $H$ does not admit a tight Hamilton cycle.
Calculating the minimum degree of $H$, we obtain for instance that $\hc_d(k) \geq 1/2, 5/9$, $5/8$, $408/625$ for $k-d=1$, $2$, $3$, $4$.
(It can be shown that the value of $j$ and the size of $X$ in the construction maximises these bounds.)
We believe that the lower bounds obtained by Han and Zhao are best possible.

\begin{conjecture}\label{con:Han-Zhao-are-best-possible}
	The minimum $d$-degree threshold for $k$-uniform tight Hamilton cycles coincides with the lower bounds given by Construction~\ref{const:tight-general}.
\end{conjecture}

In fact, in our view something stronger could be true.
Recall that $\hf_d(k)$ denotes the Hamilton framework threshold and $\hv_d(k)$ the Hamilton vicinity threshold.
Our main theoretical results (Theorem~\ref{thm:framework} and~\ref{thm:hamilton-vicinities}) state that $\hc_d (k) \leq \hf_d(k) \leq \hv_d(k)$ for all $1 \leq d \leq k-1$.
We believe these inequalities to be tight, or in other words it suffices to study vicinities to determine the thresholds of tight Hamilton cycles.
\begin{conjecture}\label{con:tight-cycles--general-threshold}
	We have $\hc_d (k) = \hv_d(k)$ for all $1 \leq d \leq k-1$.
\end{conjecture}

Finally, recall that the idea behind vicinities is to analyse the Hamilton cycle threshold through the perspective of link graphs.
We do not have a reason to believe that some link graphs behave fundamentally different than others.
Hence, the $k$-uniform Hamilton $d$-vicinity threshold could potentially by determined by solving an Erdős--Gallai-type problem for $(k-d)$-graphs.
To formalise this, we introduce a last piece of notation.

\newcommand{\eg}{\mathrm{eg}}
\begin{definition}\label{def:erdos-gallai}
	For $\ell \in\NATS$, let $\eg(\ell)$ be the smallest number $\delta >0$ such that, for every $\mu >0$, there are $\gamma >0$ and $n_0 \in \NATS$ with the following property.
	
	Suppose that $G$ is an $\ell$-graph on $n \geq n_0$ vertices with edge density at least $\delta + \mu$.
	Then there is a subgraph $C\subset G$ which is
	\begin{enumerate}[label=\textnormal{(\roman*)}]
		\item \label{itm:problem-erdos-gallai-connected}  tightly connected, \hfill {\upshape(connectivity)}
		\item \label{itm:problem-erdos-gallai-matching}  has a fractional matching of density $1/(\ell+1) + \gamma$ and\hfill {\upshape(space)}
		\item \label{itm:problem-erdos-gallai-dense}  has edge density at least $1/2+\gamma$. \hfill {\upshape(density)}
	\end{enumerate}
\end{definition}

%If a $k$-graph $H$ satisfies $\overline{\delta}_{d}(H) \ge \eg(k-d)$, then its $(k-d)$-uniform link are structured enough for $H$ to admit a Hamilton frameworks.
The following result was implicitly shown in the proof of Theorem~\ref{thm:main-simple-general}.
\begin{theorem}\label{thm:erdos-gallai}
	For $1 \leq d \leq k-1$, we have $\hc_{d}(k) \leq \hv_d(k) \leq \eg(k-d)$.
\end{theorem}
Observe that when $k-d$ is odd, then the $(k-d)$-uniform link graphs of Construction~\ref{const:tight-general} reveal the same lower bounds on $\eg(k-d)$, that is $\eg(1)  \geq 1/2$ and $\eg(3)\geq 5/8$.
We do not believe that this is a coincidence (see Figure~\ref{fig:conjectures}). 
\begin{conjecture}\label{con:erdos-gallai}
	When $k-d$ is odd, then $\eg(k-d)$ is attained by the $(k-d)$-uniform link graphs of Construction~\ref{const:tight-general}.
\end{conjecture}
A first step towards this conjecture would be to show that $\eg(3) = 5/8$.
We remark that a similar conjecture can be formulated for even $k-d$ by generalising the statement of Lemma~\ref{lem:cooley-mycroft}.

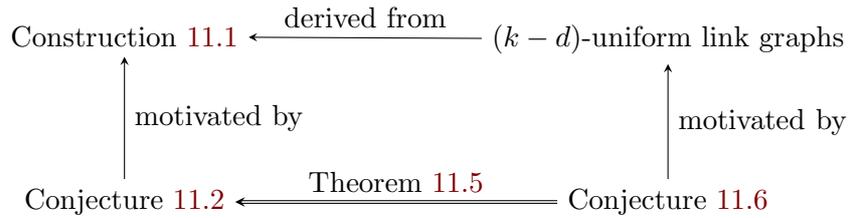
\begin{figure}
\centering

\begin{tikzpicture}
\matrix (m) [matrix of math nodes,row sep=4em,column sep=8em,minimum width=3em]
{
	\text{Construction~\ref{const:tight-general}} & \text{$(k-d$)-uniform link graphs} \\
	\text{Conjecture~\ref{con:Han-Zhao-are-best-possible}} & \text{Conjecture~\ref{con:erdos-gallai}} \\};

\path[-stealth] (m-1-2) edge node [above] {derived from} (m-1-1);
\path[-stealth] (m-2-1) edge node [right] {motivated by} (m-1-1);
\path[-stealth] (m-2-2) edge node [right] {motivated by} (m-1-2);
\path[-stealth] (m-2-2) edge [double] node [above] {Theorem~\ref{thm:erdos-gallai}} (m-2-1); 
\end{tikzpicture} 

\caption{Interplay between conjectures and constructions.}
\label{fig:conjectures}
\end{figure}

We remark that the combination of the connectivity and space property of Definition~\ref{def:erdos-gallai} are a special case of the Erdős--Gallai problem for tight cycles as studied by Allen, B\"ottcher, Cooley and Mycroft~\cite{ABCM17}.
Conversely, the combination of the connectivity and density property has, as far as we are aware, not been studied before and appears to be a natural question of independent interest.

Our belief here is that the size of an edge-maximal tight component undergoes a phase-transition once a certain threshold is passed.
Let us highlight the case $\ell = 3$, which is particularly relevant since it would give insight to understand $\hc_d(k)$ when $k-d=3$.
As demonstrated by the link graphs in Construction~\ref{const:tight-general}, there are $3$-graphs with edge density $5/8-o(1)$ in which the largest tight component has edge density $1/2 -o(1)$.
However, it might be that the edge density of the largest tight component in a $3$-graph of edge density $5/8+o(1)$ is much larger than $1/2+o(1)$ and instead of order at least $((1 + \sqrt 5)/4)^3 \approx 0.52951$.
(The latter bound is obtained from a $3$-graph $H$ on $n$ vertices with a subset of $n(1 + \sqrt 5)/4$ vertices $X$, and all edges that do not share precisely two elements in $X$.)

More generally, we could ask the following.
For given $k$ and $n$, let $f(\delta)$ be the edge density of a largest tight component in a $k$-graph on $n$ vertices.
Is $f$ continuous?
If not, what are its discontinuities?
We do not have answers to these questions, but they seem to present an inviting avenue for future research.

\section*{Acknowledgements}
We would like to thank Yi Zhao for bringing this problem to our attention and an anonymous referee for helpful comments that improved the presentation.

\bibliographystyle{amsplain}
\bibliography{bibliography}

\appendix

\section{Regular setups from regular slices} \label{appendix:slicesandsetup}

In the following, we sketch how the Regular Setup Lemma (Lemma~\ref{lem:regular-setup}) can be obtained from the Regular Slice Lemma of Allen, Böttcher, Cooley and Mycroft~\cite{ABCM17}.
Essentially, we can construct a regular setup via some modifications done in the output of the Regular Slice Lemma, which we state now.

Recall that given a regular complex $\cJ$ for a $k$-graph $G$ and $d_k > 0$, the weighted reduced $k$-graph $R(G)$ and the $d_k$-reduced graph $R_{d_k}(G)$ were introduced in Definition~\ref{definition:weightedreduced}.
We introduce some further notations to capture how well the degrees of vertices are transferred to a regular slice.
Recall that, for a set $S \subset V(G)$ of $d$ vertices, the {relative degree of $S$ in a $k$-graph $G$} of $n$ vertices was defined as $\reldeg(S;G) = \deg(S)/\binom{n-d}{k-d}$.
Similarly, if $G$ is instead a weighted $k$-graph with weight function $\reld$, then we define 
\[\reldeg(S;G) = \frac{\sum_{e \in G\colon S \subseteq e} \reld(e)}{\binom{n-d}{k-d}}.\]
In other words, $\reldeg(S;G)$ is the (weighted) proportion of edges of $G$ containing $S$.
Finally, for any set $X \subseteq V(G)$ we define the \emph{mean relative degree of $X$ in $G$}, denoted by $\reldeg(X; G)$, to be the mean average of $\reldeg(S;G)$ over all $d$-sets $S \subset X$.
Recall that $v$ is a vertex of $G$, we write $\deg_G(v)$ for the degree of $v$ in $G$ and $\reldeg_G(v) = \deg_G(v)/\binom{n-1}{k-1}$ for the relative degree of $v$ in $G$, and $\reldeg_G(v;\cJ)$ was introduced before Definition~\ref{definition:representativerooted}.

With this notation at hand, we can present the statement of the Regular Slice Lemma.

\begin{lemma}[Regular Slice Lemma {\cite[Lemma 10]{ABCM17}}]\label{lem:regular-slice}
	Let $k \geq 3$ be a fixed integer.
	For all positive integers $t_0$, positive~$\eps_k$ and all functions $r: \NATS \rightarrow \NATS$ and $\eps: \NATS \rightarrow (0,1]$,
	there are integers~$t_1$ and~$n_0$ such that the following holds for all $n \ge n_0$ which are divisible 
	by~$t_1!$. 
	Let $G$ be a $k$-graph whose vertex set $V$ has size $n$. 
	Then there exists a $(k-1)$-complex	$\cJ$ with $V(\cJ) = V$ which is a $(t_0,t_1,\eps(t_1),\eps_k,r(t_1))$-regular slice for
	$G$, such that $\cJ$ has the following additional properties.
	\begin{enumerate}[\upshape(R1)]
		\item\label{item:regularslice-degrees} For each $1 \le j \leq k-1$,
		each set $S$ of $j$ clusters of $\cJ$, 
		we have
		\[\big|\reldeg(S;R(G))-\reldeg(\cJ_S;G)\big|<\eps_k.\]
		\item\label{item:regularslice-rootedcount}
		For each vertex $v \in V(G)$, we have
		\[ \big|\reldeg_G(v;\cJ)-\reldeg_G(v) \big|<\eps_k \]
	\end{enumerate}
\end{lemma}

We remark that the original Regular Slice Lemma has even more properties that $\cJ$ can be guaranteed to satisfy and we state only those that we need.
In particular, from their original statement~\cite[Lemma 10]{ABCM17} we only use the case where $s = 1$ (only one $k$-graph is regularised), $q = 1$ and $\mathcal{Q}$ is a trivial partition in one cluster.
Then~\ref{item:regularslice-degrees} follows from item (b) in~\cite[Lemma 10]{ABCM17} by setting $X = \cP$,
and~\ref{item:regularslice-rootedcount} follows from item (c) in~\cite[Lemma 10]{ABCM17} by setting $\ell = 1$ and $H$ to be a $k$-graph on $k$ vertices consisting of a single edge.

Note that~\ref{item:regularslice-degrees} compares the average of the relative degrees $\reldeg(\cJ_S;G)$ in the original graph with the weighted degrees of the weighted reduced graph $R(G)$, whereas the output of Lemma~\ref{lem:regular-setup} is in terms of the relative degrees in the $d_k$-reduced graph $R_{d_k}(G)$ instead.
However, this is possible due the the next lemma, which states that the relative degrees of $R(G)$ and $R_{d_k}(G)$ are not too far off if $d_k$ is small.

\begin{lemma}[{\cite[Lemma~12]{ABCM17}}] \label{lem:reduced-graph-degree}
	Consider $k,n,t,r \in \NATS$ and $\eps,\eps_k,d_k >0$.
	Let $G$ be a $k$-graph and let $\cJ$ be a
	$(\cdot,\cdot,\eps,\eps_k,r)$-regular slice for~$G$ with $t$
	clusters. Then for any set cluster $S$ of $\cJ$ we have
	\begin{align*} \label{eq:reduceddegree} 
	\reldeg(S;R_{d_k}(G)) \ge \reldeg(S;R(G)) - d_k - \zeta(S),
	\end{align*}
	where $\zeta(S)$ is defined to be the proportion of $k$-sets of clusters $Z$
	satisfying $S \in Z$ which are not $(\eps_k, r)$-regular with respect to $G$.
\end{lemma}

\begin{proof}[Proof sketch of Lemma~\ref{lem:regular-setup}]
	Apply Lemma~\ref{lem:regular-slice} with $k$, $t_0$, $\eps_k^2$, $r$, $\eps^2$ as input to get $t_1, n_0$ as output.
	Set $m_0 = \max\{ \lceil n_0/t_1 \rceil, 200 t_1 t_1! k \eps_k^{-3} \}$.
	As in the statement of Lemma~\ref{lem:regular-setup}, we are given a $k$-graph $G$ on at least $2 t_1 m_0 \ge 2 n_0$ vertices with relative degree $\overline{\delta}_d(G) \geq \delta +\mu$. 
	We have to show the existence of a suitable regular setup.
	
	Adding at most $t_1!$ isolated extra vertices to $G$, we obtain a $k$-graph $G'$ whose number of vertices is at least $n_0$ and divisible by $t_1!$.
	By the choice of $n_0$, there exists a $(k-1)$-complex $\cJ'$ with $V(\cJ') = V(G')$ which is a $(t_0, t_1, \eps^2(t_1), \eps_k^2, r(t_1))$-regular slice for $G'$.
	Let $\cP'$ be the ground partition of $\cJ'$,
	let $t$ be the number of clusters of $\cP'$,
	let $R' = R(G')$ be the weighted reduced graph of $G'$ on $\cJ'$ (whose vertex set is $\cP'$),
	and let $\mathbf{d} = (d_2, \dotsc, d_{k-1})$ be the density vector of $\cJ'$.
	By the properties detailed in Lemma~\ref{lem:regular-slice}, it follows that $\cP', R'$ satisfy the corresponding~\ref{item:regularslice-degrees}--\ref{item:regularslice-rootedcount} with $\eps_k^2$ in place of $\eps_k$.
	For brevity, we write $\eps = \eps(t_1)$ and $r = r(t_1)$ from now on.
	
	To get a regular slice for $G$ we need to remove the auxiliary vertices $V(G') \setminus V(G)$ and some more extra vertices from each cluster, to ensure that what remains in each cluster has the same size.
	Thus we choose $W \subseteq V(G')$ to be such that $W \supseteq V(G') \setminus V(G)$, $|W \cap Z_1| = |W \cap Z_2|$ and $|W \cap Z_1| \leq t_1!$ for every choice of clusters $Z_1, Z_2 \in \cP'$
	(it is certainly possible to find such a set, since $|V(G') \setminus V(G)| \leq t_1!$).
	Let $\cP$ be the partition of $V(G') \setminus W \subseteq V(G)$ whose clusters are $\{ Z \setminus W : Z \in \cP' \}$, and let $\cJ'$ be the restriction of $\cJ$ to $V(\cP)$.
	Thus $\cP$ has $t$ clusters.
	Let $R(G)$ be the weighted reduced graph of $G$ with respect to $\cJ$, and let $R_{d_k}$ be the corresponding $d_k$-reduced graph.
	Let $m$ be the common size of the clusters of $\cP$.
	
	From the Slice Restriction Lemma (Lemma~\ref{lem:regular-slice-restriction}) we deduce that $\cJ$ is a $(\cdot, \cdot, \eps, \eps_k, r)$-regular slice for $G$, with $V(\cJ) = V(G) \setminus W$ and density vector $\mathbf{d}$.
	Using this information it is straightforward to check that the properties of $\cJ', \cP', G'$ are essentially transferred to $\cJ, \cP, G$.
	For instance, if $G'$ is $(d, \eps_k^2, r)$-regular with respect to a $k$-set $X' = \{ Z'_1, \dotsc, Z'_k \}$, then $G$ is $(d, \eps_k, r)$-regular with respect to the corresponding $k$-set $X = \{Z_1, \dotsc, Z_k\}$ where $Z_i = Z'_i \setminus W$ for all $i \in [k]$.
	In particular, after doing the natural identification of the clusters in $\cP'$ with those of $\cP$ the weights and properties of $R'(G')$ are transferred to $R(G)$ in the obvious way, and we shall use this identification without further comment from now on.
	
	Note that if $X \in R_{d_k}(G)$, then $X$ is a $k$-tuple of clusters of $\cP$ such that $G$ is $(\eps_k, r)$-regular with respect to $X$ whose relative density which is \emph{at least} $d_k$.
	To satisfy the requirement of a regular setup, we require to find $G_{\cJ} \subseteq G$ such that the relative density of $G_{\cJ}$ with respect to $X$ is exactly $d_k$, simultaneously for every $X \in R_{d_k}(G)$.
	The existence of such a $G_{\cJ}$ is easily obtained from $G$ via an application of a Slicing Lemma (see for instance the work of Cooley, Fountoulakis, Kühn and Osthus~\cite[Lemma 8]{CFKO09}) which constructs $G_{\cJ}$ from $G$ by deleting edges at random to achieve the required density while still keeping the regularity properties.
	
	Now it is relatively straightforward to check that $\mathfrak{S} = (G, G_{\cJ}, \cJ, \cP, R_{d_k})$ is a $(k, m, t, \eps, \eps_k, \mathbf{d})$-regular setup.
	To see that $\mathfrak{S}$ is representative as well, one should use~\ref{item:regularslice-rootedcount} which, by construction, is satisfied by $G', \cJ'$ with $\eps^2_k$ in place of $\eps_k$.
	This is equivalent to saying that $\cJ'$ is $\eps_k^2$-rooted-degree-representative for~$G'$.
	In passing from $G', \cJ'$ to $G, \cJ$ we remove at most $t_1!$ vertices from each of the $t$ clusters, and thus at most $t t_1! \leq t_1 t_1!$ vertices in total.
	By the choice of $m$, this number is less than an $\eps^3_k/(200k)$ proportion of the size of each the clusters of $\cP$, and thus the rooted copies of edges in $\cJ$ are essentially preserved from $\cJ$.
	From this it is straightforward to check that $\cJ$ is $\eps_k$-rooted-degree-representative.
	(We remark that essentially the same argument is carried out with detailed calculations in Lemma~\ref{lemma:resilientneighboursincomplex} where a linear number of vertices is removed.)
	
	Now we check the extra properties which are required for a the regular setup $\mathfrak{S}$.
	The common size of each cluster in $\cP'$ was at least $n/t_1 \ge 2 m_0$ and at most $t_1!$ vertices were removed from each cluster to get to $\cP$, so we deduce $m \ge 2 m_0 - t_1! \ge m_0$, where the last bound follows from the choice of $m_0$.
	Since $t \ge t_0$ and $1/m \ll 1/t_0, \alpha$, we get $|W| \leq t_1! t \leq \alpha mt$.
	Since $mt = |V(\cJ)|$ and $V(G) \subseteq V(\cJ) \cup W$ we deduce $n \leq (1 + \alpha) mt$.
	
	We finish the proof by defining the $k$-graphs $R$ and $I$ and deriving their degree conditions.
	Let $I$ be the collection of $k$-sets of clusters of $\cP$ to which $G$ is not $(\eps_k, r)$-regular and set $R = R_{d_k} \cup I$ (note that the union is disjoint by the definition of $R_{d_k}$).	
	Since $\cJ'$ is a regular slice and the previous arguments, we get that the collection of irregular $k$-tuples of $\cJ$ satisfies $|I| \leq \eps_k \binom{t}{k}$.
	Note that the value $\zeta(S)$ appearing in Lemma~\ref{lem:reduced-graph-degree} is exactly $\reldeg(S;I)$.
	Thus we get, for every $d$-set of clusters $S$,
	\begin{align*}
	\reldeg(S;R)
	& = \reldeg(S;R_{d_k}) + \reldeg(S;I)
	= \reldeg(S;R'_{d_k}(G')) + \reldeg(S;I) \\
	& \ge \reldeg(S;R'(G')) - d_k
	\ge \delta + \mu/2,
	\end{align*}
	where the first inequality is due to Lemma~\ref{lem:reduced-graph-degree}, and the second inequality is~\ref{item:regularslice-degrees} applied to $R'(G')$ together with $d_k \ll \mu$.
	Thus $R$ has minimum relative $d$-degree at least $\delta + \mu / 2$, as required.	
\end{proof}

\end{document}